\documentclass[10pt,reqno]{amsart}
\usepackage{amsmath,amsfonts,amssymb,amscd,amsthm,amsbsy,bbm, epsf,calc}
\usepackage[pdftex]{graphicx}
\usepackage{color, float}	
\usepackage{datetime}

\numberwithin{equation}{section}

\theoremstyle{definition}
\newtheorem{definition}{Theorem}[section]
\newtheorem{thm}{Theorem}[section]

\newtheorem{lem}[thm]{Lemma}
\newtheorem{cor}[thm]{Corollary}
\newtheorem{prop}[thm]{Proposition}
\newtheorem{ex}[thm]{Example}

\newtheorem{DEF}[definition]{Definition}
\newtheorem{hyp}{Assumption}

\newtheorem{rem}[thm]{Remark}

\def\R{{\mathbb R}}

\def\ep{\varepsilon}
\def\e{\varepsilon}

\def\d{\textnormal{d}}

\def\union{\bigcup}

\def\dive{\text{div}}
\def\pal{\;\;\mathbb{P}\text{-a.s.}}
\def\bfb{\text{\bf{B}}}
\def\Ext{\text{Ext}}
\def\pot{\text{pot}}
\def\partialom{\partial^\omega}
\def\db{\text{D}}

\DeclareMathOperator*{\osc}{osc}

\begin{document}
\title{Quasistatic Droplet percolation}
\author{Nestor Guillen and Inwon Kim}
\address{Department of Mathematics, UCLA} 
\email{nestor@math.ucla.edu, ikim@math.ucla.edu}

    \begin{abstract}
         We consider the Hele-Shaw problem in a randomly perforated domain with zero Neumann boundary conditions. A homogenization limit is obtained as the characteristic scale of the domain goes to zero. Specifically, we prove that the solutions as well as their free boundaries converge uniformly to those corresponding to a homogeneous and anisotropic Hele-Shaw problem set in $\mathbb{R}^d$. The main challenge when deriving a limit lies in controlling the oscillations of the free boundary, this is overcome first by extending De Giorgi-Nash-Moser type estimates to perforated domains and second by proving the almost surely non-degenerate growth of the solution near its free boundary. 
    \end{abstract}

	\thanks{N. Guillen is supported by NSF DMS-1201413. I. Kim is supported by NSF DMS-0970072.}
	
	\keywords{droplet spreading, perforated domain, quasi-static phase transition, Hele-Shaw, internal DLA, percolation, stochastic homogenization, free boundary problem, elliptic regularity, obstacle problem, disordered media}

	\subjclass[2000]{
	35B65, 
	35R35, 
	76D27, 
	76M50, 
	82B43, 
	}	
\maketitle
\markboth{N. Guillen and I. Kim}{Quasistatic Droplet Percolation}

\section{Introduction}

    For a stationary ergodic set $\mathcal{O}\subset \mathbb{R}^d$ and $\epsilon$ set $\mathcal{O}_\epsilon := \epsilon \mathcal{O}_\epsilon$ we consider a nonnegative random function
    \begin{equation*}
         u^\epsilon: \mathcal{O}_\epsilon \times \mathbb{R}_+ \to \mathbb{R},	
    \end{equation*}
    which has compact support for $t=0$ and solves the free boundary problem
    \begin{equation}\label{eqn: Hele-Shaw O_ep}
	     \left \{ \begin{array}{rll}
	          -\Delta u^\epsilon  & = 1 & \text{ in } \{u^\epsilon >0 \},\\
	          \partial_n u^\epsilon & = 0 & \text{ on } \partial \mathcal{O}_\epsilon,\\
	          u^\epsilon_t & = |\nabla u^\epsilon|^2 & \text{ on } \partialom \{ u^\epsilon>0\},
	     \end{array}\right.
    \end{equation}
    where $\partial_n$ denotes the outer normal derivative on $\partial \mathcal{O}_\epsilon$ and  $\partialom$ denotes the boundary operator which acts on $\{u^\epsilon>0\}$ relative to $\mathcal{O}_\epsilon(\omega)$, where $\omega$ is the sample variable. The first two equations say $u^\epsilon$ solves an elliptic problem for fixed $t$ and the third says  $\{u^\epsilon>0\}$ is expanding with normal velocity given by $|\nabla u^\epsilon|$.
 
    This is the well known one-phase Hele-Shaw problem, posed in a random domain $\mathcal{O}_\epsilon$ with Neumann boundary conditions. In this note, we shall prove a homogenization result which essentially says (see Theorem \ref{thm: main} for the exact statement) that almost surely $u^\epsilon$ and the free boundaries $\partialom \{ u^\epsilon > 0 \}$ converge uniformly to $u$ and $\partial \{ u>0\}$, where $u:\mathbb{R}^d \times \mathbb{R}_+ \to \mathbb{R}_+$ solves the (deterministic) homogeneous problem
    \begin{equation}\label{eqn: Hele-Shaw homogeneous}
	     \left \{ \begin{array}{rll}
	          -\dive (A \nabla u)  & = \mu & \text{ in } \{u >0 \},\\
	          u_t & = Q(\nabla u) & \text{ on } \partial \{ u>0\}.
	     \end{array}\right.
    \end{equation} 
    Here $A$, known as the effective diffusivity, is defined in terms of $\mathcal{O}$ (see Section 7, \eqref{eqn: definity of A}). The constant $\mu$ represents the fraction of volume occupied by $\mathcal{O}$ and $Q$ is the quadratic form $Q(p)= (Ap \cdot p)/\mu$. Note that $u^\epsilon$ is defined in a different domain of $\mathbb{R}^d$ for each different $\epsilon$, but the putative limit $u$ is defined everywhere (and in particular in $\mathcal{O}_\epsilon$ for each $\epsilon$). Therefore, we define uniform convergence of $u^\epsilon$ to $u$ as
    \begin{equation*}
         \lim \limits_{\epsilon \to 0_+} \|u(\cdot,t)-u^\epsilon(\cdot,t)\|_{L^\infty(\mathcal{O}_\epsilon)}=0	\pal \; \forall\;t>0,
    \end{equation*}
    and similarly for the uniform convergence of the free boundaries (see Section \ref{sec: fbp in stationary ergodic domains} for precise definitions).

    \subsection{Motivation} The Hele-Shaw flow (for $d=2$) describes movement of an incompressible, viscous fluid confined between two narrow plates, as first described in the seminal paper of Hele-Shaw \cite{HS1898}. In this context, $u$ represents the pressure inside the cell, and our model arises as the surface tension zero limit when the plates trap a layer of some granular medium $\mathcal{O}_\epsilon$. Also for $d=2$ one may see \eqref{eqn: Hele-Shaw O_ep} as the evolution of a growing droplet which is spreading in a two dimensional labyrinth with walls lying vertically on top of $\partial \mathcal{O}_\epsilon$. The height of the droplet  at time $t$ is given by $u^\epsilon(.,t)$, thus the base of the droplet would be the ``positive phase''  $\{ u^\epsilon>0\}$. Under this interpretation, the condition $\Delta u^\epsilon = -1$ is a linearization of the constant mean curvature condition for the droplet, and the zero Neumann condition of $u^\epsilon$ just says the droplet meets the walls of the labyrinth at a right angle.
    Then, our main theorem (Theorem \ref{thm: main}) states that in the limit $\epsilon \to 0$ the droplets behave as if they were spreading on a free surface except their shape and motion law are changed, which typically become anisotropic (that is, the effective diffusivity $A$ might not be diagonal, see Section \ref{sec: linear homogenization revisited}).

    Let us note however, that the main motivation for this work is not the justification of the physical validity of the homogenized limit for the droplet interpretation. Instead, the goal is to understand in a ``simple'' case,  the regularity properties of diffusive free boundary problems when the underlying geometry is strongly irregular and random. By ``strongly irregular''  we mean that the randomness does not take the form of a diffusion operator in $\mathbb{R}^d$ with random uniformly elliptic coefficients, but instead that we are dealing with a random domain $\mathcal{O}_\epsilon$. A by product of our approach are uniform regularity estimates and pointwise bounds for solutions of linear elliptic equations in randomly perforated domains, which are new and of independent interest. 

    \begin{rem} A model which would be more physically realistic (for the situation of the droplet) is the one where we require $-\Delta u^\epsilon = \lambda(t)$ at each time, with $\lambda(t)>0$ being picked so that the volume of the droplet is preserved over time. Our analysis could be extended to this case if we had a well-posedness theory for the problem for  fixed $\epsilon$ ( with some continuity estimate for $\lambda(t)$), this is however not available at the moment. For more on droplet models, see \cite{Gla2005,GK09,AlbDeS2011} and references therein.
	\end{rem}

    \subsection{Background and contributions of this work} The evolution problem \eqref{eqn: Hele-Shaw O_ep}, as well as the broad study of elliptic equations in $\mathcal{O}_\epsilon$ is of interest both in the context of homogenization of partial differential equations as well as in percolation theory and analysis of (even discrete) stochastic growth models. First we review some of the relevant PDE and probabilistic ideas and how they relate to each other, followed by a discussion of our main results, which will be stated in detail in Section \ref{subsec: Main results}.

    \medskip

    \emph{Linear homogenization in random media.} Recall that the random domain $\mathcal{O}_\epsilon$ can be thought of as resulting from a degenerate limit of linear and uniformly elliptic operators in $\mathbb{R}^d$. Heuristically, for fixed $\epsilon$ the evolution problem \eqref{eqn: Hele-Shaw O_ep} is obtained  as the limit as $\delta\to 0$ of a family of Hele-Shaw problems posed in $\mathbb{R}^d$ with a random diffusivity matrix
    \begin{equation*}
	     \left \{ \begin{array}{rll}
	          -\dive (A_{\delta} \nabla u_\delta)  & = 1 & \text{ in } \{u_\delta >0 \},\\
	          (u_\delta)_t & = (A_\delta\nabla u_\delta,\nabla u_\delta) & \text{ on } \partial \{ u_\delta>0\},\\
	     \end{array}\right.
    \end{equation*}
    where $A_\delta(x)$ is the matrix given by
    \begin{equation*}
         A_\delta(x) = \left \{ \begin{array}{rl}
              \text{I} & \text{ if } x\in\mathcal{O}_\epsilon,\\
              \delta \text{I} & \text{ if } x\not\in\mathcal{O}_\epsilon. 	
         \end{array}	\right.
    \end{equation*}
    so in some sense, we are dealing with a degenerate elliptic problem in $\mathbb{R}^d$ where the diffusivity matrix field has the form $A(x,\omega) = \tilde A(x,\omega)\mathbb{I}_{\mathcal{O}(\omega)}(x)$ for some uniformly elliptic matrix $\tilde A(x,\omega)$.

    Let us focus for a moment on the theory of homogenization for (uniformly) elliptic operators in divergence form. Namely, (fixed) boundary value problems of the form
    \begin{equation*}
         \left \{ \begin{array}{rl}
         	  -\dive(A(\tfrac{x}{\epsilon},\omega)\nabla v) & = f\;\;\text{ in }  D \subset \mathbb{R}^d,\\
              v & = g \;\;\text{ on } \partial D,
         \end{array} \right.	
    \end{equation*}
    for a uniformly elliptic matrix process $A(x,\omega)$, fixed domain $D \subset \mathbb{R}^d$ and functions $f,g$. Then linear homogenization consists in finding whether  the solutions to the above boundary value problems converge as $\epsilon \to 0$ to the solution of a boundary value problem for a deterministic homogeneous operator.

    The first homogenization result was obtained by Tartar \cite[Chapter 1, Section 3]{BenLioPap1978} in the case where $A(x,\omega)\equiv A(x)$, for a periodic and uniformly elliptic $A(x)$. Later, both Kozlov \cite{Koz1978} and Papanicoloau-Varadhan \cite{PapVar1979} proved independently that homogenization holds for a stationary ergodic matrix process $A(x,\omega)$ (still assumed uniformly elliptic). Much later, the case where $A$ is degenerate elliptic and periodic, that is $A(x,\omega)=\mathbb{I}_{\mathcal{O}}$ and $\mathcal{O}$ periodic (with $\mathbb{R}^d\setminus \mathcal{O}$ connected and with Lipschitz boundary) was studied by  Cioranescu-Paulin \cite{CiPa1979}  and Cioranescu-Donato \cite{CiDo1988}. The nonlinear theory in divergence form was done by Dal Maso - Modica \cite{DaMo1986}, see \cite{JikKozOle1994,DaMo1986} for more complete references on all these topics. Most relevant to the analysis of \eqref{eqn: Hele-Shaw O_ep} is the homogenization of the Neumann problem for the Laplacian in a stationary ergodic domain, the homogenization for this boundary value problem has been considered previously by Jikov \cite{Zhi1993}. 

    The works of Kozlov and Papanicoloau-Varadhan already provide the almost sure convergence to the homogenized limit in $L^2$, moreover, they show that the effective equation will be uniformly elliptic with the same ellipticity bounds as the heterogenous problems. In this case the matrix $A(x,\omega)$ is uniformly elliptic, so one can use De Giorgi-Nash-Moser theory to get uniform interior H\"older continuity of the solutions, guaranteeing  the convergence is not just in $L^2$ but in fact locally uniform. In contrast, estimates of De Giorgi-Nash-Moser type are not readily available in a random domain $\mathcal{O}_\epsilon$, and the question of (locally) uniform convergence has not been resolved yet in this case.

    \medskip 

    \emph{Homogenization of free boundary problems.} These problems take the previous theory as a starting point, as it provides at least the effective equation for the limiting free boundary problem \emph{away} from the moving interface. The challenge is to identify the homogenized free boundary condition, which is not only a nonlinear condition but is also imposed on a lower dimensional (and often singular) set. A uniformly elliptic version of \eqref{eqn: Hele-Shaw O_ep} was first achieved by Kim-Mellet in \cite{KM09}. They used the monotonicity of this one-phase problem as well as the uniqueness theory from \cite{K03} to show that the integral of a viscosity solution with respect to time solves a certain obstacle problem, something that was expected but had only been proved for classical solutions \cite{ElJa}. Using this correspondence, the homogenization of the original problem was reduced to the homogenization of an obstacle problem. Then, the homogenization of the obstacle problem can be handled since pointwise estimates (based on Harnack's inequality and barriers) control  the free boundaries uniformly as $\epsilon \to 0$.

    There are also some results for uniformly elliptic one phase problems when one does not have any divergence structure (in which case the approach above fails), however, they are restricted to the periodic setting \cite{K09}. These results do not rely on writing an auxiliary obstacle for the time integral, instead the approach is inspired by a different method introduced in \cite{CaSoWa2005} for stochastic homogenization of fully nonlinear elliptic equations. It is worth noting that this method also involves an auxiliary obstacle problem, although this is completely unrelated to the obstacle problem above.

    Related results in homogenization of free boundaries include singular limits for flame propagation problems, in particular recent work of Caffarelli, Lee and Mellet for the periodic case\cite{CafKiAMel04} and the one-dimensional random case\cite{CafKiAMel07}. A time independent problem for droplets was considered by Mellet and Nolen\cite{MelNol2012}, and homogenization of the one-phase Stefan problem was studied by the Kim and Mellet\cite{KimMel2010}.

    \begin{rem}
         Another example of stochastic homogenization of a free boundary problem is the obstacle problem with a random highly oscillatory obstacle \cite{CaMe09}. Although we also deal with an obstacle problem in the present work (as an auxiliary problem, see Section \ref{sec: weak fb results}), here we have a random domain with Neumann boundary conditions, a very different situation from the one treated in \cite{CaMe09}.
    \end{rem}

	\emph{Percolation and growth models.} From the probabilistic side, the Hele-Shaw problem in $\mathbb{R}^d$ is the continuum counterpart (and often shown to be the continuous scaling limit) of several growth models taking place in the lattice $\mathbb{Z}^d$ (such as the abelian sandpile) or internal diffusion limited aggregation (DLA) \cite{LeYu2010}. 
Keeping with this probabilistic interpretation of Hele-Shaw, our problem in a random domain is then a continuous analogue of internal DLA in a percolation cluster in $\mathbb{Z}^d$.

    For a proper introduction to (discrete) percolation see Grimmet \cite{Gri1999}, see Barlow \cite{Bar2004} for a summary of many of the results and analytical issues related to random walks in percolation clusters in $\mathbb{Z}^d$. The survey by Biskup \cite{Bis2011} discusses many recent results and standing open questions.
 
    In the context of (continuous) random walks in infinite percolation clusters $\mathcal{C} \subset \mathbb{Z}^d$, the questions corresponding to homogenization are concerned with the long time behavior of the walk and the continuum limit when we let spatial scale of the lattice go to zero. For either case, obtaining pointwise bounds on the transition density of the random walk (i.e. the heat kernel for the percolation cluster) is essential, in other words,  Nash-Aronson type bounds for the heat kernel are required. Equivalently, one is interested in proving De Giorgi-Nash-Moser type estimates in the infinite percolation cluster. 
  
    However, it is well known that such pointwise estimates fail for supercritical percolation. In fact, in general, both continuum and discrete settings, a (uniform in $\omega$) Sobolev inequality is not expected to hold. Heuristically, the infinite percolation cluster will always have some (maybe far away) region that is  poorly connected, which could trap the random walk for a very long time, and this is reflected by the heat kernel. In geometric terms, the poor connectivity obstructs the validity of the isoperimetric inequality everywhere (and thus the Sobolev embedding and Poincar\'e's inequality). 

    This difficulty has been overcome, albeit only in the discrete case. The literature on such results is much too vast, we refer to the introduction in Barlow \cite{Bar2004} as well as Biskup's survey \cite{Bis2011} for a proper list of references. Also in \cite{Bar2004}, a proof of non uniform bounds (in $\omega$) is obtained by combining analytic and probabilistic ideas. Such bounds say that after waiting for some time  (for a random time $T_x$) at a given point $x \in \mathbb{Z}^d$ the transition densities become Gaussian. The estimates by Barlow \cite{Bar2004} go along the lines of the Fabes-Stroock proof of the Nash's inequalities \cite{FabStr1986} (for a deterministic equation). One may say that from  all the approaches to De Giorgi-Nash-Moser theory, the one that has been most adaptable to the setting of discrete percolation is that of Nash.

    \begin{rem}
         The validity of pointwise bounds on the Green's function (or heat kernel) for a domain $\mathcal{O}$ suggests that it enjoys some regularity properties. More precisely, if the Green's function of a given $\mathcal{O}$ is pointwise comparable with $|x-y|^{2-d}$, then not only does the Harnack inequality follows almost immediately, one can also prove a Sobolev inequality for functions in $H^1(\mathcal{O})$. Compare this with the result that Gaussian bounds for the heat kernel (for a uniformly elliptic operator in divergence form) hold if and only if we have a certain doubling property of the volume and the Poincar\'e inequality holds for all balls (see \cite{Gri1991,Sal1992} for Riemannian manifolds  and \cite{Del1999} for graphs, see also Appendix \ref{sec: Green's function estimates} ).
    \end{rem}

    This lack of uniform regularity estimates in $\mathcal{C}$ mirrors the lack of regularity estimates and uniform convergence for linear homogenization in stationary ergodic domains $\mathcal{O}_\epsilon$. This becomes a serious issue in the analysis of \eqref{eqn: Hele-Shaw O_ep}, since as hinted at earlier, one typically uses Harnack inequalities and regularity estimates to control the oscillations of the free boundary in the obstacle problem. If one is to carry out an approach similar to that of \cite{KM09}, then one needs to develop an elliptic regularity theory in $\mathcal{O}_\epsilon$. We manage to do obtain uniform estimates in the continuum case under the extra assumption that $\mathcal{O}_\epsilon$ is a randomly perforated domain (Assumption \ref{hyp: regularity}), at the expense of not dealing with the continuum analogues of supercritical percolation models (such as the Boolean model).

    \medskip

    \emph{Contributions of this work.} In this paper we consider randomly perforated domains $\mathcal{O}_\epsilon$, obtained by removing from $\mathbb{R}^d$ a countable number of random isolated domains with Lipschitz boundaries. These go beyond the periodic domains studied in the linear homogenization literature, but do not cover many important percolation models (see the examples in Section \ref{sec: fbp in stationary ergodic domains}). On the other hand, we only need stationary ergodicity and do not make any assumptions about independence of the random variables.

    The main result for such domains $\mathcal{O}_\epsilon$ concerns the free boundary problem \eqref{eqn: Hele-Shaw O_ep}, which we show homogenizes to \eqref{eqn: Hele-Shaw homogeneous}, the convergence being almost surely uniform both for $u^\epsilon$ and the free boundaries $\partialom\{u^\epsilon>0\}$ (see Theorem \ref{thm: main}). This result also covers the one in \cite{KM09}, giving a simpler proof of the uniform convergence of the free boundaries. Moreover, our proof of homogenization has as a by product the homogenization of the usual obstacle problems in $\mathcal{O}_\epsilon$.

    As a necessary step for the proof of this theorem we also prove uniform H\"older estimates and Harnack inequality for harmonic functions in subsets of $\mathcal{O}_\epsilon$ that have zero Neumann data on $\partial \mathcal{O}_\epsilon$ (see Theorem \ref{thm: elliptic estimates}), such estimates, as well as the main homogenization theorem, are new even for the periodic case.

    Finally, we give a new (and shorter) proof of the homogenization result due to Jikov (see \cite{Zhi1993} also \cite[Chapter 8]{JikKozOle1994}) that is closer to the proofs in the uniformly elliptic case, at the expense of assuming that $\mathcal{O}_\epsilon$ is a random perforated domain (see Assumption \ref{hyp: regularity}). In particular, we use the regularity of the domain to characterize the space of correctors in $\mathcal{O}_\epsilon$ (see Lemma \ref{lem: characterization of L^2_pot}). Moreover, we use the strong convergence granted by the uniform regularity estimates to prove that the effective diffusivity is strictly positive. 
    \begin{rem}
	    It can be argued that the assumption of a perforated domain is rather strong, however, it is hoped that the machinery developed here combined with adequate isoperimetric inequalities for subsets of $\mathcal{O}_\epsilon$ can be pushed to derive regularity bounds (Harnack, Nash-Aronson, etc) in continuum and discrete percolation using an approach closer to De Giorgi's (Nash's approach has been so far the most successful one in the discrete percolation literature). This will be explored in future work. On the other hand, as a payback for our stronger regularity assumptions, our elliptic estimates are uniform (in $\omega$) and we apply them to get pointwise bounds for the Green's function in $\mathcal{O}_\epsilon$ (see Appendix \ref{sec: Green's function estimates}).
    	
    \end{rem}

    \subsection{Outline of the paper} The precise statement of the main homogenization theorem, as well as an in depth discussion of the strategy of the proof shall be deferred to later in Section \ref{sec: fbp in stationary ergodic domains} after we review the notation and  basic definitions relevant to free boundary problems in random domains. 

    The assumptions on $\mathcal{O}_\epsilon$ and the main theorems are stated in Section \ref{subsec: Main results}. In Section \ref{subsec: examples} we discuss examples of random domains where they apply. In Section \ref{subsec: strategy} we discuss the overall plan of the proof. 

    Section \ref{sec: elliptic regularity} deals with the regularity of solutions to elliptic equations in $\mathcal{O}_\epsilon$ with zero Neumann data on $\partial \mathcal{O}_\epsilon$, including H\"older estimates and Harnack inequalities with bounds that are uniform in $\epsilon$.

    In Section \ref{sec: weak fb results} an important barrier is constructed, it is a non-negative function $\mathcal{O}_\epsilon$ with quadratic growth at infinity (it plays the role that paraboloids play in $\mathbb{R}^d$) which has zero Neumann data on $\partial \mathcal{O}_\epsilon$. This construction is non-trivial and relies on the results of previous section. This barrier is used to show several important properties of the free boundaries $\partial\{u^\epsilon>0\}$. In particular, we show they cannot develop too many thin tentacles as $\epsilon\to 0$ (i.e. in average they do not oscillate with large amplitude).

    In Sections \ref{sec: convergence obstacle problem} and \ref{sec: homogenization} we combine the control on the free boundary oscillations with tools from stochastic homogenization and viscosity solutions to prove that \eqref{eqn: Hele-Shaw O_ep} homogenizes.

    The new proof of homogenization of linear elliptic problems in random perforated domains is done in Section \ref{sec: linear homogenization revisited}. Finally, in the appendices we prove the existence of weak solutions to the auxiliary obstacle problem derived from \eqref{eqn: Hele-Shaw O_ep}, and obtain uniform bounds for the Green's function in $\mathcal{O}_\epsilon$.\\

    \textbf{Acknowledgements.} The authors would like to express their gratitude to Antoine Mellet for interesting discussions which prompted this research and to Marek Biskup for many helpful discussions on percolation.

\section{Free boundary problems in stationary ergodic domains}\label{sec: fbp in stationary ergodic domains}

    \subsection{Random perforated structures}\label{sec: random perforated structures} We review some basic concepts and notation from the theory of  stochastic homogenization for linear elliptic operators in divergence form \cite{JikKozOle1994}.

    \begin{DEF}\label{def: ergodic dynamical system}
         An \emph{Ergodic Dynamical System} is a pair where first we have a probability space $(\Omega,\mathcal{F},\mathbb{P})$ and second we have an ergodic action on $\Omega$ by $\mathbb{R}^d$. The latter means we are given maps $T_x : \Omega\to \Omega$, $x\in\mathbb{R}^d$ preserving the measure $\mathbb{P}$, satisfying the group property $T_{x+y} = T_x \circ T_y$, in a manner so that
         \begin{equation*}
              (x,\omega) \to T_x\omega	
         \end{equation*}
         is a measurable map from $\mathbb{R}^d\times \Omega$ (equipped with the product $\sigma$-algebra) to $\Omega$, and such that the only events $F \in \mathcal{F}$ which are preserved by all the maps $T_x$ are those with $\mathbb{P}(F)=0$ or $\mathbb{P}(F)=1$.
    \end{DEF}

    For the next definitions we take an ergodic dynamical system as given, we give some examples below.

    \begin{DEF}\label{def: stationary ergodic process}
    	 Given a probability space $(\Omega,\mathcal{F},\mathbb{P})$ a \emph{real valued process} is a measurable function $f: \mathbb{R}^d \times \Omega \to \mathbb{R}$,  we will say $f$ is \emph{stationary ergodic} if the underlying probability space is actually an ergodic dynamical system with an action $T$ and $f(x,\omega)$ is such that $f(x+y,\omega) = f(x,T_y\omega)$ for all $x,y$. In either case we may write $f(x)$ instead of $f(x,\omega)$ and think of $f(x)$ as a random function of $x$.
    \end{DEF}

    Naturally, if $f$ is now taking values in a finite dimensional space we will say it is stationary ergodic if all of its components in a given basis are stationary ergodic (in particular we may talk of stationary ergodic vector fields, matrix fields, and so on).
   
    \begin{rem} Note that a function $f$ is stationary ergodic if and only if there is some measurable function $\tilde f: \Omega \to \mathbb{R}$ (or $\to \mathbb{R}^d$ if $f$ is vector valued) such that $f(x,\omega) = \tilde f(T_x\omega)$.    	
    \end{rem}

    \begin{DEF}\label{def: stationary ergodic domain}
    	 A stationary ergodic domain is a map $\omega \to \mathcal{O}(\omega)\subset \mathbb{R}^d$, $\omega \in \Omega$ such that
         \begin{equation*}
              (x,\omega) \to \mathbb{I}_{\mathcal{O}(\omega)}(x)	
         \end{equation*}
         is a stationary ergodic function. Equivalently, $\mathcal{O}(\omega)$ is stationary ergodic if for some $O \in \mathcal{F}$ we have
         \begin{equation*}
              \mathcal{O}(\omega) = \{ x \mid T_x\omega \in O\}	
         \end{equation*}
    \end{DEF}

    The fundamental fact about ergodicity that we will repeatedly use is the ergodic theorem, see for instance \cite{AkKr1981} for a proof (in far greater generality that we use here).
    \begin{thm}[Ergodic Theorem]\label{thm: ergodic theorem}
         Given an ergodic dynamical system$\pal$ we have for any $\tilde f\in L^p(\Omega)$, $p \in [1,\infty)$, that if  $f^\epsilon(x,\omega):=\tilde f(T_{x/\epsilon}\omega)$ then $f^\epsilon(x)\rightharpoonup \mathbb{E}[\tilde f]$  in $L^p_{\text{loc}}(\mathbb{R}^d)$ as $\epsilon \to 0^+$.
    \end{thm}
 
    A particular and important case is when $f$ is the characteristic function of a set $\text{O}$.

    \begin{DEF}\label{def: mu}
         Given a stationary ergodic domain $\mathcal{O}$  we have for any $p \in [1,\infty)$
         \begin{equation*}
              \mathbb{I}_{\mathcal{O}_\epsilon}(x) \rightharpoonup \mu:=\mathbb{P}(\text{O}) \;\text{ in } L^p_{\text{loc}}(\mathbb{R}^d) \;\pal
         \end{equation*}
    \end{DEF}

    We now introduce the function spaces we will be working with.
 
    \begin{DEF}\label{def: function spaces}
    	 For a given domain $D$ we will denote by $H(D)$ $(H_0(D))$ the space of functions $f\in L^2(D\times \Omega)$ such that $f(\cdot,\omega) \in H^1(D) \pal$ ($f(\cdot,\omega) \in H^1_0(D) \pal$). Further, if a stationary ergodic domain is given then for any $\epsilon>0$ we introduce the set (which $D$ and $\mathcal{O}$ are being used will be clear from context)
         \begin{equation*}
              \bfb_\epsilon := \{ (x,\omega) \in D\times \Omega \mid x\in\mathcal{O}_\epsilon(\omega):=\epsilon \mathcal{O}(\omega)\}	
         \end{equation*}
         The set $\bfb_\epsilon$ inherits the $\sigma$-algebra and the product measure from $D\times \Omega$, in particular we can talk about the spaces $L^p(\bfb_\epsilon)$. Moreover, we define $H(\bfb_\epsilon)$ as the subset of $L^2(\bfb_\epsilon)$ formed by functions $f$ such that$\pal$ we have $f(\cdot,\omega) \in H^1(D \cap \mathcal{O}_\epsilon)$, $H_0(\bfb_\epsilon)$ is defined as the closure of the set of $f \in H(\bfb_\epsilon)$ which$\pal$ vanish in a neighborhood of $\partial D$ (but note they may be non-zero on $\partial \mathcal{O}_\epsilon$).
    \end{DEF} 

    \begin{DEF} Given $\omega \in \Omega$ and $A \subset \mathcal{O}_\epsilon$, we define the boundary of $A$ with respect to $\mathcal{O}(\omega)$ as
	          \begin{equation*}
	                    \partialom A := \partial A \setminus \partial \mathcal{O}.	
	          \end{equation*}
    \end{DEF}    

    As they are defined in different sets, we must define the  convergence of functions $f^\epsilon \in L^p(\bfb_\epsilon)$.

    \begin{DEF}\label{def: notions of convergence}
         A sequence $f^\epsilon \in L^p(\bfb_\epsilon)$ is said to converge to $f \in L^p(D\times \Omega)$, $p\in[1,+\infty]$ if
         \begin{equation*}
              \lim \limits_{\epsilon\to 0^+}\|f^\epsilon-f\|_{L^p(\bfb_\epsilon)}=0.
         \end{equation*}
         Moreover, if $p<\infty$ we say $f^\epsilon$ converges weakly to $f$ (also denoted $f^\epsilon \rightharpoonup f$) if in the usual sense we have
         \begin{equation*}
              \mathbb{I}_{\bfb_\epsilon}f^\epsilon \rightharpoonup f \text{ in } L^p(D\times \Omega).	
         \end{equation*}
         It is also clear what we mean by convergence of functions $f_k \in L^p(\bfb_{\epsilon_k})$ for a subsequence $\epsilon_k \to 0$ as well as what we mean when we say we have convergence in the$\pal$ sense.
    \end{DEF}

    \begin{DEF}\label{def: weak H^1 solutions}
        Given $D\subset \mathbb{R}^d$, $\epsilon>0$ and $f \in L^2(\bfb_\epsilon)$  we will say that $v\in H(\bfb_\epsilon)$ is a weak solution of 
	          \begin{equation}\label{eqn: Poisson equation}
	               \left \{ \begin{array}{rll}
	         	   -\Delta v & = f & \text{ in } D\cap\mathcal{O}_\epsilon,\\
	               \partial_n v & = 0 & \text{ on } \partial \mathcal{O}_\epsilon,\\
	               \end{array}	\right.
	          \end{equation}    
	    if given any $\phi \in H_0(\bfb_\epsilon)$ we have
	    \begin{equation*}
	         \int_{\bfb_\epsilon} \nabla v\cdot \nabla \phi\;\d x\d\mathbb{P}(\omega) = \int_{\bfb_\epsilon} f \phi\;\d x\d\mathbb{P}.	
	    \end{equation*}
    \end{DEF}
    \begin{rem}\label{rem: perforated ball}
         Often, we will deal with \eqref{eqn: Poisson equation}  when $D$ is a ball, thus it will be convenient to write
         \begin{equation*}
              \mathcal{B}^\epsilon_r(x) := B_r(x) \cap \mathcal{O}_\epsilon.
         \end{equation*}    
    \end{rem}

    \begin{DEF}\label{def: Dirichlet problem}
    	 We will say $v$ solves the Dirichlet problem
		      \begin{equation}\label{eqn: Dirichlet problem in  D}
		             \left \{ \begin{array}{rll}
		                -\Delta v & = f & \text{ in } D\cap\mathcal{O}_\epsilon,\\
		                \partial_n v & = 0 & \text{ on } \partial \mathcal{O}_\epsilon,\\
		                v & = 0 & \text{ on } \partialom D,
		           \end{array}	\right.
		      \end{equation}
		 if in addition to solving \eqref{eqn: Poisson equation} in the weak sense we also have $v \in H_0(\bfb_\epsilon)$.        			
    \end{DEF}

    The mixed boundary problem \eqref{eqn: Dirichlet problem in  D} has a unique solution, which can be seen by applying the Lax-Milgram Theorem to a properly picked bilinear form in the Hilbert space $H_0(\bfb_\epsilon)$.

    \subsection{Viscosity solutions of the Free Boundary Problem}
    Given a (deterministic) domain $\mathcal{O} \subseteq\mathbb{R}^d$ and a function $u:\mathcal{O}
\times \mathbb{R}_+ \to \mathbb{R}_+$, we introduce the positive phase or ``droplet base'' of $u$ defined by
    \begin{align*}
         \db(u) & := \{ \; u>0 \; \}\;	\subset \mathcal{O}\times \mathbb{R}_+.
    \end{align*}
    Let us also define $\db_t(u)  := \{ u(\cdot,t) >0 \} \subset \mathcal{O}$. On the other hand, any bounded set $\db_0 \subset \mathcal{O}$ determines a function $u_0 \in H^1(\mathcal{O})$ by the conditions
	\begin{equation*}
		 -\Delta u_0 = 1 \text{ in } \db_0,\;\;\partial_n u_0 = 0 \text{ on } \partial \mathcal{O}, \;\; u_0 = 0 \text{ in } \mathcal{O}	\setminus \db_0.
	\end{equation*}
    Such $u_0$ gives the droplet profile associated to the base $\db_0$. Then, given such a $\db_0$ we consider the following Cauchy problem: to find a function $u: \mathcal{O}\times \mathbb{R} \to \mathbb{R}_+$ with $u(\cdot,0)=u_0$ which solves
 	\begin{equation}\label{eqn: deterministic Cauchy problem}
	     \left \{ \begin{array}{rll}
	          -\Delta u & = 1 & \text{ in }\{ u>0\}.\\ 
	          \partial_n u & = 0 & \text{ on } \partial \mathcal{O},\\
	          u_t  &= |\nabla u|^2 & \text{ on } \partialom \{ u>0\}.\\
	     \end{array}\right.
	\end{equation}
    It is well known that even for a smooth domain $\mathcal{O}$ and smooth initial data, solutions to \eqref{eqn: deterministic Cauchy problem} may become singular at some positive time, typically as the positive phase $\{u>0\}$  undergoes a topological change which induces a discontinuity in $u$. In this situation, some of the derivatives appearing \eqref{eqn: deterministic Cauchy problem} might cease to exist, and one must resort to a weak notion of solution in order to continue. Hence we now review the notion of viscosity solutions for \eqref{eqn: deterministic Cauchy problem}, as it was developed in \cite{K03}.

    \medskip

	We say $x_0\in\partial \mathcal{O}$ is a {\it regular} point if $\partial\mathcal{O}$ is differentiable at $x_0$. Note that $\mathcal{O}$ may not have a boundary (for example $\mathcal{O}=\mathbb{R}^d$). Also,  $E \subset \mathcal{O}\times \mathbb{R}_+$ will be called a cylindrical domain if it is of the form $A \times (a,b)$ for some open set $A$ in $\mathcal{O}$.

	\begin{DEF}\label{def: classical Subsolutions}
	     A  function $\varphi \in C(\bar{E})$ ($E$ a cylindrical domain) is a \emph{classical subsolution} of \eqref{eqn: deterministic Cauchy problem} in $E$ if 
	         \begin{itemize}
	              \item[(i)] $\varphi \in C_{x,t}^{2,1}(\overline{\{\varphi > 0\}}\cap E)$ and $\partial\{\varphi>0\}$ is locally $C^{1,1}_{x,t}$ in $E$.
	              \item[(ii)] $ -\Delta\varphi(x,t) \leq 1$ in $\{\varphi>0\}\cap E$.
	              \item[(iii)] $\varphi_n \leq 0$ on regular points of $\partial\mathcal{O} \cap \bar{E}$.
	              \item[(iv)] $\varphi_t \leq |\nabla \varphi|^2 $ on $( \partialom\{\varphi>0\}) \cap E$.
	         \end{itemize}
	         Here $\nabla \varphi(x,t)$ is taken to be the limit of $\nabla \varphi(y,s)$ as $(y,s) \to (x,t)$ from within the set $\{ \varphi>0\}$, and as before $n$ denotes the outward normal of $\mathcal{O}$. We say that $\varphi$ is a \emph{strict} classical subsolution if the inequalities in (iii) and (iv) are strict, lastly, \emph{classical supersolutions} are defined by reversing inequalities (ii),(iii) and (iv).
	\end{DEF}

	\begin{DEF}\label{def: viscosity solutions}
	     An upper (lower) semicontinuous function $u$ in $\mathcal{O}\times \mathbb{R}_+$ will be called a viscosity subsolution (supersolution) of Problem \eqref{eqn: deterministic Cauchy problem} with initial data $u_0$ if $u\leq u_0$ ($u\geq u_0$) for $t=0$ and if anytime we have a strict classical supersolution (subsolution) $\phi$ of \eqref{eqn: deterministic Cauchy problem} in some cylindrical domain $E$ with $u < \varphi$ ($u>\varphi$) on $\partial_P E$ then also $u < \varphi$ ($u>\varphi$) in $E$.
	\end{DEF}

	Given a locally bounded function $u$, we define its  lower and upper semi-continuous envelopes:
    \begin{equation}\label{eqn: semicontinuous envolopes}
         \begin{array}{rl}
		      \underline{u}(x,t) &  = \liminf \limits_{(y,s)\to (x,t)} u(y,s),\\
			  \overline{u}(x,t) & = \limsup \limits_{(y,s)\to (x,t)} u(y,s).
	     \end{array}	
    \end{equation}

	\begin{DEF}\label{def: discontinuous viscosity solutions}
	     A locally bounded function $u$ in $\mathcal{O}\times \mathbb{R}_+$ is a viscosity solution of \eqref{eqn: deterministic Cauchy problem} with initial data $u_0$ if $\overline{u}$ is a viscosity subsolution and $\underline{u}$ is a viscosity supersolution, and $u(.,0)=u_0$.
	\end{DEF}

	Now we can consider \eqref{eqn: deterministic Cauchy problem} for a (rescaled) stationary ergodic domain $\mathcal{O}_\epsilon$.

	\begin{DEF}\label{def: viscosity solution random domain}
	     Let us consider a stationary ergodic domain $\mathcal{O}(\omega)$, a compact subset $\db_0 \subset \mathbb{R}^d$ and $\epsilon>0$. A measurable function $u^\epsilon: \text{\bf{B}}_\epsilon \times \mathbb{R}_+ \to \mathbb{R}$ will be called a viscosity subsolution(supersolution, solution) of \eqref{eqn: Hele-Shaw O_ep} with initial data $u_0^\epsilon$ if $\mathbb{P}$-almost surely the function $u^\epsilon	(.,\omega)$ is a viscosity subsolution(supersolution, solution) of\eqref{eqn: deterministic Cauchy problem} with $\mathcal{O}=\mathcal{O}_\epsilon(\omega)$ and the corresponding initial data.
	\end{DEF}

	For the case when  $\mathcal{O}$ has a $C^3$ boundary we have a comparison principle, whose proof is parallel to the one in \cite{K03}. The regularity restriction on the boundary is to locally parametrize the Neumann boundary and use reflection argument to transform the flow in contact with the Neumann boundary into the flow without one (see section 3 of \cite{K05} where such parametrization is performed in a similar context). We suspect the comparison to hold with less regularity, but  we will not pursue this issue here.


    \begin{DEF}\label{def: strict precedence}
	     For nonnegative functions $u, v:\mathcal{O}\to \R$, which have compact supports and are continuous in the interior of their support, we write $u\prec v$ if $u< v $ in $\{u>0\}$ and $\overline{\{u>0\}} \subset \{v>0\}$. Following \cite{K03}, we say $u$ and $v$ are \emph{strictly separated}.
    \end{DEF}

	\begin{thm}\label{thm: comparison for Hele-Shaw C^3 boundary}
	     Let $u$ and $v$ be respectively viscosity sub- and supersolution of \eqref{eqn: deterministic Cauchy problem} with the domain $\mathcal{O}$ having $C^3$ boundary. Suppose $u(\cdot,0) \prec v(\cdot,0)$ in $\mathcal{O}$. Then $u(\cdot,t) \prec v(\cdot,t)$ for all $t\in [0,T)$. 
	\end{thm}

	Let $v$ be as given in the above theorem. When the initial free boundary $\partial\{v(x,0)>0\}$ is $C^1$, then Theorem 4.1 of \cite{ChKi2006} applies to $v$ to yield that $v(\cdot,0) \prec v(\cdot,\e)$ for any $\e>0$. Using this observation and Theorem~\ref{thm: comparison for Hele-Shaw C^3 boundary} we deduce the following:

	\begin{thm}\label{thm: cp2}
	     Let $u$ and $v$ be as given above, and let $\partial\{v(x,0)>0\}$ be $C^1$. If $u(\cdot,0) \leq v(\cdot,0)$ then $u(\cdot,t) \leq v^*(\cdot,t)$ for all $t\in [0,T)$. 
	\end{thm}

	We do not clarify the uniqueness of viscosity solutions for \eqref{eqn: Hele-Shaw O_ep} when $\partial \mathcal{O}_\e$ is less regular than $C^3$  or when $\partial\{v(x,0)>0\}$ is not $C^1$. In this case our result addresses the specific viscosity solution $u^\e$ given by the time derivative of the solution $p^\e$ of an auxiliary obstacle problem (see \eqref{eqn: epsilon obstacle problem} below).
	
    \subsection{Main results}\label{subsec: Main results}

	The main results require two assumptions on the sets $\mathcal{O}_\epsilon$, the second of which requires one further definition.

    \begin{DEF}\label{def: minimally smooth}
         An open set $P \subset \mathbb{R}^d$ is said to be \emph{minimally smooth} with constants $(\delta,N,M)$ if we may cover $\partial P$ by a countable sequence of open sets $\{U_i\}_i$ such that 
         \begin{enumerate}
              \item Each $x \in \mathbb{R}^d$ is contained in at most $N$ of the open sets $U_i$.
              \item For any $x\in\partial P$, the ball $B_\delta(x)$ is contained in at least one $U_i$.
              \item For any $i$, the portion of the boundary $\partial P$ inside $U_i$ agrees (in some cartesian system of coordinates) with the graph of a Lipschitz function whose Lipschitz semi-norm is at most $M$.	
         \end{enumerate}
    \end{DEF}

    We are finally ready to describe the two assumptions on $\mathcal{O}_\epsilon$.
 
	\begin{hyp}\label{hyp: Ergodic}
	     The domains $\mathcal{O}_\epsilon(\omega)$ are given by
	     \begin{equation*}
	          \mathcal{O}_\epsilon(\omega) := \{ \epsilon x \mid  x \in \mathcal{O}(\omega) \} = \{x \mid T_{\frac{x}{\epsilon}}\omega \in O \}	
	     \end{equation*}
	     where $\mathcal{O}(\omega)$ is a stationary ergodic domain and $T$ the associated action (see Definition \ref{def: stationary ergodic domain}).
	\end{hyp}
		
	\begin{hyp}\label{hyp: regularity}
	     There are constants $d_0,\delta,N,\text{ and }M$ (independent of $\omega)$ such that$\pal$ the complement of the set $\mathcal{O}(\omega)$ consists of a countable union of bounded sets $P_k(\omega)$ ($k\in\mathbb{N}$) such that we have
	     \begin{equation*}
	          d(P_k(\omega),P_j(\omega))\geq d_0 \text{ whenever } k\neq j.	
	     \end{equation*}
	     Here each set $P_k(\omega)$ is minimally smooth with constants $(\delta,N,M)$ and has diameter smaller than $d_0^{-1}$.
	\end{hyp}
	
	In light of Assumption \ref{hyp: regularity}, we make the following definition.

    \begin{DEF}\label{def: universal constants}
    	 A constant will be said to be \emph{universal} if it is determined by the constants $d_0,\delta,N,M$ in Assumption \ref{hyp: regularity} and the space dimension $d\geq 2$.
    \end{DEF}

	In a few words, Assumption \ref{hyp: regularity} says that $\mathcal{O}(\omega)$ is almost surely  obtained by populating $\mathbb{R}^d$ with many obstructions which are uniformly separated and sufficiently smooth (Lipschitz). In the homogenization literature and in the periodic case, this is known as a \emph{perforated domain}.
	
	\medskip

	\noindent \textbf{Initial Conditions.} We take an initial droplet base which is a bounded open set $\db_0 \subset \mathbb{R}^d$ and set $\db^\e_0:= \db_0 \cap \mathcal{O}_\e$,  we will assume that this set has no cracks, i.e. $\partialom \db^\epsilon_0=\partialom\overline{\db}^\epsilon_0$. Furthermore, we suppose that either $\db_0$ is star shaped or that $\partial\db_0$ is $C^1$. This condition guarantees the uniqueness of the solution for the limiting problem (see \cite{K03}).
	
	\medskip
	
	With this notation in hand, the first of our two main results is the following. 
	\begin{thm}\label{thm: main}
	     Let $\mathcal{O}_\epsilon$ satisfy Assumptions 1 and 2 and $\db_0^\epsilon$ be as above. Let $u^\epsilon$ solve \eqref{eqn: Hele-Shaw O_ep} with initial positive phase $\db_0^\epsilon$, and let $u$ solve \eqref{eqn: Hele-Shaw homogeneous} with initial positive phase $\db_0$ (both in the viscosity sense), where $\mu$ is as in Definition \ref{def: mu} and $A$ as in Definition \ref{def: A_0}. Then the following is true:
	     \begin{enumerate}
	          \item The free boundaries $\partialom \{ u^\epsilon>0\}$ converge uniformly to $\partial \{ u>0\}$ in the Hausdorff distance for $t>0$.
	          \item Locally uniformly in $t$ we have $u^\epsilon \to u\pal $ in $L^p\; \forall\; p\in[1,\infty)$  (see Definition \ref{def: notions of convergence}).
	          \item Moreover, we have the following pointwise limiting behavior
			       \begin{equation*}
			            \underline{u} \leq \liminf_{\epsilon\to 0} u^\e \leq \limsup_{\epsilon\to 0} u^\epsilon \leq \overline{u},
				   \end{equation*}
		      where $\underline{u}$ and $\overline{u}$ are as defined in \eqref{eqn: semicontinuous envolopes}.
		      \item In particular, if $u$ is continuous, then  $u^\epsilon \to u\pal $ in $L^\infty$ locally uniformly with respect to $t$.		
	     \end{enumerate}
	\end{thm}
	As explained earlier in the section, even for smooth initial data the solutions $u$ and $u^\epsilon$ may develop singularities, so classical solutions may not exist for all positive times and one must deal with viscosity solutions. We mention that topological changes for the positive phase of $u$ are ruled out for its support, for example, when $\db_0$ is star shaped (see Lemma \ref{lem: star shaped initial data}).  In any case, our result says that the oscillating free boundaries converge uniformly as $\epsilon\to 0$ even if topological changes and  discontinuities take place.
	
	\medskip
	
	The second result, which is needed for the proof of Theorem \ref{thm: main}, deals with the uniform continuity of solutions to the linear elliptic problem \eqref{eqn: Dirichlet problem in  D}. Since this is an interior regularity result we only need to state it for $D=B_r(x)$, and use the notation introduced in Remark \ref{rem: perforated ball}.
	
	\begin{thm}\label{thm: elliptic estimates} 
         Suppose that $\mathcal{O}$ satisfies Assumption \ref{hyp: regularity}. There are universal constants $C>0$ and $\alpha \in (0,1)$ such that if $v$ solves \eqref{eqn: Dirichlet problem in  D} in some ball $\mathcal{B}^\epsilon_r$, then 
         \begin{equation}\label{eqn: holder}
              \|v\|_{C^\alpha(\mathcal{B}^\epsilon_{r/2} )}	\leq C \left ( r^{-\alpha}\|v \|_{L^2(\mathcal{B}^\epsilon_r)}+r^2\|f\|_{L^\infty} \right ) \;\;\;\mathbb{P}\text{-a.s.}
         \end{equation}
         Moreover, if $w \geq 0$ in $\mathcal{B}^\epsilon_r$, we have a Harnack inequality,
         \begin{equation}\label{eqn: harnack}
              \sup \limits_{\mathcal{B}^\epsilon_{r/2}}	\;v  \leq C ( \;\inf \limits_{\mathcal{B}^\epsilon_{r/2}}\; v \;\;+\;\;r^2 \|f\|_{L^\infty}  ) \;\;\;\mathbb{P}\text{-a.s.}
         \end{equation}
    \end{thm}
	
    \subsection{Examples}\label{subsec: examples} Let us describe some of the random structures $\mathcal{O}(\omega)$ covered by Theorem \ref{thm: main}, namely those for which Assumptions \ref{hyp: Ergodic}  and \ref{hyp: regularity} hold (these are explained in Section 2.4). We also discuss cases that fall outside the scope of our current method as Assumption \ref{hyp: regularity} fails (but ``almost'' holds) for them.

	\begin{ex}\label{ex: cubic lattice} \emph{Site percolation with isolated obstacles:}  Consider a family of independent Bernoulli trials associated to each $z \in \mathbb{Z}^d$, defining a random subset $\mathbb{G}(\omega) \subset \mathbb{Z}^d$ where each $z$ is included probability $1-p$ and excluded with probability $p$, independently of the others.	 Let $Q= [0,1]^n$ and $\mathcal{I} \subset \subset Q$ be a domain with a Lipschitz boundary, then define 
	    \begin{equation*}
	        \mathcal{O}(\omega) = \mathbb{R}^d \setminus \union \limits_{z \in \mathbb{G}(\omega)}	\left ( \mathcal{I}+z\right ).
	    \end{equation*}
	    The random set $\mathbb{G}(\omega)$ has a stationary distribution since the Bernoulli trials are identically distributed and independent, in fact $\mathbb{G}$ is ergodic so Assumption \eqref{hyp: Ergodic} is satisfied. 
	\end{ex}

	\begin{ex}\label{ex: triangular lattice} \emph{Other regular lattices} A seemingly different example would be that of replacing $\mathbb{Z}^2$ in the above example (when $n=2$) with the \emph{triangular or hexagonal lattice}, namely  $\mathbb{L} \subset \mathbb{R}^2$ the set generated by linear combinations over $\mathbb{Z}$ of the vectors $e_1=(1,0)$ and $e_2=(1/2,1/\sqrt{2})$. We can run independent and identically distributed Bernoulli trials at each $z \in \mathbb{L}$ to get a random set $\mathbb{G}(\omega)\subset \mathbb{L}$. Take a Lipschitz domain $\mathcal{I}$ that is compactly contained in the triangle defined by $0,e_1$ and $e_2$, then as before let
		\begin{equation*}
	        \mathcal{O}(\omega) = \mathbb{R}^d \setminus \union \limits_{z \in \mathbb{G}(\omega)}	\left ( \mathcal{I}+z\right ).
	    \end{equation*}	
	\end{ex}
    One could also consider a more complicated model where the domain $\mathcal{I}$ is itself random, as long as it is minimally smooth$\pal$ with respect to some constants $(\delta,N,M)$ and it stays inside a fixed compact set inside  the cell defined by $0,e_1$ and $e_2$. Then $\mathcal{O}$ may look as given in Figure \ref{fig: triangular lattice}  ($\mathcal{O}^c$ is given by the black region)
	
	    \begin{figure}[h]
		     \centering
		     \includegraphics[height=0.35\textwidth]{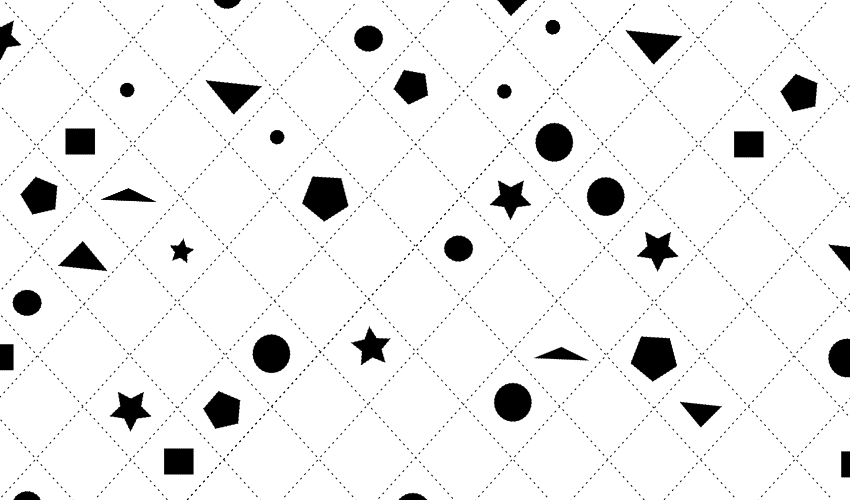}
		     \caption{A random domain $\mathcal{O}$ arising from percolation on a lattice.}\label{fig: triangular lattice}
	    \end{figure}

	\begin{ex}\label{ex: Chessboard} \emph{(Modified) Irregular chessboard}. Another example would be to consider an irregular chessboard with bounded cell sizes (compare with \cite[Example 3.4]{DaMo1986}), except now instead of coloring a cell completely, we only color a rescaled cell half its size. The domain $\mathcal{O}$ obtained in this manner does not have a lattice structure, but it does however comply with Assumptions \ref{hyp: Ergodic} and \ref{hyp: regularity} 
	\end{ex}

	\begin{ex}\label{ex: Percolation} \emph{General Site Percolation.} If in Example \ref{ex: cubic lattice} we do not assume that $\mathcal{I} \subset \subset Q$ then Assumption \ref{hyp: regularity} no longer holds and our method fails. This is because we have no means of proving a global Sobolev embedding and Poincar\'e inequalities if the perforations are allowed to be arbitrarily close to each other and are allowed to form large connected clusters. In fact, one does not expect the Sobolev and Poincar\'e inequalities to hold$\pal$ in the general case. The best one can expect is that they hold in a given ball with a probability that approaches $1$ very fast as $\epsilon \to 0$. This phenomenon is by now well understood in the context of discrete percolation, as discussed in Section 1.2.	
	\end{ex}

	\begin{ex}\label{ex: Poisson rain}\emph{Poisson Rain.} For similar reasons to Example \ref{ex: Percolation} we cannot handle the case where $\mathcal{O}(\omega)$ is built as the union of balls $B_{r_i}(x_i)$ where the points $x_i$ are given by a Poisson point process and the radii $r_i$ are independent and identically distributed. Again, it would be worthwhile to investigate the extent to which the current techniques can be pushed to handle this case. This model is also known as the Boolean model in the stochastic geometry literature \cite[Chapter 12]{Gri1999}.
	\end{ex}
	
    \begin{figure}[h]
	     \centering
	     \includegraphics[height=0.35\textwidth]{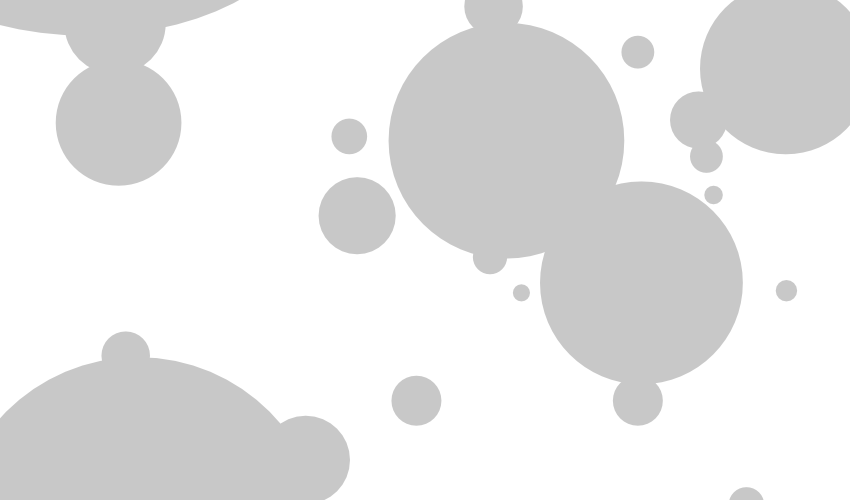}
	     \caption{An example where our assumptions do not hold is that of the Poisson rain.}\label{fig: Poisson rain}
    \end{figure}

	\subsection{Strategy}\label{subsec: strategy}		

	In the analysis of \eqref{eqn: Hele-Shaw O_ep} a constant challenge is the lower dimensional or singular nature of the free boundary condition. Namely, \eqref{eqn: Hele-Shaw O_ep} can be written (in the sense of distributions) as
	\begin{equation*}
	     \partial_t\mathbb{I}_{\{u^\epsilon>0\} }-\Delta u^\epsilon = \mathbb{I}_{\{u^\epsilon>0\} },
	\end{equation*}
    together with zero Neumann boundary conditions on $\partial \mathcal{O}_\epsilon$. The singular term is the time derivative of the characteristic function of the positivity set of $u^\epsilon$. One obvious way to get rid of the singular time derivative is integrating the equation with respect to time. In particular, we may consider the function 
    \begin{equation*}
         p^\epsilon(x,t,\omega):=\int_0^t u^\epsilon(x,s,\omega)\;ds
    \end{equation*}	
    and seek an evolution equation for $p^\epsilon$. This is particularly tractable in the one phase case, since $u^\epsilon$ is monotone in $t$. In fact, it was already known to Elliot-Janovsky \cite{ElJa} that whenever $u^\epsilon$ is smooth enough then $p^\epsilon: \mathcal{O}_\epsilon\times \mathbb{R}_+ \to \mathbb{R}$ is a non-negative function satisfying
    \begin{equation}\label{eqn: epsilon obstacle problem}	
    	 -\Delta p^\epsilon(x,t) = f^\epsilon(x,t)\mathbb{I}_{\{p^\epsilon>0\} }(x,t),
    \end{equation}
    and $\partial_n p^\epsilon = 0$ on $\partial \mathcal{O}_\epsilon$, where
    \begin{align}\label{eqn: definition of f^epsilon}
	     f^\epsilon(x,t) & =-\mathbb{I}_{(\db_0^\epsilon)^c}+\int_0^t\mathbb{I}_{\{p^\epsilon(\cdot,s)>0 \}}(x)\;ds.	
    \end{align}
    This says that $p^\epsilon$ solves an obstacle problem \cite{KiSt1980}. The reason studying $p^\epsilon$ might be useful is that (assuming smoothness) it is not hard showing that
    \begin{equation*}
    	 \{ p^\epsilon >0 \} = \{ u^\epsilon>0\}.    	
    \end{equation*}
    Thus (heuristically), one can study $\{ u^\epsilon>0\}$ using the methods available for the obstacle problem, which apply to $\{p^\epsilon>0\}$.

     In general, $u^\epsilon$ will not be smooth, and it may even become discontinuous at the onset of topological changes of $\{ u^\epsilon(\cdot,t)>0\}$, which may occur even if the initial data is smooth. This means we have to justify the above correspondence when $u^\epsilon$ is just a viscosity solution of \eqref{eqn: Hele-Shaw O_ep}. It is not clear how one may show that $p^\epsilon:= \int_0^tu^\epsilon(\cdot,s)\;ds$ is even a weak solution of \eqref{eqn: epsilon obstacle problem} without differentiating $u^\epsilon$ directly. 

    Following the method developed in \cite{KM09} we invert the previous construction. Instead of defining $p^\epsilon$ as the integral of $u^\epsilon$, we define it as a solution of the free boundary problem we expect it to solve (namely \eqref{eqn: epsilon obstacle problem} and \eqref{eqn: definition of f^epsilon}). Only later we show that it agrees with the time integral of $u^\epsilon$ (see Theorem \ref{thm: evolution}).

    \begin{DEF}\label{def: weak solution p^epsilon}
	     Let $\db_0^\epsilon \subset \mathcal{O}_\epsilon$ be as in Section \ref{subsec: Main results}. A measurable function $p^\epsilon: \bfb_\epsilon \times \mathbb{R}_+ \to \mathbb{R}$ will be called a weak solution of \eqref{eqn: epsilon obstacle problem} if$\pal$ in $\omega$ we have that $(x,t) \to p^\epsilon(x,\omega,t)$ belongs to the space
	     \begin{equation}\label{eqn: admissible functions epsilon obstacle problem}
	          \mathcal{A}:=\{ \phi: \mathcal{O}_\epsilon(\omega) \times \mathbb{R}_+ \to \mathbb{R} \mid \phi\geq 0,\;\;\phi(\cdot,t) \in H^1(\mathcal{O}_\epsilon(\omega))\;\forall\;t \}.	
	     \end{equation}
         Moreover, with $f^\epsilon$ given by \eqref{eqn: definition of f^epsilon} we have
	     \begin{equation*}
	          \int_{\mathcal{O}_\epsilon} \nabla (p^\epsilon-\phi)\cdot \nabla p^\epsilon -(p^\epsilon-\phi)f^\epsilon\;dx \geq 0
	     \end{equation*}
	     for any fixed time $t$ and any $\phi \in \mathcal{A}$.
    \end{DEF}

	The one to one correspondence of \eqref{eqn: Hele-Shaw O_ep} to an obstacle problem suggests that to identify the limit of $u^\epsilon$ we should first homogenize the problem solved by $p^\epsilon$. To do this, fix $t>0$ and let $p = \int_0^tu(\cdot,s)\;ds$, $u$ being the solution to the homogenized problem \eqref{eqn: Hele-Shaw homogeneous} with initial data $\db_0$. The linear homogenization theory (which we revisit in Section \ref{sec: linear homogenization revisited}) says that if we can show that
	\begin{equation*}
		 \mathbb{I}_{\mathcal{O}_\epsilon} f^\epsilon\mathbb{I}_{\{p^\epsilon>0 \}} \rightharpoonup \mu f \mathbb{I}_{\{p>0\}} \pal \text{ in } L^2(\mathbb{R}^d),
	\end{equation*}
    ($\mu$ as in Definition \ref{def: mu}) then $p^\epsilon$ converges $\mathbb{P}$-a.s. in $L^2_{\text{loc}}$ to a $p'$ defined in all of space and time, solving
    \begin{equation}\label{eqn: L^1 convergence of supports}
         -\dive(A\nabla p') = \mu \left (-\mathbb{I}_{\db_0^c }+\int_0^t\mathbb{I}_{\{p>0\} }\;ds \right ) \mathbb{I}_{\{p>0\} }.	
    \end{equation}
    Then, $p$ and $p'$ would solve the same elliptic equation with the same decay at infinity, and we would conclude that $p'=p$ and therefore that $p^\epsilon \to p$ for each fixed time.  

    Instead of proving \eqref{eqn: L^1 convergence of supports} directly, we will use the uniform continuity bounds for $p^\epsilon$ given by the theory in Section \ref{sec: elliptic regularity} and Section \ref{sec: weak fb results}. These estimates give us enough compactness to extract a locally uniformly converging subsequence (in space and time) from any sequence $p^{\epsilon_k}$ with $\epsilon_k \to 0$ (see Proposition \ref{prop: uniform convergence of p^epsilon}). Then, all we have to show is that if $\tilde p$ is the limit of such a subsequence, then it solves (for each fixed $t$)
    \begin{equation*}
         -\dive(A \nabla \tilde p) = \mu \left (-\mathbb{I}_{\db_0^c }+\int_0^t\mathbb{I}_{\{\tilde p>0\} }\;ds \right ) \mathbb{I}_{\{\tilde p>0\} }.
    \end{equation*} 
    \begin{rem}
	    It is worth emphasizing how this last equation is different from \eqref{eqn: L^1 convergence of supports}. In the previous equation we had a function $p'$ on the left and the positivity set of $p$ on the other side, while here we have the same function $\tilde p$ on both sides of the equation.    	
    \end{rem}
    In this case, since the problem has a unique solution for a fixed initial set $\db_0$, we would conclude that regardless of the initial sequence $\epsilon_k$ we have $\tilde p = p$, obtaining the homogenization for the obstacle problem \eqref{eqn: epsilon obstacle problem}. The key to obtaining the above identity is the stability of the free boundaries $\partialom\{p^\epsilon>0\}$ as $\epsilon \to 0$ (Lemma \ref{lem: free sets limit}), the proof of which relies on a uniform non-degenerate growth estimate for $p^\epsilon$ (Theorem \ref{thm: nondegeneracy}). It is mostly for this single convergence result that we develop the elliptic theory in Section \ref{sec: elliptic regularity} and construct the special barrier in Section \ref{sec: weak fb results}.

    Later in Section \ref{sec: homogenization} we will show that any $p^\epsilon$ solving \eqref{eqn: Hele-Shaw O_ep} gives rise to a function $u^\epsilon$ which solves \eqref{eqn: Hele-Shaw O_ep} in the viscosity sense (see Theorem \ref{thm: evolution}), and such that $p^\epsilon = \int_0^tu^\epsilon(\cdot,s)\;ds$. Using this Theorem and the strong convergence obtained for $\{ p^\epsilon>0 \}$ we are able to prove that $u^\epsilon$ converges to the solution of the homogenized problem, proving Theorem \ref{thm: main}.
		
\section{Elliptic Regularity in $\mathcal{O}_\epsilon$}\label{sec: elliptic regularity}

    \subsection{Sobolev and Poincar\'e inequalities}
    The goal of this section is showing that$\pal$ functions in $\mathcal{O}_\epsilon$ satisfy inequalities similar to the standard Sobolev and Poincar\'e inequalities in $\mathbb{R}^d$. This is where the strength of Assumption \ref{hyp: regularity} is most needed.

    We recall a classical result due to Calder\'on and Stein, concerning the extension of Sobolev functions defined in Lipschitz domains. For a proof see \cite[Chapter VI Theorem 5]{St1970}.

    \begin{thm}\label{thm: extension for minimally smooth domains}
    	 Let $A \subset \mathbb{R}^d$ be a minimally smooth domain with constants $(\delta,N,M)$, then there exists a bounded linear operator $E:H^1(A)\to H^1(\mathbb{R}^d)$ with its operator norm depending only on $\delta,n,M$.
    \end{thm}

    This extension theorem will allow us to exploit the desired inequalities from $\mathbb{R}^d$ and pass them to $\mathcal{O}_\epsilon$.

    \begin{prop}\label{prop: Poincare} Let $A \subset \mathbb{R}^d$ be a connected domain which is minimally smooth  with constants $(\delta,N,M)$ and diameter no larger than $d_0^{-1}$. There is a constant $C=C(\delta,N,M,d_0)$ such that for any $\phi \in H^1(A)$ we have the following Poincar\'e type inequality
         \begin{align*}
              \| \phi-(\phi)_A\|_{L^2(A)}\leq C\|\nabla \phi \|_{L^2(A)}\\
              \text{where } (\phi)_A:= |A|^{-1}\int_A \phi\;dx.	
         \end{align*}
    \end{prop}

    \begin{proof}
	     We argue by contradiction, relying on compactness. Suppose that for some combination of $(\delta,N,M)$ and $d_0$ such a constant $C$ does not exist. Then for each $j \in\mathbb{N}$ we can find a connected, minimally smooth domain $A_j$ with respect to $(\delta,N,M)$, with diameter no larger than $d_0^{-1}$, and a function $\phi_j \in H^1(A_j)$ such that
         \begin{equation*}
              \|\phi_j-(\phi_j)_{A_j}\|_{L^2(A_j)} \geq 2^j\|\nabla \phi_j\|_{L^2(A_j)}>0. 	
         \end{equation*}
         Moreover, after a translation, we may assume that all the $A_j$'s contain the origin, so that $A_j \subset B_{2d_0^{-1}}(0)$. On the other hand, each $\phi_j$ is not identically equal to $(\phi)_{A_j}$ in $A_j$, so we may introduce normalized functions
         \begin{equation*}
              \psi_j := \|\phi_j-(\phi_j)_{A_j}\|_{L^2(A_j)}^{-1}(\phi_j-(\phi_j)_{A_j}),	
         \end{equation*}
         so that
         \begin{equation*}
              (\psi_j)_{A_j}=0\;,\;\;\|\psi_j\|_{L^2(A_j)}=1,\;\text{ and } \;\|\nabla \psi_j \|_{L^2(A_j)}\leq 2^{-j}.	
         \end{equation*}
         All the sets $A_j$ are minimally smooth with the same constants, then Theorem \ref{thm: extension for minimally smooth domains} says we can find constant $C_0=C_0(\delta,N,M)$ and functions $\tilde \psi_j \in H^1(\mathbb{R}^d)$ with $\tilde \psi_j = \psi_j$ in $A_j$, such that
         \begin{equation*}
              \|\tilde \psi_j\|_{H^1(\mathbb{R}^d)}\leq C_0,\;\;\|\psi_j\|_{H^1(A_j)}\leq C_0(1+2^{-j}).
         \end{equation*}
         Since $\tilde \psi_j$ extends $\psi_j$, we have $\|\tilde \psi_j\|_{L^2(\mathbb{R}^d)}\geq 1$ for each $j$. Moreover, given the gradient estimate we may assume, by passing to a subsequence, that $\tilde \psi_j$ converges strongly in $L^2(\mathbb{R}^d)$ and weakly in $H^1(\mathbb{R}^d)$ (and then in fact, strongly in $H^1(\mathbb{R}^d)$) to some $\psi_\infty \in H^1(\mathbb{R}^d)$. 

         On the other hand, the fact that $A_j \subset B_{2d_0^{-1}}$ for any $j$ and that is minimally smooth show that after passing to a subsequence $A_j$ converges in the Hausdorff topology to some set $A_\infty$. Furthermore, it is not hard to see that $A_\infty$ will be connected and minimally smooth with respect to $(\delta,N,M)$, and in particular it has a non-empty interior. In conclusion,  for every $j$ we have
         \begin{equation*}
              (\tilde \psi_j)_{A_j}=0,\;\; \|\tilde \psi_j\|_{L^2(A_j)}=1,\;\;\|\nabla \tilde \psi_j\|_{L^2(A_j)}\leq 2^{-j},	
         \end{equation*}
         from where it follows that $(\psi_\infty)_{A_\infty}=0$ and $\nabla \psi_\infty\equiv 0$ in $A_\infty$, and since $A_\infty$ is connected $\psi_\infty$ must be identically zero a.e. in $A_\infty$, but we also must have $\|\psi_\infty\|_{L^2(A_\infty)}=1$ which is impossible. This contradicts the existence of the original sequence $\phi_j$ and proves the proposition.
    \end{proof}

    Although Poincar\'e's inequality itself is used in proving elliptic estimates, here we need it only to prove the Sobolev inequality, which is what will we use in the proof of Theorem \ref{thm: elliptic estimates}. We also recall the definition of a universal constant given in Definition \ref{def: universal constants} which will be used extensively.

    \begin{prop}\label{prop: Sobolev} 
	     Assume $\mathcal{O}$ satisfies Assumption \ref{hyp: regularity}, then there is a universal $C$ such that for $d\geq 3$ 
         \begin{equation}\label{eqn: Sobolev in O_epsilon}
              \|\phi\|_{L^{2^*}( \mathcal{O}_\epsilon)} \leq C\|\nabla \phi\|_{L^2(\mathcal{O}_\epsilon)}, \;\;\;\forall\;\; \phi \in H^1(\mathcal{O}_\epsilon(\omega)),
         \end{equation}
         for $\mathbb{P}$-almost every $\omega$, here $2^*:= 2d/(d-2)$.
    \end{prop}
    \begin{rem} As it is well known in the case $d=2$ and $\mathcal{O}=\mathbb{R}^2$ the above estimate with $L^\infty$ on the right hand side does not hold. Instead, the relevant estimate for $d=2$ says that $u$ is a function of bounded mean oscillation. Namely, there is a universal constant $C$ such that for any $x$ and $r>0$
	     \begin{equation*}
	     	  \int_{\mathcal{B}^\epsilon_r(x)}|u-(u)_{x,r}|\;dx\;\leq C\|\nabla u\|_{L^2(\mathcal{O}_\epsilon)}r^2,\;\;(u)_{x,r}:=\int_{\mathcal{B}^\epsilon_r(x)}u(y)\;dy.
	     \end{equation*}
         Although all of our estimates hold for $d=2$, as long as one uses the above inequality in place of the Sobolev embedding. For the sake of presentation we will only work our arguments explicitly when $d\geq 3$.
    \end{rem}

    \begin{proof}
         For any $k,\omega$, we denote by $\hat P_k(\omega)$ a  $d_0/4$-neighborhood of $P_k(\omega)$. By assumption, almost surely the sets $P_k(\omega)$ are minimally smooth (with constants independent of $k,\omega$). Then,  by Theorem \ref{thm: extension for minimally smooth domains} case we know that there is an \emph{extension operator} $E_k$ such that
         \begin{align*}
              E_k: H^1(\hat P_k(\omega) \setminus P_k(\omega) ) \to H^1(\hat P_k(\omega))
         \end{align*}
         where for some $C=C(\delta,N,M)$  (and thus independent of $k$), we have
         \begin{equation*}
              \|E_k\phi\|_{H^1(\hat P_k(\omega))}\leq C \|\phi\|_{H^1(\hat P_k(\omega) \setminus P_k(\omega))}.    	
         \end{equation*}
         We define new extensions $\hat E_k: H^1(\hat P_k(\omega)\setminus P_k(\omega))\to H^1(P_k(\omega))$ by
         \begin{equation*}
              \hat E_k \phi := E( \phi - (\phi)_k)+(\phi)_k	
         \end{equation*}
         and, putting them all together, an extension $\hat E: H^1(\mathcal{O}) \to H^1(\mathbb{R}^d)$ given by:
              \begin{equation*}
                   \hat E\phi(x) := \left \{ \begin{array}{rl}
                   \phi(x) & \text{ whenever } x \in \mathcal{O}\\
                   \hat E_k\phi(x) & \text { whenever } x \in \hat  P_k(\omega).	
              \end{array}	\right.
         \end{equation*}
         Now, in $\hat P_k(\omega) \setminus P_k(\omega)$ we have $\hat E_k \phi = (\phi-(\phi)_k)+(\phi)_k=\phi$, thus
         \begin{equation*}
              \|\nabla \hat E_k\phi\|_{L^2(P_k(\omega))}=\|\nabla E_k(\phi-(\phi)_k )\|_{L^2(P_k(\omega))}\leq C \|\phi-(\phi)_k\|_{H^1(\hat P_k(\omega))}.
         \end{equation*}
         Then, due to Assumption \ref{hyp: regularity} we may apply Proposition \ref{prop: Poincare} to $\hat P_k(\omega)$, so there is a universal $C$ such that
         \begin{equation*}
               \|\phi-(\phi)_k\|_{H^1(\hat P_k(\omega))}\leq C \|\nabla \phi\|_{L^2(\hat P_k(\omega))}		
         \end{equation*}
         since this holds for every $k$, we have proved that
         \begin{equation*}
              \int_{\left (\union \limits_k \hat P_k(\omega)\right )}|\nabla E\phi|^2\;dx \leq C^2\int_{\mathcal{O}}|\nabla \phi|^2\;dx.	
         \end{equation*}
         Therefore, 
         \begin{equation*}
              \int_{\mathbb{R}^d}|\nabla E\phi|^2\;dx \leq (1+C^2)\int_{\mathcal{O}}|\nabla \phi|^2\;dx.
         \end{equation*}
         Next, observe that, since $E\phi=\phi$ in $\mathcal{O}$ we get the bound
         \begin{equation}\label{eqn: extension bound}
              \|\phi\|_{L^{2^*}(\mathcal{O}_1)}\leq	\|E\phi\|_{L^{2^*}(\mathbb{R}^d)}.
         \end{equation}	
         Applying the standard Sobolev embedding we get
         \begin{equation*}
              \|\phi\|_{L^{2^*}(\mathcal{O}_1)}\leq	C_n\|\nabla E\phi\|_{L^{2}(\mathbb{R}^d)},
         \end{equation*}
         so, using the bound \eqref{eqn: extension bound}, we obtain for a universal $C$ 
         \begin{equation*}
              \|\phi\|_{L^{2^*}(\mathcal{O}_1)}	\leq C\|\nabla \phi\|_{L^2(\mathcal{O}_1)}.	
         \end{equation*}
         This is the desired estimate when $\epsilon=1$. These norms scale in the same way, so we may use the change of variables $x\to \epsilon x$ to obtain the estimate for general $\epsilon>0$ by reducing it to the inequality above.
    \end{proof}

    Next, we prove an inequality similar to Poincar\'e's inequality, except now we assume the function vanishes in a fixed portion of the ball. It extends a corresponding inequality in $\mathbb{R}^d$ used by Moser in his original proof of the parabolic Harnack inequality \cite{Mo1960}. It seems weaker than De Giorgi's isoperimetric inequality for $H^1$ functions, which says that if $u\in H^1(B_1)$ and $0\leq u\leq 1$, then
    \begin{equation*}
         |\{u=0\}||\{u=1\}|^{2/n} \leq C_n\|\nabla u\|_{L^2(B_1)}|\{0<u<1\}|^{1/2}.	
    \end{equation*}
    This inequality was used by De Giorgi in his proof of the H\"older regularity for elliptic equations \cite{DeG1957}. We were not able to prove an analogue to De Giorgi's inequality due to the irregular geometry of the perforated domain. However, we observe that for the purposes of regularity and Harnack's inequality, our extension of Moser's estimate is enough. We recall the statement of Moser's estimate, see \cite[Lemma 2]{Mo1960} for a proof.

    \begin{lem}\label{lem: Moser estimate}
    	 Let $\phi \in H^1(B_r)$ and $d\geq 3$. Then for any $m \in (0,1)$ there is a positive constant $C$ determined by $d$ and $m$ such that
         \begin{equation*}
         	  \left ( \int_{B_r}|\phi|^{2^*}\;dx\right)^{\tfrac{2^*}{2}} \leq C\left ( r^2\int_{B_r}|\nabla \phi|^2\;dx+\int_{A}\phi^2\;dx \right ),
         \end{equation*} 
         where $A$ is \emph{any} subset of $B_r$ such that $|A|\geq m|B_r|$.
    \end{lem}

    The extension of Lemma \ref{lem: Moser estimate} is as follows.
    \begin{prop}\label{prop: general Moser estimate}   
         Let $\phi \in H^1(\mathcal{B}^\epsilon_r(x_0))$ and $d\geq 3$.  Then for any $m \in (0,1)$ and $d\geq 3$ there are positive constants $C$ and $c_1$ determined by $\delta,N,M,d_0,d$ and $m$, such that$\pal$ the lower estimate
         \begin{equation*}
	          |\{x \in \mathcal{B}^\epsilon_r(x_0) \mid \phi = 0\}|\geq m |\mathcal{B}^\epsilon_r(x_0)|
         \end{equation*}
         guarantees the bound
         \begin{equation*}
              \|\phi\|_{L^2(\mathcal{B}^\epsilon_{c_1r}(x_0))}\leq Cr^2\|\nabla \phi\|_{L^2(\mathcal{B}^\epsilon_{r}(x_0))}.	
         \end{equation*}
    \end{prop}

    \begin{proof}
         As in the previous proposition, let us consider first the case $\epsilon=1$. Let $I \subset \mathbb{N}$ be the set of indices $k$ such that $\bar P^k(\omega ) \cap \bar{\mathcal{B}}^1_{r} \neq 0$, and let $\hat P_k(\omega)$ be as in the proof of Proposition \ref{prop: Sobolev}. Consider the domain 
         \begin{equation*}
    	      \mathcal{F}:= \mathcal{B}^1_r \cup \union\limits_{k \in I} \hat P_k(\omega).
         \end{equation*} 
         Arguing as before, there exists a bounded extension operator $E_0: H^1(\mathcal{F})\to H^1(\mathbb{R}^d)$ whose norm depends only on $\mathcal{O}(\omega)$. First suppose that $r\geq 1$ and observe that in this case $\mathcal{F} \subset B_{Cr}$ for some universal $C$. Since we have a function $E_0\phi$ which vanishes in a portion of $B_{Cr}$,  Lemma \ref{lem: Moser estimate} says that
         \begin{equation*}
         \|E_0\phi\|_{L^2(B_{Cr})}\leq C'r^2\|\nabla E_0\phi\|_{L^2(B_{Cr})},	
    \end{equation*}
    but now one can see that for larger (but still universal) constants $C'$ and $C''$, we have
    \begin{equation*}
         \|\nabla E_0\phi\|_{L^2(B_Cr)}\leq C''r^2\|\nabla \phi\|_{L^2(\mathcal{B}^1_{C''r})}.
    \end{equation*}
    Hence it follows that as long as $r\geq 1$ we have
    \begin{equation*}
         \|\phi\|_{L^2(\mathcal{B}^1_r)}\leq C''r^2\|\nabla	\phi\|_{L^2(\mathcal{B}^1_{C''r})}.
    \end{equation*}
    If we can prove the inequality for $r\leq c_0$ for some universal $c_0$ then we would have proved the proposition via a doubling argument. Let $c_0$ be small enough so that $r<c_0$ guarantees that $B_r(x_0 ) \cap \partial \Omega$ is the graph of a Lipschitz function in some Cartesian system of coordinates. In this case, a Lipschitz change of variables allows us to reduce the estimate to the case of a half ball. Here we may argue by a reflection argument to reduce it again to Lemma \ref{lem: Moser estimate}, and the proposition is proved.

    \end{proof}
    The proof of Theorem \ref{thm: elliptic estimates} has two stages, first for $f=0$ we prove Harnack's inequality using a combination of De Giorgi's and Moser's approaches (to get the $L^\infty$ estimate and oscillation estimate, respectively). 

    Once this is done, we prove a Stampacchia-type estimate to handle $f \neq 0$, with $f\in L^p$ ($p>d/2$). Combining this estimate with Harnack's inequality from the homogeneous case we get the general result. For a thorough presentation of similar ideas for the classical case of a uniformly elliptic operator with bounded measurable coefficients in $\mathbb{R}^d$, see \cite[Chapter II]{KiSt1980}. 

    \subsection{Harnack inequality (homogeneous problem)}

    As we start the proof of the regularity estimates, we recall the notion of a subsolution.

    \begin{DEF} A function $v$ is said to be a weak subsolution of \eqref{eqn: Dirichlet problem in  D} if for any positive test function $\phi$ whose support is compactly contained in $B_r$, we have
         \begin{equation*}
	          \int_{\mathcal{O}_\epsilon} \nabla v \cdot \nabla \phi\;dx \leq \int_{\mathcal{O}_\epsilon} f \phi\;dx
	     \end{equation*}
         We define a weak supersolution similarly by reversing the last inequality above. Note that a function which is both a subsolution and supersolution in this sense is a weak solution (see Definition \ref{def: weak H^1 solutions}).
    \end{DEF}

    \begin{rem}\label{rem: subsolutions_+} Given $w \in H^1$, the function $w_\lambda : = (w-\lambda)_+$ is also in $H^1$. It is well known and not hard to see that whenever $w$ is a weak solution to \eqref{eqn: Poisson equation} then $w_\lambda$ is a weak subsolution for any $\lambda\in\mathbb{R}$.
    \end{rem}

    We next recall the Cacciopoli or energy inequality for subsolutions.

    \begin{prop}\label{prop: energy inequality}
         Let $w$ be a weak subsolution with $f \equiv 0$. Then for any $\lambda \in \mathbb{R}$ and any Lipschitz function $\eta$ with compact support in $B_r$ we have the inequality
         \begin{equation}\label{eqn: energy inequality}
              \int_{\mathcal{O}_\epsilon} |\nabla (\eta w_\lambda)|^2dx \leq  \int_{\mathcal{O}_\epsilon} |\nabla \eta|^2 (w_\lambda)^2dx	
         \end{equation}	
    \end{prop}
    \begin{proof}
         The function $w$ is a weak subsolution, so for any positive $\phi \in H^1_0(B_r)$ we have
         \begin{equation*}
              \int_{\mathcal{O}_\epsilon} \nabla w \cdot \nabla \phi\;dx \leq 0
         \end{equation*}	
         Putting $\phi = \eta^2  w_\lambda$, we get
         \begin{equation*}
              \int_{\mathcal{O}_\epsilon } \nabla w \cdot \left (2\eta w\lambda \nabla \eta + \eta^2 \nabla w_\lambda \right )\;dx \leq 0	
         \end{equation*}
         we ``complete the square'' on the left by adding $\eta^2|\nabla \eta|^2 w_\lambda^2$ to both sides, getting
         \begin{equation*}
              \int_{\mathcal{O}_\epsilon} \nabla w \cdot \left (2\eta w_\lambda \nabla \eta + \eta^2 \nabla w_\lambda \right ) +(w_\lambda)^2|\nabla \eta|^2\;dx \leq \int_{\mathcal{O}_\epsilon} |\nabla \eta|^2 (w_\lambda)^2\;dx.	
         \end{equation*}
         Indeed, since $\nabla (\eta w_\lambda) = \eta\nabla w_\lambda+w_\lambda \nabla \eta$ the integral on the left equals
         \begin{equation*}
         \int_{\mathcal{O}_\epsilon} |\nabla (\eta w_\lambda)|^2dx	
         \end{equation*}
         which gives the first inequality. 
    \end{proof}

    Before we prove the first $L^\infty$ estimate, we need the following elementary proposition.

    \begin{prop}\label{prop: sequence}
         Let $A_k$ be a sequence of positive real numbers and $\alpha,\beta,\delta>0$ numbers such that
         \begin{equation*}
              A_{k+1}\leq \alpha2^{\beta k}A_k^{1+\delta}	\;\;\text{ for each } k \in \mathbb{N}.
         \end{equation*}	
         Then $\lim \limits_{k \to +\infty} A_k = 0$ if 
         \begin{equation*}
              A_0 \leq 2^{-\beta/\delta^2}\alpha^{-1/\delta}.	
         \end{equation*}
    \end{prop}
     \begin{proof}
         We shall show that $A_k \leq 2^{-k\mu}A_0$ for some $\mu>0$, from where the proposition will follow. When $k=0$ this is trivially true, so suppose that for some $k_0$ we have
         \begin{equation*}
              A_{k_0} \leq 2^{-k_0\mu}A_0.
         \end{equation*}
         Then, by the hypothesis
         \begin{equation*}
              A_{k_0+1}\leq \alpha2^{\beta k_0}A_{k_0}^{1+\delta}\leq \alpha 2^{\beta k_0} 2^{-k_0\mu(1+\delta)}A_0^{1+\delta},
         \end{equation*}
         so we are done as long as
         \begin{equation*}
              \alpha2^{\beta k_0}2^{-k_0\mu(1+\delta)}A_0^\delta \leq 2^{-(k_0+1)\mu}.
         \end{equation*}
         If we set $\mu=\beta/\delta$, this last inequality is equivalent to
         \begin{equation*}
              A_0^\delta \leq 2^{-(k_0+1)\mu-\beta k_0+k_0\mu(1+\delta)}\alpha^{-1} = 2^{-\mu-\beta k_0+\mu\delta k_0}\alpha^{-1} = 2^{-\mu}\alpha^{-1}
         \end{equation*}
         so $A_0\leq 2^{-\beta/\delta^2}\alpha^{-1/\delta}$ guarantees that $A_k \to 0$ as we wanted. 
    \end{proof}

    We are now ready to prove that subsolutions are$\pal$ bounded from above, at least when $f=0$. The proof relies on the fact that if $v$ is a subsolution then $v_\lambda$ satisfies both the Sobolev and energy inequalities for all $\lambda \in \mathbb{R}$.
 
    \begin{lem}\label{lem: De Giorgi pointwise bound}
         Let $v$ be a weak subsolution of \eqref{eqn: Dirichlet problem in  D}. Then for some universal $C$ we have
         \begin{equation*}
              \sup \limits_{\mathcal{B}^\epsilon_{r/2}} v^2 \leq \frac{C}{r^d} \int_{\mathcal{B}^\epsilon_r} v^2\;dx\;\;\pal	
         \end{equation*}	
    \end{lem}

    \begin{proof}
         As it is usual in the energy method, for each $k \in \mathbb{N}$ we define
         \begin{align*}
              r_k & := r(1+2^{-k}),\\
              M_k & := M_0(1-2^{-k}),
         \end{align*}
         and subsolutions (see Remark \ref{rem: subsolutions_+})  $v_k  := (v-M_k)_+$. We also define cut-off functions $\eta_k \;\in\; \text{Lip}(B_r)$	where $0\leq \eta_k \leq 1$ everywhere, $\eta_k \equiv  0$ in $\mathbb{R}^d\setminus B_{r_{k-1}}$, $\eta_k \equiv 1$ in $B_{r_k}$, and $|\nabla \eta_k|\leq 2^k/r$. Now, let
         \begin{equation*}
         A_k := \int (\eta_kv_k)^2dx	
         \end{equation*}
         we shall see how large $M_0$ ought to be to guarantee that
         \begin{equation*}
              \lim \limits_{k \to +\infty} A_k = 0,	
         \end{equation*}
         in which case we may conclude that $u \leq M_0$ almost everywhere in $B_{r}$.\\

         Applying H\"older's inequality, we get
         \begin{equation*}
              A_k \leq \left (\int (\eta_kv_k)^{2^*}dx \right )^{\tfrac{2}{2^*}}\left | \left \{ \eta_kv_k >0 \right \} \right |^{1-\tfrac{2}{2^*}}.	
         \end{equation*}
         Then thanks to the Sobolev type inequality \eqref{eqn: Sobolev in O_epsilon} and the Energy inequality \eqref{eqn: energy inequality},
         \begin{equation*}
              A_k \leq C \int |\nabla (\eta_kv_k)|^2dx	\left | \left \{ \eta_kv_k >0 \right \} \right |^{\tfrac{2}{d}}\leq C \int |\nabla \eta_k|^2v_k^2dx	\left | \left \{ \eta_kv_k >0 \right \} \right |^{\tfrac{2}{d}}	
         \end{equation*}
         since $|\nabla \eta_k|\leq 2^k/r$, we get
         \begin{equation*}
              A_k \leq C2^{2k}r^{-2} \int_{\{\eta_k \neq 0\}} v_k^2dx	\left | \left \{ \eta_kv_k >0 \right \} \right |^{\tfrac{2}{d}}	
         \end{equation*}
         Also from the definition of $\eta_k$ we have that $\eta_{k-1}\equiv 1$ whenever  $\eta_k \neq 0$. Moreover, $w_k \neq 0$ means that $w>M_k$, and since $M_{k+1}=M_k+M_02^{-(k+1)}$, these two observations imply that $\eta_{k-1}v_{k-1}\geq M_02^{-(k+1)}$ in $\{ \eta_kv_k\neq 0\}$. This gives the measure estimate
         \begin{equation*}
              \left | \left \{ \eta_kv_k >0 \right \} \right | \leq 2^{2(k+1)}M_0^{-2}\int (\eta_{k-1}v_{k-1})^2\;dx 	
         \end{equation*}
         and we finally get
         \begin{equation*}
              A_k \leq C2^{2k}r^{-2}(2^{2(k+1)}M_0^{-2})^{\tfrac{2}{d}}A_{k-1}^{1+\tfrac{2}{d}}\leq Cr^{-2}(4M_0^{-1})^{\tfrac{4}{d}}2^{6k}A_{k-1}^{1+\tfrac{2}{d}}.
         \end{equation*}
         From here, applying Proposition \ref{prop: sequence} with $\alpha=C_Sr^{-2}(4M_0^{-1})^{\tfrac{4}{d}},\beta=6$ and $\delta=\tfrac{2}{d}$ we conclude that for $\lim \limits A_k > 0$ it is necessary that
         \begin{equation*}
              A_0 \leq 2^{-6/\delta^2} \left (C r^{-2}4^{2\delta}M_0^{-2\delta}\right )^{-\tfrac{1}{\delta}}=2^{-6/\delta^2}C^{-\tfrac{1}{\delta}}r^{\tfrac{2}{\delta}}4^2M_0^2.
         \end{equation*}
         In other words,
         \begin{equation*}
              M_0^2 \geq 2^{-6/\delta^2}C^{\frac{1}{\delta}}r^{-\frac{2}{\delta}}4^{-2}A_0.
         \end{equation*}
         This shows that
         \begin{equation*}
              \sup \limits_{\mathcal{B}^\epsilon_{r/2}}v^2 \leq \frac{C}{r^{d}}A_0^{1/2}\leq \frac{C}{r^d}\int_{\mathcal{B}^\epsilon_r} v^2\;dx.
         \end{equation*}
    \end{proof}

    Likewise, if $v$ is a supersolution, we conclude it is bounded below since $-v$ is a subsolution. This shows that weak solutions of \eqref{eqn: Dirichlet problem in  D} are in $L^\infty_{\text{loc}}$. Next, we prove a very weak version of the Harnack inequality.

    \begin{lem}\label{lem: weak Harnack}
         Suppose that $v$ is a supersolution of \eqref{eqn: Dirichlet problem in  D} in $\mathcal{B}^\epsilon_{r}$ and that $v\geq 0$. If $M>0$ is such that
         \begin{equation*}
              |\{ x \in \mathcal{B}^\epsilon_r \mid v(x)>M\}|\geq m|\mathcal{B}^\epsilon_r|,
         \end{equation*}
         then we have
         \begin{equation*}
              \inf \limits_{\mathcal{B}^\epsilon_{c_1r}} v \geq c_0 M.
         \end{equation*}
         where $c_0$ depends only on $\mathcal{O}$ and $m$ and $c_1$ is a  small universal constant.	
    \end{lem}

    \begin{proof}
         We start by deriving an energy inequality similar to the one used in the upper bound. Fix $\gamma>0$ (later we will let it go to zero), taking $\phi=\eta^2(v+\gamma)^{-1}$ as a test function in the supersolution equation for $v$, we obtain
         \begin{align*}
              0\leq & \int \nabla v \cdot \left [  -(v+\gamma)^{-2}\eta^2 \nabla v + (v+\gamma)^{-1}2\eta\nabla \eta \right ]\;dx.
         \end{align*}
         Moving the first term to the left, this becomes
         \begin{equation*}
              \int_{\mathcal{B}^\epsilon_r}\eta^2(v+\gamma)^{-2}|\nabla v|^2\;dx\leq 2\int_{\mathcal{B}^\epsilon_r}\eta (v+\gamma)^{-1}\nabla v\cdot \nabla \eta\;dx.		
         \end{equation*}
         Then, using the Cauchy-Schwartz inequality on the the right hand side and rearranging, we obtain
         \begin{equation*}
              \int_{\mathcal{B}^\epsilon_r}\eta^2| (v+\gamma)^{-1}\nabla (v+\gamma)|^2\;dx\leq 4\int_{\mathcal{B}^\epsilon_r}|\nabla \eta|^2\;dx.	
         \end{equation*}
         We recognize on the left hand side the gradient of $\log(v+\gamma)$, so, taking $\eta$ to be a smooth function which is identically $1$ in $B_{r/2}$ and vanishes outside $B_{r}$ we arrive at the estimate
         \begin{equation*}
              \int_{\mathcal{B}^\epsilon_{r/2}}|\nabla \log(v+\gamma)|^2\;dx\;\leq C_dr^{-2}|\mathcal{B}^\epsilon_r|.	
         \end{equation*}  
         Next, consider the function $\bar v= \max \{-\log(\frac{v+\gamma}{M}),0\}$, and note that
         \begin{equation*}
              \int_{\mathcal{B}^\epsilon_{r/2}}|\nabla \bar v|^2\;dx \leq \int_{\mathcal{B}^\epsilon_{r/2}}|\nabla \log(x+\gamma)|^2\;dx.
         \end{equation*}
         Take $\gamma$ small enough so that  $v+\gamma>M$ in a set of measure greater than $2^{-1}m|\mathcal{B}^\epsilon_r|$, then we can apply Proposition \ref{prop: general Moser estimate} and conclude that
         \begin{equation*}
              \int_{\mathcal{B}^\epsilon_{c_1r/2}}\bar v^2\;dx\leq Cr^2\int_{\mathcal{B}^\epsilon_{r/2}}|\nabla \bar v|^2\;dx\leq C|\mathcal{B}^\epsilon_r|,\;\;\; C=C(d,m).		
         \end{equation*}
         On the other hand, it can be checked that $\bar v$ is a subsolution. Hence, Lemma \ref{lem: De Giorgi pointwise bound} yields that
         \begin{equation*}
              \sup \limits_{\mathcal{B}^\epsilon_{c_1r/4}}\bar v^2 \leq \frac{C}{r^n}C|\mathcal{B}^\epsilon_r|\leq C.	
         \end{equation*}
         Where $C$ is again a universal constant. Using the definition of $\bar v$ this says that
         \begin{equation*}
              -\log[(v+\gamma)/M]\leq CC \Rightarrow (v+\gamma)/M \geq e^{-C}.
         \end{equation*}
         Letting $\gamma \to 0$, we conclude that 
         \begin{equation*}
              \inf \limits_{\mathcal{B}^\epsilon_{c_1r/4}} \;v \geq e^{-C}M,	
         \end{equation*}
         which proves the lemma with $c_0 = e^{-C}$, so that $c_0$ is a universal constant.
    \end{proof}

    \begin{cor}\label{cor: Harnack}
         There are universal constants $C,c_1>0$ such that$\pal$ we have
         \begin{equation*}
              \sup \limits_{\mathcal{B}^\epsilon_{r/2}} v \leq C \inf \limits_{\mathcal{B}^\epsilon_{c_1r}} v
         \end{equation*}
         for any $v$ which is a nonnegative solution of \eqref{eqn: Dirichlet problem in  D} in $\mathcal{B}^\epsilon_{r}$.
    \end{cor}

    \begin{proof}
         Indeed, if $v$ is a solution, it is bounded from above, so let
         \begin{equation*}
              M=\tfrac{1}{2}\sup \limits_{\mathcal{B}^\epsilon_r} v.	
         \end{equation*}
         Applying Lemma \ref{lem: De Giorgi pointwise bound} to $(v-M)_+$,  there exists a constant $C$ such that
         \begin{equation*}
              M^2 r^d \leq C\int_{\mathcal{B}^\epsilon_r}(v-M)_+^2\;dx\leq  CM^2|\{x \in \mathcal{B}^\epsilon_r \mid v>M\}| 
         \end{equation*}
         and we conclude that 
         \begin{equation*}
              |\{ x \in \mathcal{B}^\epsilon_r \mid v \geq M\}| \geq C^{-1}|\mathcal{B}^\epsilon_r|. 	
         \end{equation*}
         This, together with the fact that $w\geq 0$, allows us to apply Lemma \ref{lem: weak Harnack}, which gives the bound
         \begin{align*}
              c_0M  \leq & \;\; \inf \limits_{\mathcal{B}^\epsilon_{c_1r}} v.
         \end{align*}
         Putting this together with the definition of $M$ the desired inequality is obtained.
    \end{proof}

    \begin{cor}\label{cor: oscillation}
         There are universal constants $\mu,c_1 \in (0,1)$ such that if $v$ is a solution in $\mathcal{B}^{\epsilon}_{2r}$, then 
         \begin{equation*}
              \osc \limits_{\mathcal{B}^\epsilon_{c_1r}}v \leq \mu \osc \limits_{\mathcal{B}^\epsilon_r}v\;\;\;\; \mathbb{P}\text{-a.s.}
         \end{equation*}
         In particular, there are $\alpha=\alpha(\mathcal{O})\in (0,1)$ and $C=C(\mathcal{O})>0$ such that
         \begin{equation*}
              \| v \|_{C^\alpha(\mathcal{B}^\epsilon_{c_1r})} \leq \tfrac{C}{r^\alpha}\|v\|_{L^\infty(\mathcal{B}^\epsilon_r)}\;\;\;\; \mathbb{P}\text{-a.s.}	
         \end{equation*}
    \end{cor}

    \begin{proof}
         If $v$ is constant in $\mathcal{B}^\epsilon_{c_1r}$ there is nothing to prove, so let $M>m$ be given by
         \begin{equation*}
              M = \sup \limits_{\mathcal{B}^\epsilon_{c_1r}} v,\;\; m =\inf \limits_{\mathcal{B}^\epsilon_{c_1r}} v.	
         \end{equation*}
         Then $v^*= \frac{v-m}{M-m}$ is harmonic and positive in $B_{r}\cup \mathcal{O}_\epsilon$. If $c_1<1/2$ then Corollary \ref{cor: Harnack} yields
         \begin{equation*}
              \sup \limits_{\mathcal{B}^\epsilon_{c_1r}} v^* \leq \sup \limits_{\mathcal{B}^\epsilon_{r/2}} v^* \leq C\inf \limits_{\mathcal{B}^\epsilon_{c_1r}} v^*.	
         \end{equation*}
         Suppose first that $C \inf \limits_{\mathcal{B}^\epsilon_{c_1r}}v^*<\tfrac{1}{2}$, then (note that $\osc \limits_{\mathcal{B}^\epsilon_r} v^*=1$ by construction) 
         \begin{equation*}
              \osc \limits_{\mathcal{B}^\epsilon_{c_1r}} v^* \leq \sup \limits_{\mathcal{B}^\epsilon_{c_1r}} v^*< \tfrac{1}{2} = \tfrac{1}{2} \osc \limits_{\mathcal{B}^\epsilon_r} v^*.	
         \end{equation*}
         Otherwise, we have $\inf  \limits_{\mathcal{B}^\epsilon_{r/2}}v^* \geq \frac{1}{2}C^{-1}$, in which case
         \begin{equation*}
              \osc \limits_{\mathcal{B}^\epsilon_{c_1}} v^* \leq 1-\inf \limits_{\mathcal{B}^\epsilon_{c_1r}} v^*\leq 1-\tfrac{1}{2}C^{-1} = (1-\tfrac{1}{2}C^{-1}) \osc \limits_{\mathcal{B}^\epsilon_r} v^*	
         \end{equation*}
         In either case, with $\mu = \max \{\tfrac{1}{2},1-\tfrac{1}{2}C^{-1}\} \in (0,1)$ we have
         \begin{equation*}
              \osc \limits_{\mathcal{B}^\epsilon_{r/2}}v^* \leq \mu \osc \limits_{\mathcal{B}^\epsilon_r}v^*	
         \end{equation*}
         this inequality for $v^*$ immediately implies the desired inequality for $w$.
    \end{proof}

    \subsection{Stampacchia's Maximum Principle} The last tool we need in order to prove Theorem \ref{thm: elliptic estimates} is Stampacchia's maximum principle, which deals with $f\neq 0$ and zero Dirichlet boundary conditions.

    \begin{lem}\label{lem: Stampacchia}
         We have$\pal$ that for any $v \in H_0$ solving \eqref{eqn: Hele-Shaw homogeneous} in  $\mathcal{B}^\epsilon_r$ and  $p>d/2$ ($d\geq 3$),
         \begin{equation*}
              \|v\|_{L^\infty(\mathcal{B}^\epsilon_r)} \leq C_{\mathcal{O}}r^{2-\tfrac{d}{p}}\|f\|_{L^p(\mathcal{B}^\epsilon_r)}.
         \end{equation*}	
    \end{lem}

    \begin{proof}
         The proof is a variation on the theme seen in Lemma \ref{lem: De Giorgi pointwise bound}. Fix $M>0$ and define for each $k \in \mathbb{N}$
         \begin{align*}
              M_k & := M(1-2^{-k})\\
              v_k & := (v-M_k)_+\\
              A_k & := |\{v_k>0\}|=|\{v>M_k\}|.
         \end{align*}
         Since
         \begin{equation*}
              |\{ (v-M)_+>0\}|=\lim\limits_{k\to+\infty} A_k	
         \end{equation*}
         we will conclude that $v\leq M$ a.e. in $\mathcal{B}^\epsilon_r$ by showing $A_k$ goes to zero. As before, recall that anywhere where $w_k>0$ we also have $v_{k-1}>2^{-k}M$, so by dividing and multiplying $v_{k-1}$ by $2^{-k}M$ we get the estimate
         \begin{equation}\label{eqn: Stampacchia Sobolev}
              A_k \leq (2^kM^{-1})^{2^*}\int v_{k-1}^{2^*}\;dx\leq 2^{2^*k}M^{-2^*}C_S\|\nabla v_{k-1}\|_2^{2^*}.
         \end{equation}
         Here we used the Sobolev embedding \eqref{eqn: Sobolev in O_epsilon} for the last inequality. Now, since $v_{k-1}$ vanishes outside $\mathcal{B}^\epsilon_r$ we can use the equation and H\"older's inequality to get 
         \begin{equation*}
              \int_{\mathcal{B}^\epsilon_r} |\nabla v_{k-1}|^2\;dx=\int_{\{v_k>0\}}|v_{k}f|\;dx\leq \|v_{k-1}\|_{2^*}\|f\|_{L^{2d/(d+2)}(A_{k-1})}.
         \end{equation*}
         Using \eqref{eqn: Sobolev in O_epsilon} again and dividing by $\|\nabla v_{k-1}\|_2$ on both sides we have
         \begin{equation*}
              \|\nabla v_{k-1}\|_2\leq C_S\|f\|_{2d/(d+2)}.	
         \end{equation*}
         Since the integral for $g$ is over $A_{k-1}$ we can use H\"older's inequality again to get an extra term,
         \begin{equation}\label{eqn: Stampacchia Energy}
              \|\nabla v_{k-1}\|_2\leq C_S\|g\|_2	\leq C_S\|g\|_p |\{v_{k-1}>0\}|^{\tfrac{d+2}{2d}-\tfrac{1}{p}}=C_S\|f\|_pA_{k-1}^{\tfrac{d+2}{2d}-\tfrac{1}{p}}.
         \end{equation}
         Combining \eqref{eqn: Stampacchia Sobolev} with \eqref{eqn: Stampacchia Energy} we prove the relation
         \begin{equation*}
              A_k\leq 2^{2^*k}M^{-2^*}C_S^{2^*}\|g\|_p^{2^*}A_{k-1}^{2^*\left (\tfrac{d+2}{2d}-\tfrac{1}{p}\right )}	
         \end{equation*}
         Then, we can apply Proposition \ref{prop: sequence} to the sequence $A_k$ with
         \begin{align*}
              \delta=2^*\left (\tfrac{d+2}{2d}-\tfrac{1}{p}\right )-1=\tfrac{1}{d-2}(4-\tfrac{2d}{p})>0
         \end{align*}
         to conclude that $\lim\limits_{k\to+\infty}A_k=0$ whenever $M$ satisfies the inequality
         \begin{align*}
              A_0\leq C_{\mathcal{O}}\|f\|_p^{-2^*\delta^{-1}}.	
         \end{align*}
         As in Lemma \eqref{lem: De Giorgi pointwise bound} this shows  $v\leq M$  for some $M$ for which the \emph{reverse} inequality holds, namely
         \begin{align*}
              M_0^{2^*\delta^{-1}}\leq C_{\mathcal{O}}\|f\|_p^{2^*\delta^{-1}}A_0.
         \end{align*}
         Since $\delta/2^*=\tfrac{1}{d-2}(4-\tfrac{2d}{p})\tfrac{d-2}{2d}=\tfrac{1}{d}(2-\tfrac{d}{p})$, we conclude that 
         \begin{align*}
              M\leq C_{\mathcal{O}}\|g\|_pA_0^{\tfrac{1}{d}(2-\tfrac{d}{p})}
         \end{align*}
        Now, $|A_0|\leq |B_1|r^d$, therefore $A_0^{\delta/2^*}\leq C_n r^{d\delta/2^*}=C_dr^{2-\tfrac{p}{d}}$ and the lemma is proved.
    \end{proof}

    Finally, we put all the pieces together for the full elliptic estimate.

    \begin{proof}[Proof of Theorem \ref{thm: elliptic estimates}]
         For any $\rho<r/2$ we seek to bound
         \begin{equation*}
              \osc \limits_{B_\rho(x_0)} v.	
         \end{equation*}
         In order to do this, let us decompose $v$ as
         \begin{equation*}
              v = v_0+v_1
         \end{equation*}
         where $v_0 = v$ on $(\partial B_\rho(x_0)) \cap \mathcal{O}_\epsilon$ and solves \eqref{eqn: Poisson equation} with $f\equiv 0$. Then
         \begin{equation*}
              \osc \limits_{B_\rho(x_0) } v \leq \osc \limits_{B_\rho(x_0)} v_0+\osc \limits_{B_\rho(x_0)} v_1.	
         \end{equation*}
         Applying Lemma \ref{lem: Stampacchia} with $p=+\infty$
         \begin{equation*}
              \osc \limits_{B_\rho(x_0)} v_1 \leq 2 \sup \limits_{B_\rho(x_0)}\|v_1\|	\leq C\rho^2\|g\|_{_\infty}
         \end{equation*}
         and by Corollary \ref{cor: oscillation} and Lemma \ref{lem: De Giorgi pointwise bound},
         \begin{equation*}
              \osc \limits_{B_\rho(x_0)} v_0 \leq C\|v\|_{L^2}\rho^\alpha
         \end{equation*}
         Therefore
         \begin{equation*}
              \osc \limits_{B_\rho(x_0)} v \leq C\left (\|v\|_{L^2}+\rho^{2-\alpha} \|g\|_\infty \right )\rho^\alpha	
         \end{equation*}
         and the first part of the theorem is proved. Now suppose that $v\geq 0$ in $\mathcal{B}^\epsilon_r$, then if $v_0$ and $v_1$ are as before, it is standard to see in this case that $v_0\geq 0$ everywhere in $\mathcal{B}^\epsilon_r$, so by Corollary \ref{cor: Harnack} 
         \begin{equation*}
              \sup \limits_{\mathcal{B}^\epsilon_{r/2}} v_0\leq  C_0 \inf \limits_{\mathcal{B}^\epsilon_{r/2}}	v_0.
         \end{equation*}
         Therefore
         \begin{equation*}
              \sup \limits_{\mathcal{B}^\epsilon_{r/2}} v  \leq  C_0 \inf \limits_{\mathcal{B}^\epsilon_{r/2}}	v_0+\sup \limits_{\mathcal{B}^\epsilon_{r/2}} v_1	
         \end{equation*}
         Moreover, $\inf \limits_{\mathcal{B}^\epsilon_{r/2}}	v_0\leq \inf \limits_{\mathcal{B}^\epsilon_{r/2}}	v+\|v_1\|_{L^\infty(\mathcal{B}^\epsilon_{r/2})}$, and thus
         \begin{equation*}
              \sup \limits_{\mathcal{B}^\epsilon_{r/2}} v\leq C_0 \inf \limits_{\mathcal{B}^\epsilon_{r/2}}	v + 2\|v_1\|_{L^\infty(\mathcal{B}^\epsilon_{r/2})}	
         \end{equation*}
         and using Lemma \ref{lem: Stampacchia} again we finish the proof.
    \end{proof}

\section{Weak free boundary results}\label{sec: weak fb results}

    We are now ready to analyze the behavior of the free boundary $\partialom\{p^\epsilon>0\}$ for the auxiliary problem \eqref{eqn: epsilon obstacle problem} introduced in Section \ref{subsec: strategy}. Observe that if $p^\epsilon$ was uniformly $C^1$ in $\{p^\epsilon>0\}$ and had a non-zero derivative on $\partialom \{p^\epsilon>0\}$ then the implicit function theorem would say that $\partialom \{ p^\epsilon>0\}$ is a $C^1$ hypersurface. The problem is that this is never the case for the obstacle problem, since $p^\epsilon$ vanishes quadratically at its free boundary (see Lemma \ref{lem: quadratic upper bound p^ep}). 

    On the other hand, we shall show $p^\epsilon(\cdot,t)$ vanishes not faster than quadratically (Theorem \ref{thm: nondegeneracy}). More concretely, for each fixed $t$ the supremum of $p^\epsilon(\cdot,t)$ in a ball centered at a free boundary point is comparable to the square of the radius of that ball, and most importantly, this holds uniformly in $\epsilon$.  
 
    Although this is far from a $C^1$ estimate on the free boundary, this non-degeneracy estimate is enough to prevent the free boundary from developing many thin fingers as $\epsilon$ goes to zero. This will guarantee the stability of $\{p^\epsilon>0\}$ as $\epsilon\to 0$, that is, if $p^\epsilon$ converges along some subsequence to a function $p$, then $\{ p^\epsilon >0 \}$ converges to  $\{ p > 0 \}$ along that same subsequence (see Lemma \ref{lem: free sets limit}).

    \subsection{Estimates with respect to $x$} 
    First of all, we remark that solutions to the obstacle problem in $\mathcal{O}_\epsilon$ in the sense of Definition \ref{def: weak solution p^epsilon} are continuous in $x$ for each single $t$. This follows by extending a classical theorem in potential theory to $\mathcal{O}_\epsilon$. The proof does not require uniform estimates in $\epsilon$, but using directly the Harnack inequality and Stampacchia maximum principle from Section \ref{sec: elliptic regularity} one may adapt the proof from the Euclidean case. 
    \begin{lem}\label{lem: continuity of obstacle solutions}(see \cite[Theorem 1]{Ca1998})
    	 Let $\mathcal{O}$ be a deterministic domain satisfying Assumption \ref{hyp: regularity} and $p$ a nonnegative function in $H^1(\mathcal{O})$ such that for some $f \in L^\infty(\mathcal{O})$ we have
              \begin{equation*}
                   -\Delta p = f \text{ in } \{ p>0\},\;\;\partial_n p =  0 \text{ on } \overline{\{ p>0\}} \cap \partial\mathcal{O}
              \end{equation*}
         and $-\Delta p \geq f$ in $\mathcal{O}$, $\partial_n p \geq 0$ on $\partial \mathcal{O}$. Then $p$ is a continuous function in $\overline{\mathcal{O}}$.
    \end{lem}

    From this qualitative continuity lemma, we may use Theorem \ref{thm: elliptic estimates} (again) to show $p^\epsilon$ vanishes quadratically at its free boundary, arguing as in the standard obstacle problem \cite[Theorem 2]{Ca1998}. 

    \begin{lem}\label{lem: quadratic upper bound p^ep}
         Suppose  $p^\e$ is a solution in the sense of Definition \ref{def: weak solution p^epsilon}.	 Then, for any $\epsilon,t>0$ we have
         \begin{equation*}
              -\Delta p^\e = f^\e \mathbb{I}_{\{p^\e>0\}} \text{ in } \mathcal{O}_\epsilon\times \{t\},\;\;\partial_n p^\epsilon = 0\;\;\text{ on } \partial \mathcal{O}_\epsilon \times \{t\}.
         \end{equation*}
         In particular, $p^\epsilon(\cdot,t)$ is H\"older continuous uniformly in $\epsilon$, and for any $x \in \partialom \{p^\epsilon>0\}$ we have
         \begin{equation*}
              \sup \limits_{\mathcal{B}^\epsilon_r(x)\times \{t\}} p^\epsilon 	\leq \; C\|f_\epsilon(\cdot,t)\|_\infty\; r^{2}.
         \end{equation*}
         where $C$ is a universal constant.
    \end{lem}
	
    \begin{proof}
	     Fix $t$ (and we will simply write $p^\epsilon$ instead of $p^\epsilon(\cdot,t)$), and let $x \in \partialom \{p^\epsilon>0\}$ and $r>0$ also be fixed. Recall that $-\Delta p^\epsilon = f^\epsilon$ in the interior of $\{ p^\epsilon >0 \}$ and that in general
	     \begin{equation*}
	          -\Delta p^\epsilon \geq f^\epsilon \; \text{ in } \mathcal{O}_\epsilon,\;\;\partial_n p^\epsilon \geq 0 \;\text{ on } \partial \mathcal{O}_\epsilon.
	     \end{equation*}
	     Since we already have the validity of the equation away $ \partialom \{ p^\epsilon(\cdot,t)>0\}$ we only need to focus on $x$ near $\partialom \{ p^\epsilon(\cdot,t)>0\}$. In particular, we may assume without without loss of generality that $f^\epsilon(\cdot,t)\leq 0$ in $\mathcal{B}^\epsilon_r(x)$, once we have obtained the equation in this case, the bound on $\sup p^\epsilon$ would follow for all other $r$. 
	
	     Then, we decompose $p^\epsilon$ as $p^\epsilon=p^\epsilon_1+p^\epsilon_2$, where 
	     \begin{align*}
              	-\Delta p^\epsilon_1=f^\epsilon & \text{ in } \mathcal{B}^\epsilon_r(x),\;\;\partial_n p^\epsilon_1 = 0 \;\text{ on } \partial \mathcal{O}_\epsilon,\;\; p^\epsilon_1 = p^\epsilon \;  \text{ on } \partialom \mathcal{B}^\epsilon_r(x), \text{ and }\\
              	-\Delta p^\epsilon_2\geq 0 & \text{ in } \mathcal{B}^\epsilon_r(x),\;\;\partial_n p^\epsilon_2 \geq 0\; \text{ on } \partial \mathcal{O}_\epsilon,\;\; p^\epsilon_2 = 0 \;\; \text{ on } \partialom \mathcal{B}^\epsilon_r(x).	     	  
 	     \end{align*}
	     By the maximum principle and the fact that $f^\epsilon\leq 0$ we see that $p^\epsilon_1 \leq p^\epsilon$ in $\mathcal{B}^\epsilon_r$. Moreover, since $p^\epsilon_1 \geq 0$ in $\partial B^\epsilon_r$ we may use Lemma \ref{lem: Stampacchia} to obtain an opposite bound,
	     \begin{equation*}
	         p^\epsilon_1 \geq -C\|f^\epsilon\|_\infty r^2
	     \end{equation*}
	     for some universal constant $C$. In this case, the function $v = p^\epsilon+C\|f^\epsilon\|_\infty r^2$ is nonnegative and solves
         \begin{equation*}
              -\Delta v = f^\epsilon \text{ in } \mathcal{B}^\epsilon_r(x),\;\;\partial_n v = 0 	\text{ on } \partialom \mathcal{B}^\epsilon_r(x).
         \end{equation*}
         Therefore, applying Harnack's inequality (Theorem \ref{thm: elliptic estimates}) we get
         \begin{equation*}
              \sup \limits_{\mathcal{B}^\epsilon_{r/2}(x)}v \leq C\left (v(x)+\|f^\epsilon\|_{\infty}r^2 \right ).	
         \end{equation*}
  	     From Lemma \ref{lem: continuity of obstacle solutions} we know $p^\epsilon$ is continuous and that $x \in \partialom \{p^\epsilon>0\}$, so we have $p^\epsilon_1(x)\leq p^\epsilon(x)=0$. Combining this with $v(x)= p^\epsilon_1(x)+C\|f^\epsilon\|_\infty r^2$ and the Harnack inequality above we obtain
         \begin{equation*}
         	  \sup \limits_{\mathcal{B}^\epsilon_{r/2}(x)}p^\epsilon_1 \leq C\|f^\epsilon\|r^2 
         \end{equation*}
	     for another (still universal) constant $C>0$. Since $p^\epsilon_2 \geq 0$ and it vanishes on $\partialom \mathcal{B}^\epsilon_r$ the maximum principle says that the supremum of $p^\epsilon_2$ must be achieved in the closure of the set $p^\epsilon=0$, and there we have $p^\epsilon_2\equiv-p^\epsilon_1$ so that $p^\epsilon_2 \leq Cr^2$.
		
         In conclusion, for every $x \in \partialom \{p^\epsilon>0\}$ and any $r>0$ we have 
         \begin{equation*}
              0\leq p^\epsilon \leq Cr^2 \text{ in } \mathcal{B}^\epsilon_r(x).
         \end{equation*}       
         In particular, $p^\epsilon$ vanishes quadratically on $\partialom \{p^\epsilon>0\}$, from where it follows (by a standard test function argument) that $\partial_n p^\epsilon = 0$ on $\partial \mathcal{O}_\epsilon \cap \partialom \{p^\epsilon>0\}$ and that $-\Delta p^\epsilon = \mathbb{I}_{\{p^\epsilon>0\} }f^\epsilon$ in $\mathcal{O}_\epsilon$. 

         In this case, the Harnack inequality from Theorem \ref{thm: elliptic estimates} applied to $p^\epsilon\geq 0$ yields
         \begin{equation*}
              \sup \limits_{\mathcal{B}^\epsilon_{r}(x)}\;p^\epsilon \leq C \left (\inf \limits_{\mathcal{B}^\epsilon_{r}(x)}\;p^\epsilon +r^2\|f^\epsilon(\cdot,t)\|_\infty \right ),	
         \end{equation*}
         and since $p^\epsilon(x,t)=0$, the infimum above must be zero, and the lemma follows.
    \end{proof}

    Now we focus on showing that $p^\epsilon(\cdot,t)$ does not vanish faster than quadratically. 
    \begin{thm}\label{thm: nondegeneracy}
         There is a universal $C_0>0$ such that if $x \in \overline{\{p^\epsilon>0\}}$ and $r\leq \tfrac{1}{2}d(x,\db^\epsilon_0)$, then
         \begin{equation*} 
              \sup \limits_{\mathcal{B}^\epsilon_r(x)}p^\epsilon(\cdot,t) \geq  e^{-C_0(t+1)}r^2	
         \end{equation*}	
    \end{thm}

    The argument in \cite{Ca1998}, which is later used in \cite{KM09} fails under the presence of the random domain $\mathcal{O}_\epsilon$, as it relies on the fact that the function $\tfrac{1}{2n}|x|^2$ is a supersolution and thus can be used as a barrier. This is no longer the case for our current problem case since $\tfrac{1}{2n}|x|^2$ does not always have the right Neumann boundary data on $\partial \mathcal{O}_\epsilon$.

    Therefore, we need to come up with a replacement, which one may think of as a corrected paraboloid for the domain $\mathcal{O}_\epsilon$. The  construction of such barrier is the the content of the following key lemma.

    \begin{lem}\label{lem: barrier construction}
         There is a universal constant $a$ such that $\pal$  for any $\epsilon$ and any $x_0 \in \mathcal{O}_\epsilon$ there is a continuous function $v: \mathcal{O}_\epsilon \to \mathbb{R}$, with $v \in H^1_{loc}(\mathcal{O}_\epsilon)$ and such that (in the $H^1$ sense)
         \begin{equation*}
              \left\{  \begin{array}{rll}
              -\Delta v & \geq -1 & \text{ in } \mathcal{O}_\epsilon(\omega)\\
              \partial_n v & \geq 0 & \text{ on } \partial \mathcal{O}_\epsilon(\omega)\\
              v & \geq ar^2	& \text{ in } \mathcal{O}_\epsilon \setminus \mathcal{B}^\epsilon_r(x_0) \text{ for any } r>0 \\
              v & = 0 & \text{ for } x=x_0.
              \end{array}	\right.
         \end{equation*}
    \end{lem}

    \begin{rem} The construction of $v$ is not so straightforward and will take several preliminary lemmas. It is worth remarking that, one cannot simply invoke known homogenization results \cite{PapVar1979,Zhi1993} and use strong $L^2$ convergence for elliptic equations in perforated domains and borrow a barrier from the limiting equation. This approach will not suffice, since later arguments use strongly both that $v=0$ at $x_0$ and that it grows quadratically away from it, and such a pointwise condition is not stable under perturbations.
    \end{rem}

    The first thing we will need is a basic geometric property of $\mathcal{O}_\epsilon$, which follows from Assumption \ref{hyp: regularity}.

    \begin{prop}\label{prop: barrier nodes}
         There are constants $L>l>0$ such that$\pal$ we have
         \begin{equation*}
              \text{ there exists } y_z \in \mathcal{O}_1 \text{ s.t. } B_l(y_z) \subset \mathcal{O}_1 \cap Q_L(z)\;\;\forall\;z\in\mathbb{Z}^d,
         \end{equation*}
         where $Q_L(z)$ denotes the cube $\{ x \in \mathbb{R}^d \mid |x-Lz|_\infty \leq L\}$. 
    \end{prop}

    \begin{figure}[h]
	     \centering
	     \includegraphics[height=0.35\textwidth]{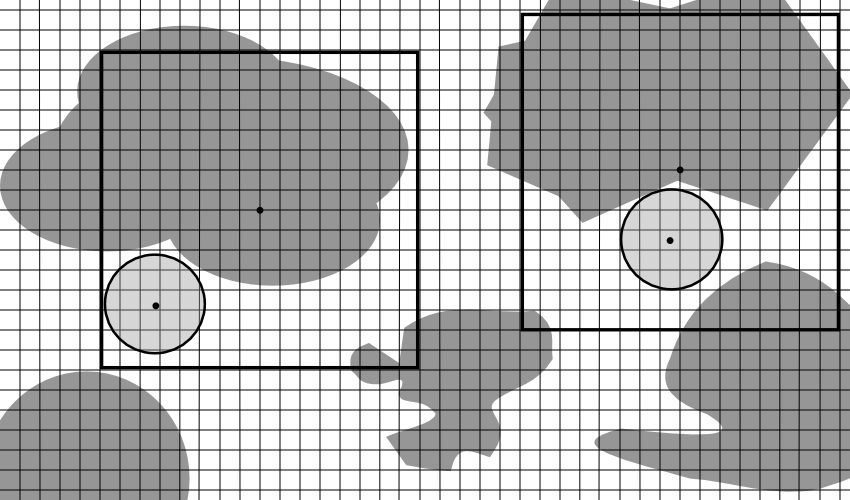}
	     \caption{The geometric configuration considered in Proposition \ref{prop: barrier nodes}.}\label{fig: barrier nodes}
    \end{figure}

    \begin{proof}
         Let $P_k(\omega)$ denote the perforations giving rise to $\mathcal{O}(\omega)$ and  $d_0,\delta,N,M$ the constants related to the regularity of $\mathcal{O}$, as described in Assumption \ref{hyp: regularity} (Section \ref{sec: fbp in stationary ergodic domains}), also let $L$ be the smallest positive integer larger than $2 d_0^{-1}$. For any cube $Q$, we will denote by  $Q^*$ the cube with same center and half the diameter.

         Let us find for each $z\in \mathbb{Z}^d$ Then, for each $z \in \mathbb{Z}^d$ the cube $Q_L(z)$ must be such that $Q_L(z)^*$ lies entirely in $\mathcal{O}(\omega)$. In that case we may take $y_z = Lz$ so that $B_L(y_z) \subset \mathcal{O}(\omega)$. If not, there is some point $x$ and some $k$ such that 
         \begin{equation*}
              x \in (\partial P_k(\omega) ) \cap \hat Q_L(z)	
         \end{equation*}
         In particular, since  $P_k(\omega)$ is a minimally smooth domain, there will be a ball of radius $r_0=r_0(\delta,N,M,d_0)$ in some neighborhood of $x$ which will lie entirely inside $\mathcal{O}(\omega)$, then in this case take $y_z$ as the center of this ball, and let $l=\min (r,L)$, and the proposition is proved.
    \end{proof}

    We shall now use the family of balls $\{ B_l(y_z)\}_{z\in\mathbb{Z}^d}$ from Proposition \ref{prop: barrier nodes}, and a standard Whitney extension \cite{St1970} to construct the barrier far away from $x_0$ (see Figure \ref{fig: barrier profile}).
 
    \begin{lem}\label{lem: barrier construction stage 1}
         There are universal constants $a_0,a_1$, with $a_0 \in (0,1)$ $a_0<a_1$ and $r_0>0$ such that$\pal$  for any $\epsilon$ and any $x_0 \in \mathcal{O}_\epsilon$ there is a continuous function $v: \mathcal{O}_\epsilon \to \mathbb{R}$, $v \in H^1_{loc}(\mathcal{O}_\epsilon)$ and such that
         \begin{equation*}
              \left\{  \begin{array}{rll}
              |\Delta v| & \leq 1 & \text{ in } \mathcal{O}_\epsilon(\omega) \setminus  \mathcal{B}^\epsilon_{r_0\epsilon }(x_0),\\
              \partial_n v & = 0 & \text{ on } \partial \mathcal{O}_\epsilon(\omega),\\
              v & \geq a_0r^2	& \text{ in } \mathcal{O}_\epsilon \setminus \mathcal{B}^\epsilon_r(x_0) \text{ for any } r\geq r_0\epsilon,\\
              v & \leq a_1r^2	& \text{ in } \mathcal{O}_\epsilon \setminus \mathcal{B}^\epsilon_r(x_0) \text{ for any } r\geq r_0\epsilon.
              \end{array}	\right.
         \end{equation*}
         The first two equations are understood in the $H^1$ sense.
    \end{lem}

    \begin{proof}
         We start by noting that it is enough to build such a barrier for $\epsilon=1$ and an arbitrary point $\epsilon^{-1}x_0 \in \mathcal{O}_1$. For if $v$ satisfies all of the above properties  for $\epsilon=1$ and the point $\epsilon^{-1}x_0$, then the function $\epsilon^2v(x/\epsilon)$ satisfies all the properties listed in the statement of the lemma.

	     Let  $l,L$, $\{y_z\}_z$ and $Q_z :=Q_L(z)$ be as in Proposition \ref{prop: barrier nodes}. The non-overlapping cubes $\{Q_z\}_{z\in\mathbb{Z}^d}$ cover all of $\mathbb{R}^d$, and we have the inclusions
     	 \begin{equation*}
     	      B_l(y_z)\subset  \mathcal{O}_1 \cap  Q_z,\;\;\forall\;z\in\mathbb{Z}^d.	
     	 \end{equation*} 
     	 Given $x \in \mathbb{R}^d$ we let $y_{z_x}$ denote the closest point to $x$ among all the $\{y_z\}_z$.  Take $R>0$ and define $I_{R} := \{ z \in \mathbb{Z}^d  \mid B_l(y_z) \subset \mathcal{B}^1_R(x_0),\;\; |Lz-x_0|\geq 1\}$. Then, we introduce the domain
     	 \begin{align*}
     	      D_{R} & = \mathcal{B}^1_R(x_0) \setminus \union \limits_{z \in I_{R}} B_{l/4}(y_z).
     	 \end{align*}
         Note that $D_{R}$ is non-empty for $R$ large enough. We will  build approximate functions $v_R$ in $D_{R}$ with controlled growth in $R$ to extract a (locally) converging subsequence as $R \to +\infty$. For any large enough $R$ we define $v_R$ in $D_{R}$ as the solution to the following obstacle problem
     	 \begin{equation*}
	          \left\{\begin{array}{rll}
	          v_R & \geq 0 & \text{ in } D_R,\\
     	      \Delta v_R & = 1 & \text{ in } \{ v_R>0\},\\
	          \partial_n v_{R} & = 0 & \text{ on } \partial \mathcal{O}_1 \cup \partial B_R(x_0),\\
	          v_R = |\nabla v_R| & = 0 & \text{ on } \partial \{ v_R>0\},\\
	          v_R & = |Lz-x_0|^2	& \text{ on } \partial B_{l/4}(y_z),\;\; z \in I_{R}.
	          \end{array}	\right.
     	 \end{equation*}
         This is understood in a weak sense modifying Definition \ref{def: weak solution p^epsilon} accordingly (the existence of $v_R$ is done following Appendix \ref{sec: construction of the obstacle problem}). Then, arguing as in Lemma \ref{lem: quadratic upper bound p^ep} one can show that $v_R$ satisfies
         \begin{equation*}
	          \left\{\begin{array}{rll}
    	      \Delta v_R & = \mathbb{I}_{\{v_R>0\} } & \text{ in } D_R,\\
	          \partial_n v_{R} & = 0 & \text{ on } \partial \mathcal{O}_1 \cup \partial B_R(x_0).
	          \end{array}	\right.         	
         \end{equation*}
     	 On the other hand, using a barrier argument, one can show that
         \begin{equation*}
         	  v_R > 0 \text{ in } B_{c_0l}(y_z),\;\;\forall\;z\in I_R
         \end{equation*}
         for some small (universal) $c_0>0$. In this case, when  $z \in I_{R}$, since $v_R$ solves an elliptic equation in a smooth domain with smooth coefficients we have that $v_R$ is in $C^2(B_{c_0l/2}(y_z)\setminus B_{c_0l/4}(y_z))$ and with universally bounded second derivatives (that is, bounded independently of $R$)
	    \begin{figure}[h]
		     \centering
		     \includegraphics[height=0.45\textwidth]{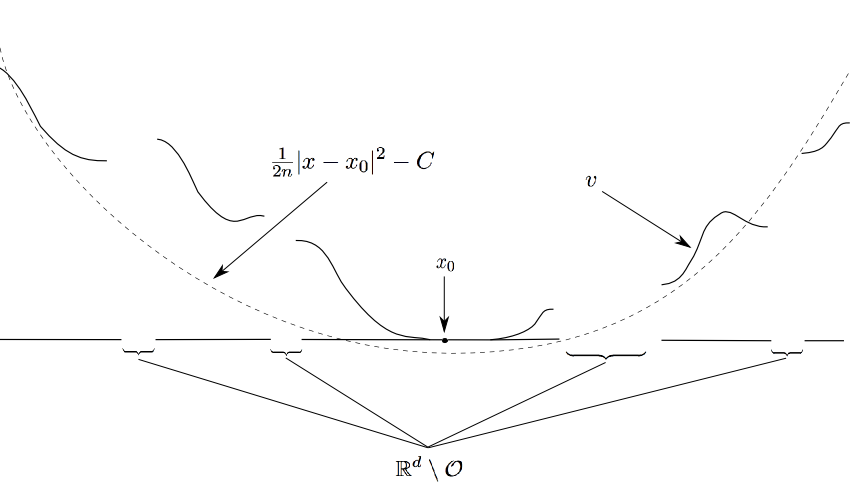}
		     \caption{A schematic description of the barrier $v$ defined in $\mathcal{O}$.}\label{fig: barrier profile}
	    \end{figure}
		Thus, invoking again an extension operator \cite{St1970} we extend $v_R$ it  to a $C^2$ function in all of $B_{l/2}(y_z)$ (for each $z \in I_R$), resulting in a function defined in all of $\mathcal{B}^1_{R}(x_0)$, resembling the one in the figure above. Let us keep calling this extension $v_R$. Then, for some universal $C_1>0$ we have 
     	 \begin{equation*}
     	      \left\{\begin{array}{rll}
     	      |\Delta v_R| & \leq C_1 & \text{ in } \mathcal{B}^1_R(x_0),\\
     	      \partial_n  v_{R} & = 0 & \text{ on } \partial \mathcal{O}_1 \cup \partial B_R(x_0),\\
     	      v_R & = |Lz-x_0|^2	& \text{ somewhere in } B_{l/4}(y_z),\;\;\forall\; z \in I_{R}.
     	      \end{array}\right.	
     	 \end{equation*}
     	 Next, note that $r_1:= \sqrt{n} L$ is such that if $x \in D_R$ and $\mathcal{B}^1_{r_1}(x)\subset D_R$ then we have
     	 \begin{equation*}
              \sup \limits_{\mathcal{B}^1_{r_1}(x)}v_R \geq |Lz_x-x_0|^2,
     	 \end{equation*}
     	 where $z_x \in \mathbb{Z}^d$ is defined as that for which $|Lz-x|$ is the smallest. From here, using the Harnack inequality from Theorem \ref{thm: elliptic estimates} and the definition of $z_x$, we can see that
     	 \begin{equation*}
     		  \inf \limits_{\mathcal{B}^1_{r_1}(x)} v_R \geq C^{-1}(\; \sup \limits_{\mathcal{B}^1_{r_1}(x)}v_R-r_1^2 \; ) \geq C^{-1} ( |Lz_x-x_0|^2-nL^2).
     	 \end{equation*}
     	 Moreover, since $z_x$ was picked so that $|Lz_x-x|\leq L$ we have (by the triangle inequality) that $|Lz_x-x_0|$ is no smaller than $|x-x_0|-L$. In this case,
     	 \begin{equation*}
     	      \inf \limits_{\mathcal{B}^1_{r_1}(x)} v_R  \geq C^{-1}\left ( (|x-x_0|-L)^2 -nL^2 \right )	
     	 \end{equation*}
     	 Thus, if we set $r_0 := 4nL$, then whenever $|x-x_0|\geq r_0$ we have that $(|x-x_0|-L)^2\geq \tfrac{1}{4}|x-x_0|^2$ and $\tfrac{1}{4}|x-x_0|^2-nL^2 \geq \tfrac{1}{16}|x-x_0|^2$. Combining these inequalities with the bound above, we have that
         \begin{equation*}
              v_R(x)  \geq C^{-1}|x-x_0|^2,\;\;\forall\; x\in \mathcal{O}\setminus \mathcal{B}^1_{r_0}(x_0),\;\;|x|\leq R-r_0.
         \end{equation*}
     	 Arguing in a similar manner we can obtain an upper bound on $|v_R|$ for all points in $D_R$, so that
     	 \begin{equation*}
     	      |v_R(x)|\leq C(1+|x-x_0|^2),\;\;\forall\; x \in D_R.
     	 \end{equation*}
     	 for a universal constant $C$. In particular, since this estimate independent of $R$ we have enough compactness to do a Cantor diagonalization argument as $R \to +\infty$. Therefore there is a sequence $R_k\to+\infty$ such that $v_{R_k}$ converges  uniformly on compact subsets of $\mathcal{O}_1$ to some $\tilde v: \mathcal{O}_1 \to\mathbb{R}$, which is continuous and belongs to $H^1_{loc}(\mathcal{O}_1)$. Moreover, for some universal constants $C,C_1,C_2$ all  $>1$  we have
     	 \begin{equation*}
     	      \left\{\begin{array}{rll}
     	      |\Delta \tilde v| & \leq C_1 & \text{ in } \mathcal{O}_1,\\
     	      \partial_n \tilde v  & = 0 & \text{ on } \partial \mathcal{O}_1,\\
     	      \tilde v & \geq C^{-1}|x-x_0|^2 & \text{ in } \mathcal{O} \setminus \mathcal{B}^1_{r_0}(x_0),\\
     	      \tilde v & \leq C_2|x-x_0|^2 & \text{ in } \mathcal{O} \setminus \mathcal{B}^1_{r_0}(x_0).
	          \end{array}\right.	
	     \end{equation*}
	     Setting $a_0:=\max\{1,(C_1C)^{-1}\}$, $a_1:=C_2C^{-1}$ we obtain the lemma with $v:=C_1^{-1}\tilde v$.
	\end{proof}

    The problem of building a barrier now reduces to defining it inside $\mathcal{B}^\epsilon_{\epsilon r_0}(x_0)$. To do this, we first observe there is a small ball around $x_0$ with universal radius where a standard parabola can be used.

    \begin{prop}\label{prop: paraboloid is a supersolution in a small ball} Under Assumption \ref{hyp: regularity} there exists a universal constant $\rho_0$ such that for any $x_0 \in \mathcal{O}$, the function
	     \begin{equation*}
	          v(x)=\tfrac{1}{2n}|x-x_0|^2	
	     \end{equation*}
    	 satisfies $\partial_n v \geq 0 \text{ on } (\partial \mathcal{O}) \cap B_{\rho_0}(x_0)$.
    \end{prop}

    \begin{proof}
    	 We pick $r$ small enough so that for any $x \in \mathcal{O}(\omega)$ the portion of the hypersurface $\partial \mathcal{O}(\omega)$ lying in $B_r(x)$ is a radial Lipschitz graph with respect to $x$. That such an $r$ exists follows from Assumption \ref{hyp: regularity} and its size is determined by the constants in this assumption, in particular, $r$ can be chosen as a universal constant. 

         In this case, we see that for any $x' \in B_r(x) \cup \partial \mathcal{O}(\omega)$ $x'-x$ is pointing towards the exterior of $\mathcal{O}$ at $x'$, that is, $\partial_n v\geq 0$ on $B_r(x) \cap \partial \mathcal{O}(\omega)$.
     \end{proof}

     If we had $\rho_0>r_0$, where $r_0,\rho_0$ are as defined in the previous Lemma and Proposition, we could try to combing the function from Lemma \ref{lem: barrier construction stage 1} with the standard parabola. The question then is whether the resulting function is also a supersolution, that is, satisfies $-\Delta v\geq -1$ in $\mathcal{O}$ and $\partial_n v \geq 0$ on $\partial \mathcal{O}$. However, we expect $\rho_0$ to be much smaller than $r_0$, so first we must extend the parabola from Proposition \ref{prop: paraboloid is a supersolution in a small ball} to a barrier defined a ball of radius much larger than $r_0$.

    \begin{prop}\label{prop: ring extension}
    	 There are universal constants $C>1$, $R_0>r_0$ and $\alpha \in (0,1)$ such that given $\epsilon>0$, $x_0 \in \mathcal{O}$ there is a semicontinuous function $v\in H^1(\mathcal{B}^1_{R_0\epsilon}(x_0))$ such that
         \begin{equation*}
              \left \{ \begin{array}{rlll}
	               -\Delta v & \geq  -1 & \text{ in } \mathcal{B}^1_{R_0\epsilon}(x_0),\\
	               \partial_n v & \geq 0 & \text{ on } \partial \mathcal{O},\\
	               v & \geq \tfrac{\alpha}{2n}|x-x_0|^2 & \text{ in } \mathcal{B}^1_{R_0\epsilon}(x_0),\\
	               v & = 0 & \text{ at } x=x_0,
              \end{array}\right.	
         \end{equation*}
         where $r_0$ is as in Lemma \ref{lem: barrier construction stage 1}. Furthermore, $v$ satisfies the bound
         \begin{equation}\label{eqn: sup and inf barrier bound}
               \inf\; v-\sup\; v \geq C^{-1},
         \end{equation}
         The infimum being over $\mathcal{B}^1_{\epsilon R_0}(x_0)\setminus \mathcal{B}^1_{\epsilon R_0/2}(x_0)$ and the supremum over $\mathcal{B}^1_{\epsilon r_0}(x_0)$.
    \end{prop}

    \begin{proof}
         As in the proof of Lemma \ref{lem: barrier construction stage 1}, by scaling we only need to construct $v$ for $\epsilon=1$. Let $R_0,\alpha>0$ be fixed and to be chosen later, and consider an obstacle problem in $\mathcal{B}^1_{R_0}(x_0)$ with obstacle given by $\tfrac{\alpha}{2n}|x-x_0|^2$ and right hand side given by $-2\alpha$. The solution, which we denote $\tilde v$, satisfies 	
         \begin{equation*}
              \left \{ \begin{array}{rlll}
	               -\Delta \tilde v & \geq  -2\alpha & \text{ in } \mathcal{B}^1_{R_0}(x_0),\\
	               \partial_n \tilde v & \geq 0 & \text{ on } \partial \mathcal{O},\\
	               \tilde v & \geq \tfrac{\alpha}{2n}|x-x_0|^2 & \text{ in } \mathcal{B}^1_{R_0}(x_0).
              \end{array}\right.	
         \end{equation*}
         Applying the theory from Section \ref{sec: elliptic regularity} leads to the bound
         \begin{equation*}
              \sup \limits_{\mathcal{B}^1_{R_0}(x_0)}\tilde v \leq C(\tfrac{\alpha}{2n}R^2_0+2\alpha R^2_0) \leq C\alpha R^2_0.	
         \end{equation*}
         Indeed, this is guaranteed by the Stampacchia maximum principle (Lemma \ref{lem: Stampacchia}) and the fact that $\tilde v$ must agree with the obstacle somewhere in $\mathcal{B}^1_{R_0}(x_0)$. In this case, pick $\alpha=\alpha(C,R_0,\rho_0)$ such that
         \begin{equation*}
              C\alpha R^2_0\leq \tfrac{1}{2n}(\rho_0^2/2).
         \end{equation*}
         Then, it follows that the function   
         \begin{equation*}
              v(x):=\left \{ \begin{array}{rl}
              	   \min\{\tfrac{1}{2n}|x-x_0|^2,\tilde v(x)\} & \text{ in } \mathcal{B}^1_{\rho_0}(x_0),\\
                   \tilde v(x) & \text{ in } \mathcal{B}^1_{R_0}(x_0) \setminus \mathcal{B}^1_{\rho_0/2}(x_0),
              \end{array}\right.
         \end{equation*}
         is well defined. This is due to  $\min\{\tfrac{1}{2n}|x-x_0|^2,\tilde v(x)\}=v(x)$ in $ \mathcal{B}^1_{\rho_0}(x_0)\setminus \mathcal{B}^1_{\rho_0/2}(x_0)$, thanks to
         \begin{equation*}
              \tilde v(x) \leq \tfrac{1}{2n}(\rho_0/2)^2\leq \tfrac{1}{2n}|x-x_0|^2\; \text{ in } \mathcal{B}^1_{\rho_0}(x_0)\setminus \mathcal{B}^1_{\rho_0/2}(x_0).	
         \end{equation*}
         It is also clear that $-\Delta v\geq -1$ in $\mathcal{B}^1_{R_0}(x_0)$ and $\partial_n v\geq 0$ on $(\partial \mathcal{O}) \cap \overline{\mathcal{B}}^1_{R_0}(x_0)$. On the other hand,
         \begin{align*}
              v(x) \geq \tfrac{\alpha}{2n}|x-x_0|^2\;\;\; & \forall\;x\in\mathcal{B}^1_{R_0}(x_0).
         \end{align*}
         Finally, \eqref{eqn: sup and inf barrier bound} can be obtain with an elementary barrier argument.
    \end{proof}

    \begin{proof}[Proof of Lemma \ref{lem: barrier construction}]
    	 We use again the scale invariance to reduce the construction to the case $\epsilon=1$. We are going to ``glue'' the functions built previously by exploiting once again the fact that the minimum of two supersolutions is a supersolution. To this end, let $\beta>0$ to be chosen later, $v_1(x)$ the function constructed in Proposition \ref{prop: ring extension} and $v_2(x)$ the function
         \begin{align*}
              v_2(x) & := \beta\tilde v_2(x)+\sup\limits_{\mathcal{B}^1_{r_0}(x_0)}v_1,
         \end{align*}
         where $\tilde v_2$ is the function constructed in Lemma \ref{lem: barrier construction stage 1}. Observe that $v_1$ is a function defined in $\mathcal{B}^1_{R_0}(x_0)$ and $v_2$ is defined in $\mathcal{O} \setminus \mathcal{B}^1_{r_0}(x_0)$. Then, the barrier $v$ is defined in $\mathcal{O}$ by
         \begin{equation}\label{eqn: gluing the barrier}
              v(x) := \left \{ \begin{array}{ll}
                   v_1(x) & \text{ in } \mathcal{B}^1_{r_0}(x_0),\\
                   \min\{v_1(x),v_2(x)\} & \text{ in } \mathcal{B}^1_{R_0}(x_0)\setminus \mathcal{B}^1_{r_0/2},\\
                   v_2(x) & \text{ in } \mathcal{O} \setminus \mathcal{B}^1_{R_0/2}.
              \end{array}	\right.
         \end{equation}
         Let us see this $v$ is indeed well defined. First, note that
         \begin{equation*}
         	  \min\{v_1,v_2\}  = v_1 \text{ in } \mathcal{B}^1_{r_0}(x_0)\setminus \mathcal{B}^1_{r_0/2}(x_0),
         \end{equation*}
         which is guaranteed by the definition of $v_2$. On the other hand, due to \eqref{eqn: sup and inf barrier bound} there is a $\beta$ determined by $v_1,a_1$ and $R$ (and thus universal) such that
         \begin{equation*}
              \beta a_1 R^2+ \sup\limits_{\mathcal{B}^1_{r_0}(x_0)} v_1 \leq \inf \limits_{\mathcal{B}^1_{R}(x_0)\setminus \mathcal{B}^1_{R/2}(x_0)} v_1.
         \end{equation*}
         This guarantees that
         \begin{align*}
              \min\{v_2,v_1\} & = v_2 \text{ in } \mathcal{B}^1_{R}(x_0)\setminus \mathcal{B}^1_{R/2}(x_0),
         \end{align*}
         proving that $v$ given by \eqref{eqn: gluing the barrier} is a well defined function. Furthermore, as $v_1$ and $v_2$ are supersolutions of the elliptic equation in their domains of definition, the same holds for $v$, thus 
         \begin{equation*}
              -\Delta v \geq -1 \text{ in } \mathcal{O},\;\;\partial_n v\geq 0 \text{ on } \partial \mathcal{O}.	
         \end{equation*}
         Next, note that $v_1$ and $v_2$ satisfy the bound (in their respective domains)
         \begin{equation*}
              v_i(x) \geq a|x-x_0|^2,
         \end{equation*}
         for some small universal constant $a$, thus $v(x)\geq a|x-x_0|^2$ in $\mathcal{O}$. Finally, since $v_1(x_0)=0$ and $v_1=v$ near $x_0$ it follows that $v(x_0)=0$. This constructs the desired barrier when $\epsilon=1$, finishing the proof.
    \end{proof}

    With the barrier at our disposal, we can now proceed to prove the non-degeneracy estimate.

    \begin{proof}[Proof of Theorem \ref{thm: nondegeneracy}]
         We carry out the proof of the non-degeneracy estimate in two steps. Note that since the estimate will be uniform for $x$ in the support of $p^\epsilon$, we only need to prove it when
         \begin{equation*}
              x \in \{ p^\epsilon(\cdot,t)>0\}	
         \end{equation*}

         \emph{Step 1.} Suppose first that $(x,t)$ and $r$ are chosen such that
         \begin{equation}\label{eqn: nondeg}
              f^\epsilon(\cdot,t) \geq \tfrac{1}{2} \text{ in } \mathcal{B}^\epsilon_{r}(x).
         \end{equation}
         Let $v$ be the barrier given by Lemma \ref{lem: barrier construction} for the point $x$ and $\epsilon>0$. The function $w(\cdot,t)=p^\epsilon(\cdot,t)-\tfrac{1}{2}v(\cdot)$ is strictly positive at $x$, which is guaranteed by the fact that $p^\epsilon(x,t)>0$ and that $v(x)= 0$ by construction.

         Then, for any $r>0$ and with $D_r:= \mathcal{B}^\epsilon_{r}(x) \cap \{p^\epsilon>0\}$ we have 
         \begin{equation*}
              \sup\limits_{D_r} w > 0
         \end{equation*}
         Moreover, since $\Delta v \leq 1$ in $\mathcal{O}_\epsilon$ and $\partial_n v \geq 0$ on $\partial \mathcal{O}_\epsilon$,
         \begin{equation*}
              \Delta w \geq 0	\text{ in } \mathcal{B}^\epsilon_r(x),\;\;\partial_n w \leq 0 \text{ on } \partial \mathcal{O}_\epsilon. 
         \end{equation*}
         As a consequence of the maximum principle, we obtain
         \begin{equation*}
              \sup \limits_{\partialom D_r} w = \sup\limits_{D_r} \; w > 0.
         \end{equation*}
         Since $v\geq 0$ everywhere, the supremum on the left cannot be achieved on the portion of $\partial D_r$ lying within $\partialom \{ p^\epsilon(\cdot,t)>0\}$, since there we have $w=0-v\leq 0$. We conclude that the supremum over $\partialom D_r$ must be achieved at a point in $\partialom \mathcal{B}^\epsilon_r(x)$. In other words, 
         \begin{equation*}
              \sup \limits_{\partialom \mathcal{B}^\epsilon_r(x)} \left [p^\epsilon(\cdot,t)-\tfrac{1}{2}v \right ]>0.
         \end{equation*}
         Then, using that $v(x) \geq ar^2$ outside $\mathcal{B}^\epsilon_r(x)$, we get $\sup \limits_{\mathcal{B}^\epsilon_r(x)} p^\epsilon(\cdot,t) \geq \tfrac{1}{2}ar^2$.

         We have proved the following: for any $x,t$ and $r$ for which \eqref{eqn: nondeg} holds we have the estimate
         \begin{equation}\label{eqn: conditional nondeg}
              \sup \limits_{\mathcal{B}^\epsilon_r(x)} p^\epsilon(\cdot,t) \geq \tfrac{1}{2}ar^2.	         	
         \end{equation} 

         \emph{Step 2.} We are left to prove the theorem without assuming \eqref{eqn: nondeg}. For any $k \in \mathbb{N}$ write $t_k = \tfrac{1}{2}k$ and also $t_0=0$. Arguing by induction, we will show for any $k \in \mathbb{N}$ that if $t \in (t_{k-1},t_k)$ and $x,r$ are such that $r\leq d(x,\db^\epsilon_0)$, then
         \begin{equation}\label{eqn: nondeg induction}
              \sup \limits_{\mathcal{B}^\epsilon_r(x)} p^\epsilon(\cdot,t) \geq a4^{-(k+2)}r^2	         		
         \end{equation}
         The theorem follows from this since for  $C_0=\log_2(a)+10$ we have $e^{-C_0(t+1)}\leq a4^{-(k+2)}$ $\;\forall\;k,\; t \geq t_k$.

	     To prove this first for $k=1$, note that \eqref{eqn: nondeg} holds for any combination of $x,t,r$ for which $t \leq \tfrac{1}{2}$ and $r\leq d(x,\db^\epsilon_0)$, since in this case we have  $\mathcal{B}^\epsilon_{r}(x) \subset (\db^\epsilon_0)^c$, so
         \begin{equation*}
              f^\epsilon(\cdot,t) = \mathbb{I}_{(\db_0^{\epsilon})^c}-\int_0^{t}\mathbb{I}_{\{p^\epsilon>0\}}\;ds\geq \tfrac{1}{2},\;\;  \text{ in } \mathcal{B}^\epsilon_{r}(x).        	
         \end{equation*}
         Then applying the estimate from Step 1, we prove the base case of the induction argument. Next, suppose we have proved \eqref{eqn: nondeg induction} up to $k-1$. Then take any $t \in (t_{k-1},t_k)$ and consider the function $h(x):=p^\e(x,t)-p^\e(x,t_{k-1})$. Recall that $p^\epsilon(x,t)$ is increasing in $t$, therefore $h\geq 0$, moreover, $p^\e(\cdot,t)\geq h$.

         Assume first that $\mathcal{B}^\epsilon_{r/2}(x)$ lies entirely outside $\{ p^\epsilon(\cdot,t_{k-1})>0\}$, then $h$ solves
         \begin{equation*}
              \Delta h = 1=\int_{t_{k-1}}^t	\mathbb{I}_{\{p^\epsilon(\cdot,s)>0\}} ds \geq \tfrac{1}{2} \;\text{ in } \tilde D_r,
         \end{equation*}
         where $\tilde D_r :=  \left (\{p^\epsilon(\cdot,t)>0\} \setminus \{p^\epsilon(\cdot,t_{k-1})>0\}\right )\cap \mathcal{B}^\epsilon_{r/2}(x)$. Then taking $v$ as the barrier from Lemma \ref{lem: barrier construction} for the point $x$, we have again that $w=h-\tfrac{1}{2}v$ has a positive supremum in $\tilde D_r$, and that by the maximum principle it must be achieved on $\partial \tilde D_r$, and arguing again as in Step 1 we conclude
         \begin{equation*}
              \sup \limits_{\mathcal{B}^\epsilon_{r}(x)}\;h \geq \sup \limits_{\mathcal{B}^\epsilon_{r/2}(x)}\;h \geq \tfrac{1}{2}a(r/2)^2= \tfrac{1}{8}ar^2\geq a4^{-(k+2)}r^2. 	
         \end{equation*}
         and thus we get the same lower bound for the supremum of $p^\epsilon(\cdot,t)$ over $\mathcal{B}^\epsilon_{r}(x)$. Now, if it was not the case that $\mathcal{B}^\epsilon_{r/2}(x) \subset \{ p^\epsilon(\cdot,t_{k-1})>0\} $, we would know there is a point $x_0 \in \overline{\{ p^\epsilon(\cdot,t_{k-1})>0\}}$ such that
         \begin{equation*}
              \mathcal{B}^\epsilon_{r/2}(x_0) \subset \mathcal{B}^\epsilon_r(x)	
         \end{equation*}
         and by the inductive hypothesis we would have
         \begin{equation*}
              \sup \limits_{\mathcal{B}^\epsilon_{r/2}(x_0)} p^\epsilon(\cdot,t_{k-1}) \geq  a4^{-((k-1)+2)}(r/2)^2= a4^{-(k+2)}r^2.	
         \end{equation*}
         Since  $p^\epsilon(x,t)$ is nondecreasing in $t$ we get $\sup \limits_{\mathcal{B}^\epsilon_{r}(x)} p^\epsilon(\cdot,t_{k-1}) \geq \sup \limits_{\mathcal{B}^\epsilon_{r/2}(x_0)} p^\epsilon(\cdot,t_{k-1})$, so we arrive at
         \begin{equation*}
              \sup \limits_{\mathcal{B}^\epsilon_{r}(x)} p^\epsilon(\cdot,t_{k-1}) \geq a4^{-(k+2)}r^2.	
         \end{equation*}
         With this the induction argument is finished and the theorem is proved.
    \end{proof}

    \begin{rem}
         Observe that the non-degeneracy coefficient in Theorem \ref{thm: nondegeneracy} decays exponentially with time. While this is likely not an optimal estimate, it is enough for our current purposes.
    \end{rem}

    \subsection{Estimates with respect to time} Let us now discuss the regularity of $p^\e$ in the time variable. Since  $f^\epsilon(x,t)$ is non-decreasing in time it follows that $p^\e$  is non-decreasing as well.  Moreover, in \cite{KM09}, it was straightforward from the formulation of the problem that there is some universal constant $C$ independent of $\epsilon$ such that
    \begin{equation}\label{eqn: control_time}
         \sup \limits_{t_1\neq t_2}\frac{|p^\e(x,t_1)-p^\e(x,t_2)|}{|t_1-t_2|} \leq C
    \end{equation}
    This estimate still holds in our setting, despite the appearance of the extra term $\int_0^t\mathbb{I}_{\{p^\epsilon(\cdot,s)>0\}} \;ds$.

    \begin{prop}\label{prop: control_time}
         Almost surely there is a $C>0$ such that \eqref{eqn: control_time} holds for all $0 < \epsilon < 1$.		
    \end{prop}

    \begin{proof}
         Given $t_1$ and $t_2$ note that $p^\epsilon(.,t_i)$ each solve obstacle problems in $\mathcal{O}_\epsilon$ with zero data at infinity, zero obstacle and right hand sides given by $f^\epsilon(\cdot,t_i)$. Now, note that
         \begin{equation*}
              \left | f^\epsilon(x,t_1)-f^\epsilon(x,t_2)\right | = \left | \int_0^{t_1} \mathbb{I}_{\{p^\epsilon>0\}}\;ds-\int_0^{t_2}\mathbb{I}_{\{p^\epsilon>0\}}\;ds\right | \leq |t-s|,	
         \end{equation*}	
         then, using a standard barrier argument,  \eqref{eqn: control_time} follows as done for instance in \cite{B2001}.
    \end{proof}

    We can combine this proposition with Lemma 4.3 to yield the following estimate in time.
    \begin{lem}\label{lem: control_FB_time}
         The positivity set of $p^\epsilon$, that is $\db_t(p^\e) := \{ x \mid p^\epsilon(x,t)>0\}$ is evolving continuously in time. Namely, $\mathbb{P}\text{-a.s.}$ there is a universal $C>0$ such that if $s<t$ and $\epsilon \in (0,1)$ then 	
         \begin{equation*}
              \db_s(p^\e) \subset \db_t(p^\e) \subset \db_s(p^\e) + B_{C\sqrt{t-s}}.
         \end{equation*}
    \end{lem}

    \begin{proof}
         The first inclusion is due to the fact that $p^\e$ is non-decreasing in the time variable. To show the second inclusion, suppose $x\in \Gamma_t(p^\e)$ is such that $B_\delta(x) \cap \db_s(p^\e) = \emptyset$.  It follows that
         \begin{equation*}
              p^\epsilon(y,s) = 0\;\;\;\text{ for all } y\in \mathcal{B}^\epsilon_\delta(x)	
         \end{equation*}
         Then $|p^\epsilon(y,t)|	\leq C(t-s)$ for all $y \in \mathcal{B}^\epsilon_\delta(x)$, since $x \in \Gamma_t(p^\e)$ Lemma \ref{thm: nondegeneracy} says that $p^\epsilon(y,t)\geq C\delta^2$ for some $y \in \mathcal{B}^\epsilon_\delta(x)$ which gives $\delta^2\leq C(t-s)$ and proves the lemma.
    \end{proof}

    As a corollary, we get continuity of the droplet for $t=0$, which will be useful in Section \ref{sec: homogenization}.
    \begin{cor}\label{cor: distance from the initial phase}
         $\mathbb{P}\text{-a.s.}$ there is a $C$ such that $ \db_0^\e \subset \db_t(p^\epsilon) \subset \db_0^\e + B_{C\sqrt{t}}$ for any $t$ and $\epsilon\in (0,1)$.
    \end{cor}

\section{Convergence of the obstacle problems}\label{sec: convergence obstacle problem}

    Due to the theory in Section \ref{sec: elliptic regularity}, the functions $p^\epsilon$ are uniformly continuous in compact subsets of $\mathcal{O}_\epsilon \times \mathbb{R}$ and bounded in $H^1_{\text{loc}}(\mathcal{O}_\epsilon)$. Thus,$\pal$ from any sequence $\epsilon_k \to 0$ we may extract a uniformly converging subsequence. In this section we identify the limit of this sequence as the unique solution of a homogeneous free boundary problem. 

    \begin{thm}\label{thm: convergence of obstacle problems}
	     Let $p:\mathbb{R}^d\times \mathbb{R}_+\to\mathbb{R}$ be the (unique) weak solution to
         \begin{equation}\label{eqn: homogeneous obstacle evolution problem}	
              \left \{ \begin{array}{rrll}
	          -\text{div}(A \nabla  p) & = & \mu\left ( -\mathbb{I}_{\db_0^c}+\int_0^t \mathbb{I}_{\{ p(\cdot,s)>0\}}\;ds\right) & \text{ in }  \{ p>0\} \\\
	           p=|\nabla   p| & = & 0 & \text{ on } \partial \{ p >0 \}\\
              \end{array}\right.
         \end{equation}
         Then, if $\{p^\epsilon\}_\epsilon$ are weak solutions of \eqref{eqn: epsilon obstacle problem} with initial data $\db_0^\epsilon$, we have $p^\epsilon \to p\pal$ in $L^p$ norm for every $p \in [1,\infty]$ (see Definition \ref{def: notions of convergence}). Here $A$ is the effective diffusivity of $\mathcal{O}$ (see Section \ref{sec: linear homogenization revisited}, Definition \ref{def: A_0}) and $\mu \in [0,1]$ is the volume density of $\mathcal{O}$ given in Definition \ref{def: mu}.
    \end{thm}

    \begin{rem} The notion of weak solution for \eqref{eqn: homogeneous obstacle evolution problem} is the same as that given for \eqref{eqn: epsilon obstacle problem} in Definition \ref{def: weak solution p^epsilon} when $\mathcal{O}=\mathbb{R}^d$. The existence theory for \eqref{eqn: homogeneous obstacle evolution problem} above problem is discussed in Appendix \ref{sec: construction of the obstacle problem}. Uniqueness for this problem will later follow (Corollary \ref{cor: uniqueness for aux obstacle problems})  by combining the uniqueness results for the Hele-Shaw problem in $\mathbb{R}^d$ (see \cite{K03}) together with the theory developed first in \cite{KM09} and extended in Theorem \ref{thm: evolution} which is independent of this section.
    \end{rem}

	First, let us carry out the extraction of a converging subsequence in detail.

    \begin{prop}\label{prop: uniform convergence of p^epsilon}
    	 Consider weak solutions $p^\epsilon$ as in Definition \ref{def: weak solution p^epsilon} with some fixed initial data $\db_0^\epsilon$. Then, given a sequence $\epsilon_k>0$, $\epsilon_k \to 0$ we have that for $\mathbb{P}$-almost every $\omega$ there is a subsequence $\epsilon_k'$ and a continuous function $p:\mathbb{R}^d \times \mathbb{R}_+ \to \mathbb{R}$ with compact support such that for any $x \in \mathcal{O}_\epsilon(\omega)$, and $R,T>0$ we have
         \begin{equation*}
              \lim \limits_{k \to \infty}\sup \limits_{\mathcal{B}^\epsilon_R(x)\times [0,T]} |p^{\epsilon_k'}(x,t,\omega)-p(x,t)|=0.	
         \end{equation*}
    \end{prop}

    \begin{proof}
    	 For each $\epsilon$ and $\mathbb{P}$-almost every $\omega$ let $\tilde p^\epsilon: \mathbb{R}^d\times \mathbb{R}_+\to\mathbb{R}$ be a H\"older continuous function which agrees with $p^\epsilon(\omega)$ in $\mathcal{O}_\epsilon(\omega)$. That such a $\tilde p^\epsilon$ exists$\pal $ is a consequence of Lemma \ref{lem: quadratic upper bound p^ep} and Proposition \ref{prop: control_time}, which also tell us that the H\"older norm of $p^\epsilon$, and thus that of $\tilde p^\epsilon$, is bounded uniformly for $\epsilon\in (0,1)$.

         Then, given any sequence $\epsilon_k \to 0$ we only need to apply the Arzela-Ascoli Theorem and a Cantor diagonalization argument to get a H\"older continuous function $p(x,t)$ such that $\tilde p^{\epsilon_k} \to p$ uniformly in compact subsets of $\mathbb{R}^d\times \mathbb{R}_+$, that is
         \begin{equation*}
              \lim \limits_{k \to \infty}\sup \limits_{B_R(x)\times [0,T]} |\tilde p^{\epsilon_k'}(x,t,\omega)-p(x,t)|=0	
         \end{equation*}
         for any $R,T>0$. Since $\tilde p^{\epsilon_k} \equiv p^{\epsilon_k}$ in $\mathcal{B}^\epsilon_R(x)\times [0,T] \subset B_R(x)\times [0,T]$ by construction, we are done.
    \end{proof}

    As outlined in Section \ref{subsec: strategy}, we will use the free boundary estimates from Section \ref{sec: weak fb results} to show that
    \begin{equation*}
         f^\epsilon \mathbb{I}_{\mathcal{O}_\epsilon} \rightharpoonup \mu f\;\;\text{ in } L^2(\mathbb{R}^d),	
    \end{equation*}
    where $f(x,t)$ is given in terms of the putative limit $p(x,t)$ by 
    \begin{equation*}
         f(\cdot,t) = -\mathbb{I}_{\db_0^c}+\int_0^t\mathbb{I}_{\{p(\cdot,s)>0\} }\;ds.	
    \end{equation*}
    From here, stochastic homogenization theory (see Section \ref{sec: linear homogenization revisited}) will imply Theorem \ref{thm: convergence of obstacle problems}. The weak convergence stated above will be guaranteed by the next lemma. For convenience we let $D_t(p^\epsilon):= \{ x \mid  p^\epsilon(x,t)>0\}$ as in Section \ref{sec: weak fb results}.


   \begin{lem}\label{lem: free sets limit}
    	 The following holds $\pal$ in $\omega$: fix $t>0$, and consider a sequence $\epsilon_k \to 0$ such that the functions $p^{\epsilon_k}(\cdot,\omega,t)$ converge uniformly to a function $p(x)$. Then for every $\delta>0$ there exists $k$ such that whenever $k>k_0$ we have that $\db_t(p^{\epsilon_k})$ lies entirely in a $\delta$ neighborhood of $\{ p>0\} \cap \mathcal{O}_{\epsilon_k}$ and that $\{ p>0\} \cap \mathcal{O}_{\epsilon_k}$ lies in a $\delta$ neighborhood of $\db_t(p^{\epsilon_k})$.
    \end{lem}

    \begin{proof}
         For brevity we will omit $\omega$ throughout the proof. For each $\epsilon$ we define
         \begin{equation*}
              \rho_\epsilon := \max \limits_{x \in \overline{\db}_t(p^\epsilon)}\; \min_{y \in \overline{\db}_t(p)\cap\mathcal{O}_\epsilon}|x-y|,
         \end{equation*}	     
         all we want to show is that $\rho_{\epsilon_k} \to 0$ as $k\to+\infty$, let us omit the $k$ for now and simply write $\epsilon$. Note that the maximum and minimum above are always achieved, and we may assume without loss of generality that $\rho_{\epsilon}>0$. Therefore, we may pick a point $x_\epsilon \in \overline{\db}_t(p^\epsilon)$ such that
         \begin{equation*}
         	  p(\cdot,t) \equiv 0 \text{ in } \mathcal{B}^\epsilon_{\rho_\epsilon}(x_\epsilon). 
         \end{equation*}
         By Theorem \ref{thm: nondegeneracy}, we know that as long as $\rho_\epsilon\leq \tfrac{1}{2}d(x_\epsilon,\db_0^\epsilon)$, there is a point $\tilde x_\epsilon \in \mathcal{B}^\epsilon_{\rho_\epsilon}(x_\epsilon)$ such that 
         \begin{equation*}
              p^\epsilon(\tilde x_\epsilon,t)\geq e^{-C_0(t+1)}\rho_\epsilon^2.	
         \end{equation*}
         As a parenthesis, we note that the argument when $\tfrac{1}{2}d(x_\epsilon,\db_0^\epsilon)\leq \rho_\epsilon$ is straightforward, since $\db_0^\epsilon \subset \db_t(p^\epsilon)$ for all $t>0$, so we will ignore this case. Now, by the H\"older estimate from Theorem \ref{thm: elliptic estimates}, we have 
         \begin{equation*}
              p^\epsilon(\cdot,t)\geq e^{-C_0(t+1)}\rho_\epsilon^2-Cr^\alpha \text{ in } \mathcal{B}^\epsilon_{r}(\tilde x_\epsilon).
         \end{equation*}
         Taking $r=\left [\tfrac{1}{2C}e^{-C_0(t+1)}\rho_\epsilon^2\right ]^{1/\alpha}$ we obtain
         \begin{equation*}
              p^\epsilon(\cdot,t)\geq \tfrac{1}{2}e^{-C_0(t+1)}\rho_\epsilon^2 \text{ in } \mathcal{B}^\epsilon_{r}(\tilde x_\epsilon).
         \end{equation*} 
         Since $p \equiv 0$ in $\mathcal{B}^\epsilon_{\rho_\epsilon}(x_\epsilon)$ this means that
         \begin{equation*}
              \tfrac{1}{2}e^{-C_0(t+1)}\rho_\epsilon^2 \leq \|p^\epsilon(\cdot,t)-p\|_{L^\infty(\mathcal{O}_\epsilon)}.
         \end{equation*}
         By assumption the quantity on the right is going to zero along $\epsilon_k$ as $\epsilon_k \to 0$, so the inequality above says that $\rho_{\epsilon_k} \to 0$ as well. A similar argument handles the the quantity
         \begin{equation*}
              \rho^\epsilon := \max \limits_{x \in \overline{\db}_t(p)\cap\mathcal{O}_\epsilon}\; \min_{y \in \overline{\db}_t(p^\epsilon)}|x-y|,
         \end{equation*}
         so one can also show that $\rho^{\epsilon_k} \to 0$ as $\epsilon_k \to 0$. We only need to note that $p$ inherits the non-degeneracy property of Theorem \ref{thm: nondegeneracy} from $p^\epsilon$ due to the uniform convergence, so we can exchange the roles of $p^{\epsilon_k}$ and $p$ in the previous argument.

         This proves the first part of the lemma: all we need to do for any given $\delta$ is to take  $k_0$ be such that $\max\{ \rho_{\epsilon_k},\rho^{\epsilon_k}\}<\delta$ for all $k>k_0$. To deal with the convergence of the integral, note that $\{p(\cdot,t)=0\}$ is a measurable set and thus $\partial\{p(\cdot,t)=0\}$ has Lebesgue measure zero. Combining this with the fact that the domains $\db_t(p^{\epsilon_k})$ are uniformly bounded sets (using that $\db_0$ is bounded and Corollary \ref{cor: distance from the initial phase}), we conclude that the volume of a neighborhood of $\partial \{ p(\cdot,t)=0\}$ can be made arbitrarily small. The first part of the Lemma implies that for any $\delta$ we can pick $k_0$ such that if $N_\delta$ denotes a $\delta$-neighborhood of $\partial \{ p(\cdot,t)=0\}$ then as long as $k>k_0$ we have 
         \begin{equation*}
              \int_{\mathcal{O}_{\epsilon_k}(\omega)}|\mathbb{I}_{_{\{p^{\epsilon_k}(\cdot,\omega,t)>0\}}}-\mathbb{I}_{_{\{p(\cdot)>0\}}}|\;dx =              \int_{\mathcal{O}_{\epsilon_k}(\omega) \cap N_\delta }|\mathbb{I}_{_{\{p^{\epsilon_k}(\cdot,\omega,t)>0\}}}-\mathbb{I}_{_{\{p(\cdot)>0\}}}|\;dx.
         \end{equation*}
         The quantity on the right is bounded from above by the volume of $N_\delta$, which can be made arbitrary small, and we conclude that the limit of the integral is zero, as we wanted.
    \end{proof}	

    \begin{cor}\label{cor: L^1 convergence of the phases}
    	 Under the same assumptions as the previous Lemma, we have $\mathbb{P}$-a.s. in $\omega$ that
         \begin{equation*}
              \lim \limits_{k \to \infty}\int_{\mathcal{O}_{\epsilon_k}(\omega)}|\mathbb{I}_{_{\{p^{\epsilon_k}(\cdot,\omega,t)>0\}}}-\mathbb{I}_{_{\{p(\cdot)>0\}}}|\;dx = 0.
         \end{equation*}
    \end{cor}

    Integrating the previous convergence estimate with respect to time, we conclude that the functions $f^{\epsilon_k}$ converge in a strong enough norm to what we want.

    \begin{cor}\label{cor: tau limit}
         Let $t,\omega,\epsilon_k$, $p^{\epsilon_k}$ and $p$ be as in Lemma \ref{lem: free sets limit}. Then, for every $t\geq 0$ fixed
         \begin{equation*}
              f^{\epsilon_k}(x,\omega,t) \to f(x,t)=-\mathbb{I}_{\db_0^c}+\int_0^t\mathbb{I}_{\{p(\cdot,s)>0 \}}\;ds.	
         \end{equation*}
         This convergence takes place in the $L^2$ norm, meaning that
         \begin{equation*}
              \lim \limits_{k \to \infty}  \|f^{\epsilon_k}(x,\omega,t)-f(x,t)\|_{L^2(\mathcal{O}_{\epsilon_k})} = 0. 
         \end{equation*}
    \end{cor}

    \begin{proof}
         Since the convergence is uniform in space for every fixed $t$, Lemma \ref{lem: free sets limit} and Corollary \ref{cor: L^1 convergence of the phases} say that
         \begin{equation}\label{eqn: tau limit 1}
              \lim\limits_{k \to \infty} \int_{\mathcal{O}_{\epsilon_k}} |\mathbb{I}_{_{\{p^{\epsilon_k}>0\}}}(x,t)-\mathbb{I}_{_{\{p>0\}}}(x,t)|^2\;dx=0\;\;\;\forall\;t\in (0,T).
         \end{equation}	
         Integrating in $t$, we obtain the following estimate
         \begin{equation*}
         \begin{array}{lll}
              \int_{\mathcal{O}_{\epsilon_k}}\left |f^{\epsilon_k}(x,t)-f(x,t) \right |^2\;dx	&\leq & \int_{\mathcal{O}_{\epsilon_k}}\left (\int_0^t\mathbb{I}_{_{\{p^{\epsilon_k}>0\}}}(x,s)-\mathbb{I}_{_{\{p>0\}}}(x,s) \;ds\right )^2\;dx\\ \\
              &\leq& \int_{\mathcal{O}_{\epsilon_k}}t\int_0^t|\mathbb{I}_{_{\{p^{\epsilon_k}>0\}}}(x,s)-\mathbb{I}_{_{\{p>0\}}}(x,s) |^2\;dsdx\\ \\
               &\leq& T\int_0^T \int_{\mathcal{O}_{\epsilon_k}}|\mathbb{I}_{_{\{p^{\epsilon_k}>0\}}}(x,s)-\mathbb{I}_{_{\{p>0\}}}(x,s) |^2\;dx\;ds.
              \end{array}
         \end{equation*}
         Then from \eqref{eqn: tau limit 1} and the dominated convergence theorem the corollary follows.
    \end{proof}

    With these lemmas at hand, we may now prove the convergence theorem.

    \begin{proof}[Proof of Theorem \ref{thm: convergence of obstacle problems}]
	
         First of all, for any vanishing sequence of $\epsilon$'s and$\pal$ we can find  a function $p$ and a subsequence $\epsilon_k$ such that  $p^{\epsilon_k}(\cdot,\omega,\cdot)$ converges to $p$ by applying Proposition \ref{prop: uniform convergence of p^epsilon}. If we show this $p$ is the unique weak solution to \eqref{eqn: homogeneous obstacle evolution problem} then this limit would be independent of the original vanishing sequence of $\epsilon$'s and $\omega \in \Omega$, meaning $p^\epsilon$ converges to $p$ and the theorem would be proven.

         Recalling that $f^{\epsilon_k}(x,t,\omega) = -\mathbb{I}_{(\db_0^\epsilon)^c}+\int_0^t\mathbb{I}_{\{p^{\epsilon_k}(\cdot,s)>0\} }\;ds $, we claim that
         \begin{equation}\label{eqn: weak convergence right hand sides for p^epsilon}
              \mathbb{I}_{\mathcal{O}_{\epsilon_k}}(x)\mathbb{I}_{\{p^{\epsilon_k}>0\}}(x,t)f^{\epsilon_k}(x,t)\rightharpoonup \mu\mathbb{I}_{\{p>0\}}(x,t) f(x,t)\;\text{ in } L^2,
         \end{equation}
         where $f = -\mathbb{I}_{\db_0^c}+\int_0^t\mathbb{I}_{\{p(\cdot,s)>0\}}\;ds$. Indeed, for any $\phi \in L^2(\mathbb{R}^d)$ we want to bound
         \begin{align*}
              I_1 = \left | \int_{\mathcal{O}_{\epsilon_k}} \mathbb{I}_{_{\{p^{\epsilon_k}>0\}}}f^{\epsilon_k}(x,t) \phi-\mathbb{I}_{_{\{p>0\}}}f(x,t) \phi \;dx\right |.
         \end{align*}
         After rearranging the terms, we see that 
         \begin{align*}
              I_1 &\leq \int_{\mathcal{O}_{\epsilon_k}} |\mathbb{I}_{_{\{p^{\epsilon_k}>0\}}}-\mathbb{I}_{_{\{p>0\}}}| |f^{\epsilon_k}(x,t) \phi| \;dx  + \int_{\mathcal{O}_{\epsilon_k}} |f^{\epsilon_k}(x,t) \phi-f(x,t)| |\mathbb{I}_{_{\{p>0\}}}\phi |\;dx\\
              &\leq \| f^{\epsilon_k}(.,t)\|_\infty \|\mathbb{I}_{_{\{p^{\epsilon_k}>0\}}}-\mathbb{I}_{_{\{p>0\}}}\|_1^{1/2}\|\phi\|_2+\|f^{\epsilon_k}(.,t)- f(.,t)\|_2\|\phi\|_2.
         \end{align*}
         The first term goes to zero thanks to Corollary \ref{cor: L^1 convergence of the phases} and the bound $\|f^{\epsilon_k}(.,t)\|_\infty\leq C(t)$ for all $\epsilon_k$. At the same time, the second term goes to zero due to Corollary \ref{cor: tau limit}. It follows that
         \begin{align*}
              \lim \limits_{\epsilon_k \to 0} \int \mathbb{I}_{_{\mathcal{O}_{\epsilon_k}}}\mathbb{I}_{_{\{p^{\epsilon_k}>0\}}}f^{\epsilon_k}(x,t) \phi \;dx & = \lim \limits_{\epsilon_k \to 0}\int \mathbb{I}_{_{\mathcal{O}_{\epsilon_k}}}\mathbb{I}_{_{\{p>0\}}}f(x,t) \phi \;dx.	
         \end{align*}
         On the other hand, Theorem \ref{thm: ergodic theorem} says that
         \begin{equation*}
              \lim\limits_{\epsilon \to 0} \int \mathbb{I}_{_{\mathcal{O}_\epsilon}}\psi\;dx = \mu \int \psi\;dx
         \end{equation*}
         for $\psi =\mathbb{I}_{_{\{p(\cdot,t)>0\}}}f(\cdot,t) \phi$, and \eqref{eqn: weak convergence right hand sides for p^epsilon} follows. In this case, we can apply Theorem \ref{thm: stochastic homogenization linear} and conclude that
         \begin{align*}
              \int \left (A \nabla p, \nabla \phi \right )\;dx & = \int \mu \left (\mathbb{I}_{\db_0^c}+\int_0^t \mathbb{I}_{\{p(\cdot,s)>0\}}\;ds\right ) \mathbb{I}_{\{p(\cdot,s)>0\}}\phi\;dx\;\;\forall\;\phi\in H^1. 
         \end{align*}
         where $A$ is the effective diffusivity defined in Section \ref{sec: linear homogenization revisited}. This shows that the limit $p(x,t)$ is the unique weak solution to \eqref{eqn: homogeneous obstacle evolution problem} and we are done.
    \end{proof}

    To finish, we remark that combining the uniqueness of the limit and Lemma \ref{lem: free sets limit} we have proved that the free boundaries $\partialom \{p^\epsilon(\cdot,t)>0\}$ converge uniformly to $\partial\{ p(\cdot,t)>0\}$ for every $t>0$.

    \begin{cor}\label{cor: strong convergence of FB}
    	 For $\mathbb{P}$-almost every $\omega$, $t>0$ and for every $\delta>0$ there exists $\epsilon_0$ such that for any $\epsilon<\epsilon_0$ we have that $\db_t(p^\epsilon)$ lies entirely in a $\delta$ neighborhood of $\db_t(p) \cap \mathcal{O}_\epsilon$ and $\db_t(p) \cap \mathcal{O}_\epsilon$ lies in a $\delta$ neighborhood of $\db_t(p^\epsilon)$.
    \end{cor}

    \begin{rem}
         In general one cannot show that $\db_t$ is continuous in time with respect to the Hausdorff distance. Intuitively, what may occur is two fingers of $\db_t(p)$ contact each other at time $t=t_0$ along some contact set $\mathcal{C}$ having positive $n-1$ dimensional Hausdorff measure. Then at the next moment $\mathcal{C}$ instantly disappears and we may have a discontinuity of $\db_t(p)$ at $t=t_0$.
    \end{rem}

\section{Homogenization}\label{sec: homogenization}

    We have now established the convergence of $p^\epsilon$ in the $L^\infty$ norm (see Definition \ref{def: notions of convergence}). Using this and the relationship between $p^\epsilon$ and $u^\epsilon$, we will investigate the convergence of $u^\epsilon$ and $\db^\epsilon$.

    \subsection{The relationship between $u^\epsilon$ and $p^\epsilon$}

    The first goal is to justify rigorously the relation 
    \begin{equation*}
         p^\epsilon(x,t,\omega) = \int_0^t u^\epsilon(x,s,\omega)ds. 
    \end{equation*}
    Since $\omega$ plays no role in the proof of this relation it suffices to prove it for a deterministic domain $\mathcal{O}$ for which Assumption \ref{hyp: regularity} holds, even for a deterministic $\mathcal{O}$ we shall use the $\partial^\omega$ notation.

    \begin{thm}\label{thm: evolution}
         Let  $\mathcal{O} \subset \mathbb{R}^d$ be  such that Assumption \ref{hyp: regularity} holds, and $\db_0 \subset \mathcal{O}$ such that $\partialom \db_0 = \partialom\overline{\db_0}.$ Let $p:\mathcal{O}\times \mathbb{R}_+ \to \mathbb{R}$ be a weak solution to \eqref{eqn: epsilon obstacle problem} and \eqref{eqn: definition of f^epsilon}.

         For each $t>0$, define $u(\cdot,t)$ as the unique weak solution to the boundary value problem (see Lemma \ref{lem: from p to u via Perron's method} for the definition of $u$ for non-smooth $\partialom\{p(\cdot,t)>0\}$)
         \begin{equation}\label{eqn: u in terms of p}
              \left\{\begin{array}{rlrll}
              -\Delta u & = & 1 &\text{ in } & \{p(\cdot,t)>0\}\cap\mathcal{O},\\ 
              \partial_nu & =& 0 &\text{ on } &\partial\mathcal{O},\\ 
              u & = &0 &\text{ on }& \partialom\{p(\cdot,t)>0\}.
             \end{array}\right.
         \end{equation}
         Then the following is true:
         \begin{itemize}
              \item[(a)] $u(x,t)$ is a viscosity solution of \eqref{eqn: deterministic Cauchy problem} (see Definition \ref{def: discontinuous viscosity solutions}). Moreover $u$ is continuous at $t=0$, i.e. $u(x,t)$ converges uniformly to $u_0(x)$, corresponding to the initial phase $\db_0$.
              \item[(b)]$u$ is equal to the left hand side time derivative of $p$:
              \begin{equation*}
                   u(x,t) = \partial_t^-p(x,t) := \lim_{h\to 0} \dfrac{p(x,t)-p(x,t-h)}{h} \text{ for } t>0.
              \end{equation*}
         \end{itemize}
\end{thm}

    The proof of this theorem is close to that of Theorem 3.1 in \cite{KM09}, but we present a detailed proof for completeness. A key idea in the proof is showing that if a supersolution touches the function $u(x,t)$ from above, then its integral with respect to time touches $p$ from above at the same point. This will in turn show that classical supersolutions starting above $u$ stay above it for all times.

    First, we carry out the construction of the function $u$ in detail.

    \begin{lem}\label{lem: from p to u via Perron's method}
    	 Let $p$ be as in the statement of Theorem \ref{thm: evolution}, then define $u(x,t)$ by
         \begin{equation*}
              u(x,t) := \sup \limits v(x).
         \end{equation*}
		 Here the supremum is taken (for each fixed $t$) over all lower semi-continuous functions $v:\mathcal{O}\to\mathbb{R}$ for which there exists some $s\in(0,t)$ such that in the viscosity sense we have
         \begin{equation*}
              -\Delta v  \geq 1\text{ in } \{p(\cdot,s)>0\}, \quad \partial_n v =0 \text{ on } \partial\mathcal{O} \quad\text{ and }\quad \overline{\{ v>0\}}\subset \{p(\cdot,s)>0\}.
         \end{equation*} 
         Then, $u(x,t)$ is a lower semi-continuous function solving in the viscosity sense
         \begin{equation*}
         	  -\Delta u = 1 \text{ in } \{ p(\cdot,t)>0\},\;\;\;\; \partial_n u = 0 \text{ on } \partial \mathcal{O}.
         \end{equation*}
         Moreover, we have the pointwise bound
         \begin{align}
	          t^{-1}p(x,t) & \leq u(x,t), \label{eqn: time}
	     \end{align}
	     and the coincidence of the positivity sets
	     \begin{align*}
              \{u>0\} = \{p>0\} \text{ and } \partialom\{u>0\} = \partialom\{p>0\}.
          \end{align*}         
    \end{lem}

	\begin{proof}
         That $u(x,t)$ is lower semicontinuous follows because it is essentially the limit of an increasing sequence of lower semi-continuous functions, we omit the details. Since $p(\cdot,s)/s$ is continuous in $x$ for every $s>0$ and it is a weak solution of \eqref{eqn: epsilon obstacle problem}, we conclude that it solves
         \begin{equation*}
         	  -\Delta (p(\cdot,s)/s) \leq 1 \text{ in } \{p(\cdot,s)>0\}\cap\mathcal{O}\text{ and } \partial_nu \geq 0 \text{ on } \partial \mathcal{O}.
         \end{equation*}
         Then, and application of the comparison principle says that
         \begin{equation*}
              p(x,s)/s \leq u(x,t) \text{ for all } s<t.
         \end{equation*}
         By the continuity of $p$ in time (Proposition \ref{prop: control_time}) and the semicontinuity of $u$ it follows that
         \begin{equation*}
              p(x,t)/t \leq u(x,t)\;\;\forall\;t>0.
         \end{equation*} 
         Thus $u(x,t)>0$ whenever $p(x,t)>0$. Again by lower semi-continuity, $u=0$ on $\partialom \{p>0 \}$, so
         \begin{equation*}
              \{u>0\}=\{p>0\},\quad \partialom\{u>0\} =\partialom\{p>0\}.
         \end{equation*}
         Also note that, by definition,
         \begin{equation*}
              -\Delta  u(\cdot,t)= 1 \text{ in } \{p(\cdot,t)>0\}\cap\mathcal{O} \text{ and }  u_n=0 \text{ at regular points of } \partial\mathcal{O}\cap \{p(\cdot,t)>0\}.
         \end{equation*}
         Lastly, observe that $-\Delta u(\cdot,t) \leq 1$ in the interior of $\mathcal{O}$, thus  $-\Delta u^*(\cdot,t) \leq 1$ in $\mathcal{O}$ and $u^*_n \leq 0$ on $\partial\mathcal{O}$, which finishes the proof.
    \end{proof}

    Next, we prove the key lemma that will allow us to relate test functions for $u$ with those for $p$. 

    \begin{lem}\label{lem: integrating a test function}
    	 Let $u$ be given by Lemma \ref{lem: from p to u via Perron's method} and $\phi$ a classical strict supersolution (see Definition \ref{def: classical Subsolutions} ) in some cylinder $\Sigma_0 := \mathcal{B}^1_r(x_0)\times (t_0-\tau_0,t_0)$, for some $x_0,t_0$ and with $\tau_0 \in (0,1)$.  Suppose that $u\leq \phi$ in $\Sigma_0$, $\overline{u} < \phi$ on $\{ \phi>0\}$ and $\{ u>0\}\cap \partialom_p\Sigma_0 \subset\subset \{ \phi>0\}$.

         There is a small enough  $\tau>0$ (depending on $u,\phi$) such that if $\Sigma := \mathcal{B}^1_r(x_0)\times (t_0-\tau,t_0)$ and
         \begin{align*}
 	          \tilde p(x,t) & := p(x,t)-p(x,t_0-\tau),\;\Phi(x,t) :=\int_{t_0-\tau}^t\phi(x,s)\;ds\;\;\;\forall\;(x,t)\in\Sigma,
         \end{align*}
         then we have the strict inequality $\tilde p < \Phi \text{ in } { \{ \Phi>0\}} \cap \Sigma$.
    \end{lem}

    \begin{proof}
    	 First, let us compare the pointwise values of $u$ and $\tilde p$ in $\Sigma$. To this end we observe that 
         \begin{equation*}
             \{ u>0 \} =\{p >0 \} = \{ \tilde	p>0\}.
         \end{equation*}
         Moreover, we know that $\tilde{p}(\cdot,t)$ solves (for each fixed $t\in (t_0-\tau,t_0)$)
         \begin{equation*}
              \left \{ \begin{array}{rll}
              	   -\Delta \tilde{p} & = \left (-\mathbb{I}_{\db_0^c}+\int_0^t\mathbb{I}_{\{p(\cdot,s)>0\}}\;ds \right )\mathbb{I}_{\{p(\cdot,t)>0 \} \setminus \{p(\cdot,t_0-\tau)>0 \} }+\left (\int_{t_0-\tau}^t\mathbb{I}_{\{p(\cdot,s)>0\}}\;ds \right )\mathbb{I}_{\{p(\cdot,t_0-\tau)>0\} } & \text{in } \mathcal{O},\\
                   \partial_n \tilde{p} & = 0 & \text{on } \partial\mathcal{O}.
              \end{array} \right.
         \end{equation*}
         In particular, since $\tau<\tau_0<1$ we obtain  $\int_0^t\mathbb{I}_{\{p(\cdot,s)>0\}}\;ds\leq \tau_0<1$ outside $\{ p(\cdot,t_0-\tau)>0\}$. On the other hand,  $\{ p(\cdot,t)>0\}$ is monotone increasing in time, so $\mathbb{I}_{\db_0^c}\equiv 1$ in $\{ p(\cdot,t_0-\tau)>0\}^c$. From all this it follows that
         \begin{equation*}
              \left (-\mathbb{I}_{\db_0^c}+\int_0^t\mathbb{I}_{\{p(\cdot,s)>0\}}\;ds \right )\mathbb{I}_{\{p(\cdot,t)>0 \} \setminus \{p(\cdot,t_0-\tau)>0 \} }\leq 0.
         \end{equation*}
         Then, since $\{p >0 \} = \{ \tilde	p>0\}$ and $\{ p(\cdot,t)>0\}$ is monotone increasing with $t$ we conclude that
         \begin{equation*}
             -\Delta\tilde{p} \;\leq \int_{t_0-\tau}^t\mathbb{I}_{\{p(\cdot,s)>0\}}\;ds\;\mathbb{I}_{\{\tilde{p}(\cdot,t_0-\tau_0)>0\}}\leq (t-t_0+\tau)\mathbb{I}_{\{\tilde{p}(\cdot,t)>0\}} \text{ in } \mathcal{O},\;\;\partial \tilde p = 0 \text{ on } \partial \mathcal{O}.
         \end{equation*}
         The above being not only in the $H^1$ sense but also in the viscosity sense, due to the solution $p(\cdot,t) \in H^1$ (and thus $\tilde p(\cdot,t)$) being continuous, by Theorem \ref{thm: elliptic estimates}. In this case $(t-t_0+\tau)^{-1}\tilde p(\cdot,t)$ is a (viscosity) subsolution for the same elliptic equation solved by $u$ in the domain $\{ \tilde p(\cdot,t)>0\}=\{ u(\cdot,t)>0\}$. Applying the comparison principle we obtain
         \begin{equation*}
              \tilde{p}(x,t) \leq (t-t_0+\tau)u(x,t)\;\;\forall\;(x,t)\in\mathcal{O}\times (t_0-\tau,t_0).
         \end{equation*}
         Recall that $\overline{u}<\phi$ in $\{\phi>0\}$ and $\{ u>0\}\cap \partialom_p\Sigma_0 \subset\subset \{ \phi>0\}$, then it follows there is a $\delta>0$ such that
         \begin{equation*}
              u+\delta\leq \phi \text{ on } (\partialom_p \Sigma)\cap \{ \phi>0 \}.         	
         \end{equation*}
         Therefore for all $(x,t)$ on this same set we have $\tilde{p}(x,t) \leq (t-t_0+\tau)(\phi(x,t)-\delta)$. At the same time,
         \begin{equation*}
              (t-t_0+\tau)(\phi(x,t)-\delta) = \int_{t_0-\tau}^t\phi(x,s)\;ds+\int_{t_0-\tau}^t\phi(x,t)-\phi(x,s)-\delta\;ds.
         \end{equation*}      
         so picking $\tau$ small enough (depending on $\delta$ and the continuity of $\phi$) we have
         \begin{equation*}
              (t-t_0+\tau)(\phi(x,t)-\delta) < \int_{t_0-\tau}^t\phi(x,s)\;ds=\Phi(x,t)  \text{ in } \Sigma \cap \{ \Phi>0\}.
	         \end{equation*}
         Then, using that $\tilde p\equiv 0$  outside $\{ \Phi>0\}$  we obtain
         \begin{equation}\label{eqn: boundary comparison tilde p and Phi}
              \tilde p \leq \Phi \text{ on } \partialom_p\Sigma\;,\;\;\tilde p <\Phi \text{ on } \partialom_p\Sigma \cap \{ \Phi>0\}.	
         \end{equation}
         On the other hand, as $\phi$ is a strict classical supersolution in $\Sigma_0$ and $\Sigma\subset\Sigma_0$, it can be checked that (in the $H^{-1}$ sense)
         \begin{equation*}
              -\Delta \phi > (\mathbb{I}_{\{\phi(\cdot,t)>0 \}})_t-\mathbb{I}_{\{\phi(\cdot,t)>0\} } \text{ in } \Sigma,\;\;\partial_n \phi >0	\text{ on } (\partial \mathcal{O})\cap \Sigma.
         \end{equation*}
         Integrating this with respect to $t$ from $t_0-\tau$ to $t$, it follows that
         \begin{equation}\label{eqn: supersolution big phi}
              \left \{ \begin{array}{rll}
              	   -\Delta\Phi(\cdot,t) & > -\mathbb{I}_{\{\phi(\cdot,t)>0\}}+\mathbb{I}_{\{\phi(\cdot,t_0-\tau)>0\}}+\int_{t_0-\tau}^t \mathbb{I}_{\{\phi(\cdot,s)>0 \}}\;ds & \text{in } \Sigma,\\
                   \partial_n \Phi & \geq 0 & \text{on }  (\partial \mathcal{O}) \cap \Sigma.
              \end{array}\right.            
         \end{equation}
         Also because $\phi$ is a supersolution, $\{\phi(\cdot,t)>0\}$ is monotone increasing in $t$, thus $\{\phi>0\} = \{ \Phi>0\}$ for $t>t_0-\tau$ and the inequality above may be rewritten as
         \begin{equation*}
              -\Delta\Phi(\cdot,t) > \left ( \mathbb{I}_{\{\Phi(\cdot,t_0-\tau)>0\}^c}-\int_{t_0-\tau}^t \mathbb{I}_{\{\Phi(\cdot,s)>0 \}}\;ds)\right )\mathbb{I}_{\{\Phi(\cdot,t)>0\}}.
         \end{equation*}
         Finally, in  $\Sigma$ we have $\mathbb{I}_{\{ \tilde p>0 \}}\leq \mathbb{I}_{\{ \Phi >0 \}}$, thus
         \begin{equation*}
              -\Delta \tilde p(\cdot,t) \leq (t-t_0+\tau)\mathbb{I}_{\{\Phi(\cdot,t)>0\} }.
         \end{equation*}
         On the other hand,
         \begin{equation*}
         	   \left (-\mathbb{I}_{\{\Phi(\cdot,t_0-\tau)>0\}^c}+\int_{t_0-\tau}^t \mathbb{I}_{\{\Phi(\cdot,s)>0 \}}\;ds\right )\mathbb{I}_{\{\Phi(\cdot,t)>0\}}> (t-t_0+\tau)\mathbb{I}_{\{\Phi>0\} }.
         \end{equation*}     
         From all this we conclude that
         \begin{equation*}
              \left \{ \begin{array}{rlll}
              	   -\Delta\Phi(\cdot,t) & > & -\Delta \tilde p(\cdot,t) & \text{in } \mathcal{B}^1_r(x_0)\times \{t\},\\
                   \partial_n \Phi(\cdot,t) & \geq & \partial_n \tilde p(\cdot,t) & \text{on }  (\partial \mathcal{O}) \times \{t\}.
              \end{array}\right.            
         \end{equation*}         
         Thus, by the comparison principle,
         \begin{equation*}
              \sup \limits_{\mathcal{B}^1_{r}(x_0)} (\Phi -\tilde p)(\cdot,t)<\sup \limits_{\partialom \mathcal{B}^1_{r}(x_0)} (\Phi -\tilde p)(\cdot,t)\leq 0
         \end{equation*}
         where we used \eqref{eqn: boundary comparison tilde p and Phi} to get the strict inequality. This shows that $\Phi<\tilde p$ in $\Sigma_0$ and the lemma is proved.
    \end{proof}

	\begin{proof}[Proof of Theorem \ref{thm: evolution}]

         Let us show that $u$ as defined by Lemma \ref{lem: from p to u via Perron's method} is a subsolution in the sense of Definition \ref{def: discontinuous viscosity solutions}, the argument for the supersolution property is similar. If $u$ were not a subsolution we could find some cylinder $E$ and a classical strict supersolution $\phi(x,t)$ such that $\overline{u}< \phi$ on $\partialom_p E$ and yet $\overline{u} \geq \phi$ somewhere in $E$.

         Let $t_0$ be the supremum of the times $t'$ such that  $\overline{u}\leq\phi$ within $E\cap \{t\leq t'\}$. Then $\phi$ touches $\overline{u}$ from above at some point $(x_0,t_0)\in E$. By using the strong maximum principle for the Poisson equation with Neumann data, we can easily obtain a contradiction if we have $p(x_0,t_0)>0$ or $\phi(x_0,t_0)>0$. Thus we only need to rule out the case where $x_0 \in \partialom \{p(\cdot,t_0)>0\} \cap \partialom \{\phi(\cdot,t_0)>0\}$. 

         By looking at the first time when $\overline{u}$ fails to lie strictly below $\phi$, we obtain $t_0>0$, $x_0 \in \overline{\mathcal{O}}$ and small numbers $r,\tau_0 \in (0,1)$ such that $\overline{u}\leq \phi$ in $\Sigma = \mathcal{B}^1_r(x_0)\times (t_0-\tau_0,t_0)$ and $\overline{\{ u>0 \}} \cap \Sigma\subset\{ \phi>0 \}$. Moreover, by a small perturbation argument on $\phi$ we may assume that the free boundaries of $u$ and $\phi$ stay separated on $\partial_p \Sigma$, so that $u$ and $\phi$ comply with the assumptions of Lemma \ref{lem: integrating a test function} in a cylinder $\Sigma_0$ centered at $(x_0,t_0)$. 

         Due to the Lemma we conclude then that $\tilde p< \Phi$ in a neighborhood of $(x_0,t_0)$, but this is a contradiction since $x_0 \in \partialom \{p(\cdot,t_0)>0\}\cap \partialom \{\phi(\cdot,t_0)>0\}$, which means that $p(x_0,t_0)=\Phi(x_0,t_0)$. From this contradiction it follows that $(x_0,t_0)$ cannot exist and thus that $\overline{u}$ can never go above $\phi$, proving that $u$ is a subsolution.



         To prove $u$ is a viscosity solution of \eqref{eqn: Hele-Shaw O_ep} all is left is checking that $u$ and $\overline{u}$ both satisfy the initial data. Due to Corollary~\ref{cor: distance from the initial phase} we have that
         \begin{equation*}
              \db_0\subset \{p(\cdot,t)>0\} \subset \db_0+ B_{t^{1/2}}(0).
         \end{equation*}
         As we are assuming $\partial \db_0= \partial\overline{\db}_0$, it follows that $\db_t$ converges uniformly to $\db_0$ in Hausdorff distance. This proves part (a) of the theorem, it remains to show (b). To this end let us define for $t\geq h$,
         \begin{equation*}
              u_h^-(x,t):= \dfrac{p(x,t)-p(x,t-h)}{h}.
         \end{equation*}
         Since $u_h^-$ is monotone decreasing with respect to $h$, there exists a function $u^-$ such that $u_h^-(x,t)$ converges pointwise to $u^-(x,t)$ as $h\to 0$.  Due to the inequality above we see that $u^-\leq u$. Let us show that also $u^- \geq u$ to conclude that $u^-=u$.

         Consider any lower semicontinuous function $\phi$ with support in $\{p(\cdot,s)>0\}$ for some $s<t$ such that 
         \begin{equation*}
              -\Delta \phi \leq 1\;\;\text{ in }\{p(\cdot,s)>0\}\cap\mathcal{O}, \quad \partial_n \phi=0\quad \text{ on }\partial\mathcal{O}.
         \end{equation*}
         Then, from the  comparison principle we obtain $u_h \geq \phi$ for any $h> t-s$. From the definition of $u$ it follows then that $u \leq u^-$, and the theorem is proved.

    \end{proof}

    Recall that when $\partial\mathcal{O}$ is $C^3$ and $\partial \db_0$ is $C^1$ we know that uniqueness holds for \eqref{eqn: deterministic Cauchy problem}. In light of Theorem \ref{thm: evolution} as well as Theorem \ref{thm: cp2} we obtain uniqueness for \eqref{eqn: homogeneous obstacle evolution problem} too. We recall that existence for this problem is discussed in Appendix \ref{sec: construction of the obstacle problem}.
 
    \begin{cor}\label{cor: uniqueness for aux obstacle problems}
         Suppose $\mathcal{O}$, $\db_0$ and $u$ are as given in Theorem~\ref{thm: evolution}, and suppose in addition that $\partial\mathcal{O}$ is $C^3$ and $\partial \db_0$ is $C^1$, then the function $p(x,t):=\int_0^t u(x,s) ds$ is the unique weak solution of \eqref{eqn: homogeneous obstacle evolution problem}.
    \end{cor} 
    
    \subsection{Proof of Theorem \ref{thm: main}} From now on,  $u^\epsilon(x,t,\omega)$ will be the specific viscosity solution of \eqref{eqn: Hele-Shaw O_ep} given in terms of $p^\epsilon$ for $\mathbb{P}$-almost every $\omega$ via Theorem \ref{thm: evolution}. In particular, we have
    \begin{equation*}
         u^\epsilon(x,t,\omega):= \partial_t^- p^\epsilon(x,t,\omega).
    \end{equation*}
    We recall the notion of upper and lower semi-continuous limits,
    \begin{align*}
	    u^*(x,t,\omega) &: = \limsup_{(y,s)\in\overline{\mathcal{O}}_\e,\ep \to (x,t),0} u^\epsilon(y,s,\omega),\\
	    u_*(x,t,\omega) &: = \liminf_{(y,s)\in\overline{\mathcal{O}}_\e,\ep \to (x,t),0} u^\epsilon(y,s,\omega).
    \end{align*} 
    Note that $u^*$ and $u_*$ are well-defined a.e. in $\mathbb{R}^d\times [0,\infty)\times \Omega$, since Assumption \ref{hyp: regularity} guarantees that
    \begin{equation*}
	     \mathbb{R}^d = \limsup_{\epsilon\to 0} \mathcal{O}_\epsilon\pal
	\end{equation*}
    These limits, known also as half-relaxed limits in the viscosity solution literature, are convenient for two reasons: first the functions $u^\epsilon$ are defined in different spaces (so we cannot talk about the usual pointwise limit in $x$), and second, we must deal with the possible lack of regularity of $u^\epsilon$, which may even  be discontinuous at later times. As we have mentioned already, this is a problem even for a fixed $\epsilon$. The lack of regularity means we cannot simply arrive at uniform convergence by using a compactness argument. However, if we can show that $u^*=u_*$ the uniform convergence follows.

    In any case, in working towards the proof of Theorem \ref{thm: main} we wish to investigate whether $u_* =u^*$ and whether both are equal to the unique solution of \eqref{eqn: Hele-Shaw homogeneous}. To this end, let $p:\mathbb{R}^d\times \mathbb{R}_+ \to \mathbb{R}$ be the unique solution to \eqref{eqn: homogeneous obstacle evolution problem} and  let $u:\mathbb{R}^d\times \mathbb{R}_+ \to \mathbb{R}$  be the corresponding viscosity solution to \eqref{eqn: Hele-Shaw homogeneous} given by Theorem \ref{thm: evolution}. In this case we have that $u=\partial^-_t p\;$ for a.e. $x,t$ and
    \begin{equation*}
         \left\{\begin{array}{rlrl}
              -\text{div}(A\nabla u)  & = & \mu & \text{in } \; \;\;\;\{ p (\cdot,t)>0\},\\
               u & = & 0 & \text{on } \; \partial\{ p (\cdot,t)>0\}.
         \end{array}\right.
    \end{equation*}

    We already know that $p^\epsilon \to p$ (Theorem \ref{thm: convergence of obstacle problems}) uniformly, this implies at once that
    \begin{equation*}
         u^\epsilon \to u \text{ in } L^p_{\text{loc}}(\mathcal{B}^\epsilon_R(x_0)\times [0,T])\pal	
    \end{equation*}
    for any $x_0,R,T$ and any $p \in [1,+\infty)$. This does give a homogenization result of $u^\epsilon$ to $u$ in the $L^p$ topology, but it does not clarify whether $u^\epsilon$ converges to $u$ pointwise. Moreover, using  Corollary \ref{cor: strong convergence of FB} together with $\partialom\{ p^\epsilon>0\}=\partialom\{u^\epsilon>0\}$ and  $\partial\{ p>0\} = \partial\{u>0\}$ we obtain the uniform convergence of the free boundaries in the Hausdorff distance.

    All that is left to investigate is the uniform convergence of $u^\epsilon$ itself. Note that what we want to verify is, roughly speaking, whether taking the time derivative and taking the limit $\epsilon\to 0$ for $p^\epsilon$  commute.

    The plan is to exploit the comparison principle for the limiting problem to relate $u$ to $u^*$ and $u_*$. We begin with the verification of the elliptic equations that $u^*$ and $u_*$ satisfy.	
    \begin{lem}\label{lem: u_* supersol u^* subsol}
	     For every fixed $t>0$ and in the viscosity sense, we have
         \begin{align}
         	  -\text{div}\left (A \nabla u^* \right)  & \leq \mu \; \text{ in } \mathbb{R}^d,\label{eqn: subsolution}\\
              -\text{div}(A\nabla u_*)  & \geq \mu \; \text{ in } \{u_*(\cdot,t)>0\}.\label{eqn: supersolution}
         \end{align}
    \end{lem}

    \begin{proof}
         We only prove \eqref{eqn: supersolution} explicitly, as \eqref{eqn: subsolution} is proven with the same method. Suppose \eqref{eqn: supersolution} does not hold in the viscosity sense, then there exists $(x_0,t_0)\in\{u_*>0\}$, $t_0>0$, and a smooth function $\phi$ such that $u(\cdot,t_0)-\phi$ has a maximum at $x_0$ while $\text{div}(A \nabla \phi) (x_0)>-\mu$.  By adding $c+\delta(x-x_0)^2$ to $\phi(x)$ with $c=u(x_0,t_0)-\phi(x_0)$ and with sufficiently small $\delta>0$, we may assume that the maximum of $u(\cdot,t_0)-\phi$ is zero and strict. Now, since $\phi$ is smooth there is some small $r>0$ such that 
         \begin{equation*}
              \text{div}(A\nabla\phi)>-\mu  \text{ in } B_r(x_0)\text{ and } u(\cdot,t_0)<\phi \; \text{ on } \partial B_r(x_0).
         \end{equation*}
         Next, consider the function $\varphi:B_r(x_0)\to \mathbb{R}$ determined by
         \begin{equation*}
              -\text{div}(A \nabla \varphi) = -\mu \text{ in } B_r(x_0)\text{ and } \varphi=\phi \;\text{ on } \partial B_r(x_0).
         \end{equation*}
         By the comparison principle, $\varphi>\phi$ in $B_r(x_0)$, and thus $\varphi-\phi$ has an interior maximum in $B_r(x_0)$. On the other hand, for each $\epsilon>0$  consider the function $\varphi^\epsilon: \mathcal{B}^\epsilon_r(x_0) \to \mathbb{R}$ determined by
         \begin{equation*}
              \Delta \varphi^\epsilon= -1\text{ in } \mathcal{B}^\epsilon_r(x_0), \quad \varphi^\epsilon_n=0 \text{ on } \partial\mathcal{O}_\epsilon \text{ and } \varphi^\e=\phi \text{ on } \partial \mathcal{B}^\epsilon_r(x_0).
         \end{equation*}
         Then Theorem \ref{thm: stochastic homogenization linear} (see Section \ref{sec: linear homogenization revisited}) says that$\pal$ as $\epsilon \to 0$ $\varphi^\epsilon$ converges uniformly in compact subsets of $\mathcal{B}^\epsilon_r(x_0)$ to $\varphi$.  Therefore if $\epsilon$ is small enough $u^\epsilon-\varphi^\epsilon$ has an interior maximum in $\mathcal{B}^\epsilon_r(x_0)$ while $u^\epsilon <\phi=\varphi^\epsilon$ on $\partial \mathcal{B}^\epsilon_r(x_0)$, contradicting the strong comparison principle. This contradiction proves \eqref{eqn: supersolution}.
    \end{proof}

    Next, we recall that $u^\epsilon(x,t) \geq t^{-1}p^\epsilon(x,t)$ and that
    \begin{equation*}
         \lim \limits_{\epsilon \to 0}\|p^\epsilon(\cdot,t)-p(\cdot,t)\|_{L^\infty(\mathcal{O_\epsilon})} = 0,\;\;\forall\;t>0,
    \end{equation*}
    from where it follows that $u_* (x,t) \geq t^{-1}p(x,t)$. Thus
    \begin{equation}\label{eqn: u positivity set easy inclusion}
         \{u >0\}\subset \{u_*>0\}.
    \end{equation}
    Then, applying the comparison principle for the regular Poisson equation, the above yields
    \begin{equation*}
    	 \overline{u}\leq u^*.
    \end{equation*}
    In light of this, all that is left to do to prove Theorem \ref{thm: main} is that
    \begin{equation}\label{eqn: u positivity set difficult inclusion}
         \overline{u} \geq u^* \text{ and } \Gamma(\overline{u}) = \Gamma(u).
    \end{equation}
    This is the content of the next proposition. 
    \begin{prop}\cite[Proposition 3]{KM09}\label{prop: subsolution property for u^*}
         $u^*$ is a viscosity subsolution of \eqref{eqn: Hele-Shaw homogeneous}, in particular,  $u^* \leq \overline{u}$.
    \end{prop}

    The proof is essentially finished, we just gather the facts we have proven so far.

    \begin{proof}
           Due to Lemma~\ref{lem: u_* supersol u^* subsol}, $u^*$ satisfies $\text{div}(A\nabla u^*) \geq -\mu$ in $\mathbb{R}^d$. Therefore we only need to check that $u^*$ satisfies the desired property for subsolutions on the free boundary $\partial \{ u^*>0\}$. 

         Suppose  $u^*$ is not a subsolution of \eqref{eqn: Hele-Shaw homogeneous}. Then  there exists a point $(x_0,t_0)\in \partial\{ u^*>0 \}$ and a classical supersolution $\phi$ of \eqref{eqn: Hele-Shaw homogeneous} in a cylinder $\Sigma:=B_r(x_0)\times (t_0-\tau,t_0]$ which touches  $u^*$ from above at $(x_0,t_0)$,
         \begin{equation*}
              u^*\leq \phi \text{ in } \Sigma, \quad u^*(x_0,t_0)=\phi(x_0,t_0). 
         \end{equation*}
         Clearly this contradicts the maximum principle whenever $(x_0,t_0)$ is not on $\partial\{u*>0\}$ or $\partial\{\phi>0\}$. Thus we only treat in detail the case where $\partial\{u^*>0\}$ and $\partial \{\phi>0\}$ touch at $(x_0,t_0)$. Note that, due to \eqref{eqn: u positivity set easy inclusion}  we have $\overline{u} \leq u^* \leq \phi$ in $\Sigma$, and we conclude $\overline{u}$ crosses $\phi$ from below in $\Sigma$  at $(x_0,t_0)$, contradicting the fact that $\overline{u}$ is a viscosity subsolution of \eqref{eqn: Hele-Shaw homogeneous}.
    \end{proof}
 
    \begin{proof}[Proof of Theorem~\ref{thm: main}]
	
	     We apply Corollary \ref{cor: strong convergence of FB} and conclude that
         \begin{equation}\label{eqn: coincidence of free boundaries}
              \begin{array}{c}
	             \Gamma(u^*)=\Gamma(u_*)=\Gamma(u)=\Gamma(\overline{u}),\\
	             \db(u^*)=\db(u_*)=\db(u)=\db(\overline{u}).
              \end{array}
         \end{equation}	
         Then, thanks to \eqref{eqn: u positivity set difficult inclusion} and Proposition \ref{prop: subsolution property for u^*} we obtain
         \begin{equation*}
              u \leq u_*\leq u^*\leq \overline{u}.	
         \end{equation*}      
         In particular, if $u$ is continuous we have $u=\overline{u}$ which from the above estimate means that $u_*=u^*$, this finishes the proof of the theorem.
    \end{proof}
    Finally, let us point out a case where it is known that $u$ is continuous in time.
    \begin{lem}\label{lem: star shaped initial data}
         If $\db_0$ is star shaped with respect to a ball $B_r(0)$, then so is $\db_t(u)$. Moreover, $u$ is H\"older continuous in both space and time.
    \end{lem}

    \begin{proof}
         First let us show that $\db_t(u)$ is star shaped with respect to $B_r(0)$. For this, for given $\delta>0$  and $x_0\in B_r(0)$ let us  consider the perturbed function
         \begin{equation*}
              \tilde{u}(x,t) = (1+\delta)^{-2}u((1+\delta)(x-x_0)+x_0,t).
         \end{equation*}
         Then $\tilde{u}$ also solves \eqref{eqn: Hele-Shaw homogeneous} with the initial data $\tilde{u}(x,0) \leq u(x,0)$, due to our assumption. Therefore Theorem~\ref{thm: comparison for Hele-Shaw C^3 boundary} yields that $\tilde{u}\leq u $ in $\mathbb{R}^d\times [0,\infty)$, and we conclude since $\delta>0$  is arbitrary.

         It follows that $\Gamma_t(u)$ is locally Lipschitz continuous at each time.  Since $u(\cdot,t)$ solves a uniformly elliptic equation in its positive set with Dirichlet boundary condition on $\Gamma_t(u)$, it follows that $u(\cdot,t)$ is H\"{o}lder continuous in $\overline{\db}_t(p)$. Moreover, due to Lemma~\ref{lem: control_FB_time} and $\db_t(u)$ being star shaped it follows that $u$ and $\Gamma(u)$ is also H\"older continuous over time.

    \end{proof}

	\section{Linear Stochastic homogenization revisited}\label{sec: linear homogenization revisited}

	In this section we discuss the homogenization of linear elliptic equations in random domains $\mathcal{O}_\epsilon$, which was used in Section \ref{sec: convergence obstacle problem} to identify the limit of $p^\epsilon$. We continue to make extensive use the notation introduced in Section \ref{sec: fbp in stationary ergodic domains}.

	Let $\mathcal{O}$ be a stationary ergodic domain and let $\text{O}$ be such that $O=\{ \omega \mid 0 \in \mathcal{O}(\omega)\} $ as  in Definition \ref{def: stationary ergodic domain}, and consider a sequence of functions $f_\epsilon \in L^2(\bfb_\epsilon)$ such  that $\mathbb{P}\text{-a.s.}$ we have $\mathbb{I}_{\mathcal{O}_\epsilon} f^\epsilon \rightharpoonup f$ to some $f\in L^2(D)$. Here we will be studying the sequence $\{v^\epsilon\}_\epsilon$ given by the Dirichlet problem
	\begin{equation*}
	     \left \{ \begin{array}{rll}
	          \Delta v^\epsilon & = f^\epsilon & \text{ in } D \cap \mathcal{O}_\epsilon,\\
	          \partial_n v^\epsilon & = 0  & \text{ on } \partial \mathcal{O}_\epsilon,\\
	          v^\epsilon & = 0  & \text{ on } \partialom D.
	     \end{array} \right.
	\end{equation*}
	We are interested in identifying a deterministic limit as $\epsilon \to 0$ as done in the uniformly elliptic but stochastic case in \cite[Section 3]{PapVar1979} and \cite{Koz1978}. The case of a stationary ergodic domain with Neumann data has been previously considered by Jikov (see \cite{JikKozOle1994}). In all cases, the main issue is dealing with corrector, once the corrector problem is solved, the method pioneered by Tartar in the periodic case \cite[Chapter 1]{BenLioPap1978} provides the proof of Theorem \ref{thm: stochastic homogenization linear}. The review of the aforementioned results is given below. 

	\begin{thm}\label{thm: stochastic homogenization linear}
		 Suppose that $\mathcal{O}$ satisfies Assumptions \ref{hyp: Ergodic} and \ref{hyp: regularity}, and let $u^\epsilon$ and $f^\epsilon$ be as above. Then $\mathbb{P}\text{-a.s.}$ as $\epsilon \to 0^+$ we have the weak limit
		 \begin{equation*}
		      \mathbb{I}_{\mathcal{O}_\epsilon}	u^\epsilon \rightharpoonup \mu u \text{ in } L^2(D).
		 \end{equation*}
		 Here $\mu = \mathbb{P}(\text{O})$ (as Definition \ref{def: mu}) and $u \in H^1_0(D)$ is the unique solution of
	     \begin{equation}\label{eqn: linear homogenized equation}
	          \text{div}(A \nabla u) = f \text{ in } D,\;\;\; u = 0 \text{ on } \partial D,
	     \end{equation}
	     whee $A$ is a positive definite matrix (the ``effective diffusivity'') that we will describe in detail below.
	\end{thm}

	\begin{DEF}\label{def: L^2_pot}
	     The space of stationary potential vector fields, $L^2_{\pot}(\Omega) \subset L^2(\Omega,\mathbb{R}^d)$, is defined as
	     \begin{equation*}
	          L^2_{\pot}(\Omega):=\{ v \in L^2(\Omega;\mathbb{R}^d) \mid \mathbb{E}[v]=0,\;\;\tilde v(x,\omega):=v(T_x\omega) \text{ is a potential vector field}\pal\}	
	     \end{equation*}
	     Given a measurable subset $\text{O}\subset \Omega$ we will denote by $L^2_{\pot}(\Omega,\text{O})$ the vector space
	     \begin{equation*}
	          L^2_{\pot}(\Omega,\text{O})=\{ v\in L^2(\text{O};\mathbb{R}^d) \mid \exists \; v' \in L^2_\pot(\Omega) \text{ s.t. } v=v' \text{ in } \text{O}\}.	
	     \end{equation*}    
	     We will also work with a space of stationary ergodic scalar functions which give $\dot{H}^1_{\text{loc}}$ functions,
	     \begin{equation*}
	          \mathcal{H}^1(\Omega):=\{ \phi \in L^2(\Omega) \mid \;\tilde \phi(x,\omega):=\phi(T_x\omega) \in H^1_{\text{loc}}(\mathbb{R}^d)\;\pal\},	
	     \end{equation*}
	     and the restriction of these functions to $\text{O}$,
	     \begin{equation*}
	     	  \mathcal{H}^1(\Omega,\text{O}) := \{ \phi \in L^2(\text{O}) \mid \exists \; \phi' \in \mathcal{H}^1(\Omega) \text{ s.t. } \phi=\phi' \text{ in } \text{O}\}.
		 \end{equation*}
	\end{DEF}

    We shall now use the extension operator from Proposition \ref{prop: Poincare} together with the harmonic extension to build an extension operator the probability space. Although this requires assuming that the domain is Lipschitz and well connected (that is, Assumption \ref{hyp: regularity}), it allows for a proof of homogenization in stationary ergodic domains that is conceptually closer to the proof of homogenization for a uniform elliptic operator in $\mathbb{R}^d$. 
    \begin{lem}\label{lem: stationary ergodic extension}
    	 If the domain $\mathcal{O}$ satisfies Assumption \ref{hyp: regularity}, then there exists an extension operator
         \begin{equation*}
         	  \Ext: \mathcal{H}^1(\Omega,\text{O})\to \mathcal{H}^1(\Omega),
         \end{equation*}
         which commutes with the ergodic action in the following sense: take any $u \in \mathcal{H}^1(\Omega)$ and $x \in \mathbb{R}^d$, then if we apply $\Ext$ to the restriction of $u\circ T_x$ to $\text{O}$ we have
         \begin{equation*}
              \Ext(u)\circ T_x=\Ext(u\circ T_x).
         \end{equation*}
    \end{lem}
    \begin{proof}
         We define $\Ext$ by first defining it in $H^1(\mathcal{O}(\omega))$ for every $\omega$, and then lifting the construction back to the probability space. The problem is doing this construction in some canonical way that guarantees it will commute with $T_x$, and thus that it can be lifted at all to functions in $\Omega$.

         It is not hard to see that the harmonic extension is well suited for this. Thus, given $u \in \mathcal{H}^1(\Omega,\text{O})$ let $\tilde u(\cdot,\omega)$ denote its realization as a function in $\mathcal{O}(\omega)$, and define $\Ext^\omega(\tilde u(\cdot,\omega))$ as the unique function agreeing with $\tilde u(x,\omega)$ in $\mathcal{O}(\omega)$ such that for every $k$ it minimizes (with $\hat P_k(\omega)$ as Proposition \ref{prop: Sobolev})
         \begin{equation*}
              \int_{\hat P_k(\omega)}|\nabla v|^2\;dx	
         \end{equation*}
         among all functions $v \in H^1(\hat P_k(\omega))$ that agree with $\tilde u(x,\omega)$ in $\hat P_k(\omega) \setminus P_k(\omega)$. The existence of $\Ext^\omega(\tilde u(\cdot,\omega))$ follows by Proposition \ref{prop: Poincare}, which is how we are using  Assumption \ref{hyp: regularity}. It also implies the bound
         \begin{equation*}
              \|\Ext^\omega(\tilde u(\cdot,\omega))\|_{\dot{H}^1(B_r(0))}\leq C\|\tilde u(\cdot,\omega)\|_{\dot{H}^1(\mathcal{B}^\epsilon_r(0))}, 	
         \end{equation*}
         which holds for any $r,\epsilon>0$ with the universal constant $C$ given by Proposition \ref{prop: Poincare}.
         
    \end{proof}

	\begin{lem}\label{lem: characterization of L^2_pot}
		 If the domain $\mathcal{O}$ satisfies Assumption \ref{hyp: regularity} then the space $L^2_{\pot}$ is complete.
	\end{lem}

	\begin{proof}
         Consider  a sequence $\{v'_n\}_n$ of functions belonging to $ L^2_{\pot}(\Omega)$ such that $v'_n \to  v_0$ in the $L^2(\text{O};\mathbb{R}^d)$ norm. Then, by definition, for each $n$ there is a function $u_n \in \mathcal{H}^1(\Omega)$ such that $\nabla u_n = v'_n$ in $\text{O}$, and by Lemma \ref{lem: stationary ergodic extension} we know that if we define functions $u_n':=\Ext(u_n) \in \mathcal{H}^1(\Omega)$ then
         \begin{equation*}
              \|u'_n\|_{\mathcal{H}^1(\Omega)}\leq C\|u'_n\|_{\dot{\mathcal{H}}^1(\Omega,\text{O})}\leq C\|v'_n\|_{L^2_{\pot}(\Omega,\text{O})}.	
         \end{equation*}
         It follows that the sequence $\{u'_n\}_n$ is weakly compact in $\mathcal{H}^1(\Omega)$. Let $u$ be a limit point of this sequence, then, the convergence of the $v_n'$ says that (in the $L^2$ norm)
         \begin{equation*}
         	  v'_n \to \nabla u.
         \end{equation*}
         We conclude that $\nabla u = v_0$ in $\text{O}$, which shows that $v_0 \in L^2_{\pot}$ and proves the lemma.
	\end{proof}

	Lemma \ref{lem: characterization of L^2_pot} leads to the following Theorem, which corresponds to the stochastic analogue of the ``cell problem'' in periodic homogenization with uniform ellipticity.

	\begin{thm}\label{thm: corrector problem}
	     For any $\xi \in \mathbb{R}^d$ there exists a sequence $\{ w_\xi^\epsilon\}_{\{\epsilon>0\}}$, $w_\xi^\epsilon: \mathbb{R}^d \times \Omega \to \mathbb{R}$ such that
	     \begin{enumerate}
	          \item $w_\xi^\epsilon \in H^1_{\text{loc}}(\mathbb{R}^d)$ $\mathbb{P}\text{-a.s.}$
	          \item $\nabla w_\xi^\epsilon(x,\omega) = \nabla w_\xi^1(\tfrac{x}{\epsilon},\omega)$, and $\nabla w_\xi^1$ is a stationary ergodic vector field.
	          \item For any $\epsilon$ and $\mathbb{P}\text{-a.s.}$ we have
	               \begin{equation*}
	                    \left \{ \begin{array}{rll}
		          	    \Delta w_\xi^\epsilon & = 0 & \text{ in } \mathcal{O}_\epsilon,\\
		                \partial_n w_\xi^\epsilon & = 0  & \text{ on } \partial \mathcal{O}_\epsilon.\\
                        \end{array} \right.
	              \end{equation*}
	          \item $\lim \limits_{\epsilon \to 0} w^\epsilon_\xi(x) = (x,\xi)  \text{ in } L^2_{\text{loc}} \;\;\mathbb{P}\text{-a.s.}$	
	     \end{enumerate}         
	\end{thm}

	\begin{proof}
	     We note that for each $\xi$ there is a unique $\Psi_\xi \in L^2_{\text{pot}}$ such that
	     \begin{equation}\label{eqn: random corrector equation}
	          \mathbb{E}[ \mathbb{I}_{\text{O}}(\xi+\Psi_\xi,v)]=0	\;\;\forall\;v\in L^2_{\text{pot}}.
	     \end{equation}
	     Indeed, this follows since $L^2_{\text{pot}}$ is complete by Lemma \ref{lem: characterization of L^2_pot}, in which case we may apply the Lax-Milgram theorem. In particular, note that the map $\xi\to \Psi_\xi$ is clearly linear. Moreover, by Lemma \ref{lem: characterization of L^2_pot} there is a function $\chi_\xi(x,\omega)$ (not necessarily stationary) such that
	     \begin{equation*}
	          \nabla \chi_\xi(x,\omega) = \Psi_\xi(T_x\omega)	
	     \end{equation*}
	     whenever $\omega \in T_x\text{O}$, that is whenever $x \in \mathcal{O}(\omega)$, and $\chi_\xi(\cdot,\omega) \in H^1_{\text{loc}}\;$ $\mathbb{P}\text{-a.s.}$. Finally, we define
	     \begin{equation*}
	          w^\epsilon_\xi(x,\omega) = x\cdot \xi +\epsilon \chi_\xi(\tfrac{x}{\epsilon},\omega)-a^\epsilon(\omega)
	     \end{equation*}
	     where $a^\epsilon(\omega)$ is chosen so that
	     \begin{equation*}
	          \int_{\mathcal{B}^\epsilon_1(0)} w^\epsilon_\xi(x,\omega)\;dx=0 \pal\forall\;\epsilon>0.
	     \end{equation*}
	     It is clear that  $w^\epsilon_\xi$ satisfies Properties (1) and (2). On the other hand, note that
	     \begin{equation*}
	          \int_\Omega \mathbb{I}_{\text{O}}(\omega)(\nabla w^1(x,\omega),v(\omega))\text{d}\mathbb{P}(\omega) = 0\;\;\forall\;v\in L^2_{\pot}.	
	     \end{equation*}
	     In particular, this holds for $v = \nabla \phi$, when $\phi \in \mathcal{H}^1(\Omega)$. Letting $\tilde \phi(x,\omega):= \phi(T_x\omega)$, and integrating the above identity with respect to $x$ over a ball $B_R(0)$, we obtain
	     \begin{equation*}
	          \int_{B_R(0)} \int_\Omega \mathbb{I}_{\text{O}}(\omega)(\nabla w^1(x,\omega),\nabla \phi(x,\omega))\text{d}\mathbb{P}(\omega)\;dx = 0.	
	     \end{equation*}
	     Equivalently (using Fubini's theorem),
	     \begin{equation*}
	          \int_{\Omega} \int_{\mathcal{B}^1_R(0)}(\nabla w^1(x,\omega),\nabla \tilde \phi(x,\omega)) dx \;\text{d}\mathbb{P}(\omega) = 0.
	     \end{equation*}	     
	     Which holds for any $\tilde \phi$ obtained from a $\phi \in \mathcal{H}^1(\Omega)$. From here one can see that
	     \begin{equation*}
	          \int_{\mathcal{B}^1_R(0)}(\nabla w^1(x,\omega),\nabla \phi(x,\omega)) dx = 0, 
	     \end{equation*}	
	     since $R>0$ was arbitrary, this proves Property (3). To show Property (4) note that$\pal$ the sequence $\{w^\epsilon_\xi\}$ is bounded in $H^1_{\text{loc}}$. Due to the ergodicity of $\nabla w^{\epsilon}_\xi$  it follows that
	     \begin{equation*}
	          \nabla w^\epsilon_\xi \rightharpoonup \mathbb{E}[\xi+ \nabla \chi^1_\xi] = \xi \text{ in } L^2 (\Omega).
	     \end{equation*}
	     This shows that $w^\epsilon_\xi$ converges in $L^2_{\text{loc}}$ and weakly in $H^1$ to a function with gradient $\xi$. Since $w^\epsilon_\xi$ has average zero in $\mathcal{B}^\epsilon_1$ it also follows that this function must be precisely $x\cdot \xi$, and the theorem is proved.
	\end{proof}

	\begin{rem}\label{rem: test function limit}  It is helpful to think of $w^\epsilon_\xi$  as corrections to the affine functions $x \to x \cdot \xi$ for $\xi \in \mathbb{R}^d$ as to make them ``harmonic'' in the domain $\mathcal{O}_\epsilon$ with zero Neumann data. They contain enough information about the problem to allow us to identify the limiting operator, this is the essence of the method developed first by Tartar to prove homogenization for a periodic and uniformly elliptic matrix $A(x)$.
    \end{rem}

	\begin{DEF}\label{def: A_0}
	     Observe that,  given $\xi \in \mathbb{R}^d$, the vector
	     \begin{equation*}
	          \mathbb{E}\left[ \mathbb{I}_{\mathcal{O}_\epsilon}(x)\nabla w_\xi^\epsilon(x)\right ]
	     \end{equation*} 
	     does not depend on $x$ or $\epsilon>0$, and is in fact a linear function of $\xi$. This defines a matrix $A$ by setting
	     \begin{equation}\label{eqn: definity of A}
	          A\xi := \mathbb{E}\left[ \mathbb{I}_{\mathcal{O}_\epsilon}(x)\nabla w_\xi^\epsilon(x)\right ] 	
	     \end{equation}
	     This $A$ will be called the effective diffusivity. It is not hard to see that $A$ is a symmetric matrix with nonnegative eigenvalues, we shall see that Assumption \ref{hyp: regularity} guarantees $A$ is indeed positive definite. 
	\end{DEF}
    \begin{lem}\label{lem: positivity of A_0}
    	 Suppose that $\mathcal{O}$ satisfies Assumptions \ref{hyp: Ergodic} and \ref{hyp: regularity}, then there is a universal $\lambda_0>0$ such that
         \begin{equation*}
              (A\xi,\xi)\geq \lambda_0|\xi|^2\;\;\forall\;\xi\in\mathbb{R}^d.	
         \end{equation*}
    \end{lem}
    \begin{proof}
         First of all, we note that Theorem \ref{thm: elliptic estimates} applies to the functions $w^\epsilon_\xi$, and conclude that in any cube  $\mathcal{Q}^\epsilon_R:= Q_R \cap \mathcal{O}_\epsilon$	we have $w^\epsilon_\xi(x) \to \xi\cdot x\pal$ in the $L^\infty$ norm, as in Definition \ref{def: notions of convergence}. In particular, if we fix $R=1$ then$\pal$ for all small enough $\epsilon$ we have 
         \begin{equation*}
	          \begin{array}{ll}
                   w^\epsilon_\xi(x) \leq - 0.5|\xi| & \text{in } \{ x \in \mathcal{Q}^\epsilon_1 \mid x\cdot \xi\leq -1\},\\
                   w^\epsilon_\xi(x) \geq \;\;\; 0.5|\xi| & \text{in } \{ x \in \mathcal{Q}^\epsilon_1 \mid x\cdot \xi\geq  1\}.		          	
	          \end{array}
         \end{equation*}
         Now, due to Assumption \ref{hyp: regularity}, there exists a network of disjoint cubes within $\mathcal{Q}_1$ going from $\mid x\cdot \xi\leq -1$ to $\mid x\cdot \xi\geq 1$ (see Proposition \ref{prop: network of tubes}). Integrating $\nabla w^\epsilon_\xi\cdot \xi$ along each of these tubes we can find a constant $c_0>0$ independent of $\epsilon$ and $\xi$ such that
         \begin{equation*}
         	  \int_{\mathcal{Q}^\epsilon_1}\nabla w^\epsilon_\xi\cdot \xi\;dx\geq c_0|\xi|^2
         \end{equation*}
         for all small enough $\epsilon$. Taking the limit as $\epsilon \to 0$, and using Theorem \ref{thm: ergodic theorem}, we conclude that
         \begin{equation*}
              \mu |Q_1|(A_0\xi,\xi)\geq c_0	|\xi|^2,
         \end{equation*}
         we conclude that all the eigenvalues of $A_0$ are strictly positive, and the lemma is proved.
    \end{proof}

	The next lemma follows Tartar \cite[Chapter 1]{BenLioPap1978}. It concerns the limit of (for any fixed $\phi \in C^\infty_c(\mathbb{R}^d)$)
	\begin{equation*}
	     \int_{\mathbb{R}^d}\mathbb{I}_{\mathcal{O}_\epsilon}\nabla u^{\epsilon}\nabla \phi\;dx	
	\end{equation*}
	as $\epsilon \to 0$. Note both $\mathbb{I}_{\mathcal{O}_\epsilon}$ and $\nabla u^\epsilon$ are at best converging weakly (separately) as $\epsilon\to 0$, so computing this limit is a priori non-trivial. However, by using the fact that $\nabla u^\epsilon$ is not just any vector field (it has a potential), we can get around the lack of strong convergence. This is where we use the ``corrected solutions'' $w^\epsilon_\xi$, in fact, the argument does not depart far from the periodic case from this point on.

	\begin{lem}\label{lem: compensated compactness} 
		 Let $A$ be as in Definition \ref{def: A_0}, and suppose that $\epsilon_k$ is a subsequence going to zero and $u^{\epsilon_k}$ converges in $L^2(D)$ to some $u \in H^1_0(D)$. Then, after passing to a subsequence (which we still call $\epsilon_k$) we have that
	     \begin{equation*}
	          \mathbb{I}_{\mathcal{O}_{\epsilon_k}} \nabla u^{\epsilon_k} \rightharpoonup 	A \nabla u\;\; \text{ in } L^2(D) 	
	     \end{equation*}
	\end{lem}

	\begin{proof}
	     Let $X$ denote the weak limit of $\mathbb{I}_{\mathcal{O}_{\epsilon_k}}\nabla u^{\epsilon_k}$. Given any $\phi \in C^\infty_c(\mathbb{R}^d)$, $\xi \in \mathbb{R}^d$ we will show that
	     \begin{equation*}
	          \int_{\mathbb{R}^d} X\cdot \xi \phi\;dx = \int_{\mathbb{R}^d} A \nabla u \cdot \xi \phi\;dx
	     \end{equation*}
	     which will imply that $X = A \nabla u$. Let $w^\epsilon_\xi$ be as above, we have
	     \begin{equation*}
	          -\int_{\mathbb{R}^d} \text{div}(\mathbb{I}_{\mathcal{O}_{\epsilon_k}} \nabla u^{\epsilon_k} ) w^{\epsilon_k}_\xi\phi \;dx = \int_{\mathcal{O}_{\epsilon_k}} \nabla u^{\epsilon_k}\cdot (\nabla \phi )w^\epsilon_\xi\;dx+\int_{\mathcal{O}_{\epsilon_k}} \nabla u^{\epsilon_k}\cdot \nabla (w^\epsilon_\xi) \phi \;dx
	     \end{equation*}
	     On the other hand, since $\text{div}(\mathbb{I}_{\mathcal{O}_{\epsilon_k}} \nabla w^{\epsilon_k}_\xi)=0$, we also have for any $k$
	     \begin{equation*}
	          0 = \int_{\mathcal{O}_{\epsilon_k}} u^{\epsilon_k} \nabla \phi \cdot \nabla w^{\epsilon_k}_\xi \;dx+\int_{\mathcal{O}_{\epsilon_k}} \phi \nabla u^{\epsilon_k}\cdot \nabla w^{\epsilon_k}_\xi \;dx
	     \end{equation*}
	     Therefore
	     \begin{equation*}
	          -\int_{\mathbb{R}^d} \text{div}(\mathbb{I}_{\mathcal{O}_{\epsilon_k}} \nabla u^{\epsilon_k} ) w^{\epsilon_k}_\xi\phi \;dx = \int_{\mathcal{O}_{\epsilon_k}} \nabla u^{\epsilon_k}\cdot (\nabla \phi )w^{\epsilon_k}_\xi\;dx-\int_{\mathcal{O}_{\epsilon_k}} u^{\epsilon_k} \nabla \phi \cdot \nabla w^{\epsilon_k}_\xi \;dx
	     \end{equation*}
	     Then, due to the convergence of the different terms, taking the limit in the last identity gives
	     \begin{equation*}
	          -\int_{\mathbb{R}^d} \text{div}(X) (x\cdot \xi)\phi\;dx = \int_{\mathbb{R}^d} X \cdot (\nabla \phi )(x\cdot \xi)\;dx-\int_{\mathbb{R}^d} u \nabla \phi \cdot A \xi \;dx. 
	     \end{equation*}
	     Integrating by parts on the left we arrive at
	     \begin{equation*}
	          -\int_{\mathbb{R}^d} \text{div}(X)(x\cdot \xi)\phi\;dx = -\int_{\mathbb{R}^d}\text{div}(X) (x\cdot \xi) \phi\;dx-\int_{\mathbb{R}^d} X\cdot \xi \phi\;dx-\int_{\mathbb{R}^d} u \nabla \phi \cdot A \xi \;dx.
	     \end{equation*}
	     It follows that
	     \begin{equation*}
	          \int_{\mathbb{R}^d} X\cdot \xi \phi\;dx = -\int_{\mathbb{R}^d} u \nabla \phi \cdot A\xi \;dx
	     \end{equation*}
	     and integrating by parts on the right and using the symmetry of $A$ we finish the proof.
	\end{proof}

	\begin{proof}[Proof of Theorem \ref{thm: stochastic homogenization linear}]
		 From the equation solved by $u^\epsilon$ we have that $\mathbb{I}_{\mathcal{O}_{\epsilon}} \nabla u^\epsilon$ is bounded in $L^2(\mathbb{R}^d)$ independently of $\epsilon$. Moreover, due to Theorem \ref{thm: elliptic estimates} we know that along a subsequence the functions $u^\epsilon$ converge to a function $u \in H^1_{\text{loc}}(D)$.

	     Then Lemma \ref{lem: compensated compactness} says (after passing to another subsequence) that $\mathbb{I}_{\mathcal{O}_{\epsilon_k}}\nabla u^{\epsilon_k} \rightharpoonup A \nabla u$. Then, since
	     \begin{equation*}
	          -\int \mathbb{I}_{\mathcal{O}_{\epsilon_k}} \nabla u^{\epsilon_k} \cdot \nabla \phi\;dx = \int f^\epsilon \phi\;dx	
	     \end{equation*} 
	     we conclude after passing to the limit that
	     \begin{equation*}
	          -\int A \nabla u \cdot \nabla \phi\;dx = \int f \phi\;dx.	
	     \end{equation*}
	     Thus $u$ is a weak solution of the homogenized Dirichlet problem. Since this problem has a unique solution we see that $u$ does not depend on the original subsequence $\epsilon_k$ so $u^{\epsilon}$ converges to $u$.
	\end{proof}

\appendix
\section{Existence of solutions to the auxiliary problem}\label{sec: construction of the obstacle problem}

    In this section we discuss the existence of solutions to the auxiliary problem \eqref{eqn: epsilon obstacle problem}. 

    \begin{thm}\label{thm: existence epsilon obstacle problem}
         Given $\db_0^\epsilon \subset \mathcal{O}_\epsilon$, there is a solution to \eqref{eqn: epsilon obstacle problem} in the sense of Definition \ref{def: weak solution p^epsilon}.
    \end{thm}
    
    For the sake of simplicity we will omit the $\epsilon$ superscript through this section. The one tool we will need has to do with solutions to the obstacle problem in stationary ergodic domains, which falls within the abstract theory of variational inequalities. See \cite[Chapter II]{KiSt1980} for a proof.

    Recall the set $\bfb_\epsilon \subset \mathbb{R}^d \times \Omega$ introduced in Definition \ref{def: function spaces}, as $\epsilon$ is fixed in this section, we omit it.

    \begin{prop}\label{prop: obstacle problem solutions}
Given any function $f : \bfb  \to \mathbb{R}$ there exists a unique function
		 \begin{equation*}
		 	  p[f]:\bfb \to \mathbb{R}
		 \end{equation*}
		 such that $p[f] \geq 0$, $p[f] \in H^1(\mathcal{O}) \pal$ and
		 \begin{equation*}
		      \int_{\mathcal{O}}\nabla (p[f]-\phi)\cdot \nabla \phi + (p[f]-\phi)f\;dx\geq 0 \;\;\forall \;\phi\in H^1(\mathcal{O})\pal	
		 \end{equation*}
    \end{prop}

    \begin{rem} We note that the property of $p$ being a solution in the sense of Definition \ref{def: weak solution p^epsilon} is equivalent to the identity
	     \begin{equation*}
	     	  p(\cdot,t) = p[f(\cdot,t)]\;\;\forall\;t\geq 0,
	     \end{equation*} 
	     where $f(\cdot,t) = \mathbb{I}_{\db_0^c}-\int_0^t \mathbb{I}_{\{p(\cdot,s)>0\} }\;ds$. Thus, the weak solution we are seeking can be thought of as a fixed point of the map above.
    	
    \end{rem}

    \begin{proof}[Proof of Theorem \ref{thm: existence epsilon obstacle problem}] 	
         We construct the solution through an iterative procedure. Fix $\db_0\subset\mathcal{O}$ and define $p_0$ by
         \begin{equation*}
              \left \{ \begin{array}{rlll}
              	   \Delta p_0 & = & -t & \text{ in } \db_0,\\
                   \partial_n p_0 & = & 0 & \text{ on } \partial \mathcal{O}(\omega),\\
                          p_0 & \equiv & 0 & \text{ in } \db_0^c.
              \end{array}\right.
         \end{equation*}
         For $n\geq 1$, define $p_n$ by $p_n(\cdot,t) := p[f_n(\cdot,t)]$ where $f_n$ is given by
         \begin{equation}\label{eqn: iteration scheme}
              f_n(\cdot,t) := \mathbb{I}_{\db_0^c}-\int_0^t \mathbb{I}_{\{p_{n-1}(\cdot,s)>0 \}}\;ds.	
         \end{equation}
         Let us show that $f_{n+1}\leq f_n$ for all $n$. Arguing by induction, we show first $f_2\leq f_1$. Since $p_0>0$ in $\db_0^\epsilon$,
         \begin{align*}
              f_1 & = \mathbb{I}_{\db_0^c}-t\;\mathbb{I}_{\db_0}	
         \end{align*}
         so that $f_2\leq f_1$ will be implied by $p_1>0$ a.e. in $\db_0$. At the same time, this is equivalent to showing that $\{ p_1>0 \}$ is dense in $\db_0$ (by the continuity of $p_1$). If $\{ p_1>0\}$ was not dense in $\db_0$, there would be a small ball contained in $\db_0$ where $p_1 \equiv 0$. Then, let $\phi_0$ solve the Dirichlet problem in that ball with right hand side equal to $-t$, zero Neumann data on $\partial \mathcal{O}$ and zero Dirichlet data on the boundary of the ball. Then, integration by parts yields that (we extend $\phi_0$ by zero)
         \begin{equation*}
			  \int_{\mathcal{O}}\nabla (p_1-\phi)\cdot \nabla \phi + (p_1-\phi)f_1\;dx< 0\;\text{ for } \phi=p_1+\phi_0,	
         \end{equation*}
         which is impossible due to the definition of $p_1$. Thus $p_1>0$ a.e. in $\db_0$ and $f_2\leq f_1$, as we wanted.

         For the inductive case, suppose that $f_{n+1} \leq f_{n}$ for some $n$. Then, the comparison principle says that
         \begin{equation*}
         	  p_{n+1}=p[f_{n+1}] \geq p[f_{n}]=p_{n}.
         \end{equation*}
         In particular
         \begin{equation*}
              \mathbb{I}_{\{ p_{n}>0\} }\leq \mathbb{I}_{ \{p_{n+1}>0\} },
         \end{equation*}
         and by integrating with respect to time this leads to $f_{n+2} \leq f_{n+1}$, as we wanted.

         We have shown not only that $f_n$ is monotone decreasing in $n$, but that $p_n$ is monotone increasing in $n$. Moreover,  $p_n$ has a bounded norm $H^1$ norm$\pal$, and $f_n$ is uniformly bounded pointwise for each $t$, so they must converge pointwise a.e. in $\bfb \times \mathbb{R}_+$ to functions $p$ and $f$.

         On the other hand, $\pal$ the functions $p_n(x,\omega,t)$ are uniformly continuous in compact subsets of $\mathcal{O}(\omega)\times \mathbb{R}_+$, so they also converge uniformly for each fixed $\omega$ to their pointwise limit $p(x,\omega,t)$. In this case, one may argue as in Section \ref{sec: convergence obstacle problem} to show that
         \begin{equation*}
              \mathbb{I}_{\{p>0\} } = \lim \limits_n \mathbb{I}_{\{p_n>0\} }\;\pal,	
         \end{equation*}
         the limit being in $L^2$ sense. This shows that the pointwise limit of $f_n$ is in fact 
         \begin{equation*}
              f = \mathbb{I}_{\db_0^c}-\int_0^t \mathbb{I}_{\{p(\cdot,s)>0 \}}\;ds,
         \end{equation*}
         and we conclude that  $p(\cdot,t)=p[f(\cdot,t)]$ for all $t\geq 0$, so $p$ is a solution and the theorem is proved.

    \end{proof}

\section{Pointwise bounds for the Green's function in $\mathcal{O}$}\label{sec: Green's function estimates}

    Here we obtain pointwise estimates for the Green's function $G^\epsilon$ of $\mathcal{O}_\epsilon(\omega)$, that is, the solution of
    \begin{equation*}
    	 \begin{array}{rll}
    	      \Delta G^\epsilon(\cdot,y,\omega) & = \delta_y(\cdot) & \text{ in } \mathcal{O}_\epsilon(\omega),\\
              \partial_n G^\epsilon(\cdot,y,\omega) & = 0 & \text{ on } \partial \mathcal{O}_\epsilon(\omega),
    	 \end{array}
    \end{equation*}
    which is a function $G^\epsilon(\cdot,\cdot,\omega):\mathcal{O}(\omega)\times \mathcal{O}(\omega) \to \mathbb{R}$. These estimates say that for $\mathbb{P}$-almost every $\omega$ the function $G^\epsilon(x,y,\omega)$ is pointwise comparable with $|x-y|^{2-d}$, the Green's function of $\mathbb{R}^d$.

    Although we do not need them elsewhere in this work, bounds of this type are of interest in percolation theory, as hinted at above (see also the discussion in \cite{Bar2004} as well as in \cite{Bis2011}). They are also crucial in the analysis of boundary behavior of harmonic functions in Lipschitz domains \cite{LiStWe1963}. Proofs of these bounds following De Giorgi's or Moser's approaches are currently not available (even in the discrete case). This motivates us to review the proof of the bounds for the Green's function in the context of a perforated domain, using Assumption \ref{hyp: regularity}.

    \medskip

    The main theorem of this section is the following.

    \begin{thm}\label{thm: green function bound} 
         There exists a universal constant $C>0$ such that
         \begin{equation*}
              C^{-1}|x-y|^{2-d}\leq G^\epsilon(x,y,\omega) \leq C|x-y|^{2-d},\;\;\;\forall\;x,y\in \mathcal{O}_\epsilon(\omega)	\pal
         \end{equation*}
         when $d\geq 3$, and replacing $|x-y|^{2-d}$ above with $\log{|x-y|}$ when $d=2$.	
    \end{thm}
    The proof is along the lines of the Littman-Weinberger-Stampacchia theory \cite{LiStWe1963} for fundamental solutions of elliptic operators in divergence form. Beyond the elliptic theory in Section \ref{sec: elliptic regularity} and a capacity comparison estimate (Lemma \ref{lem: capacity lower bound}) the proof follows the classical one. In particular, the notion of capacity \cite[Chapter 4]{EvGa1992} will play a central role in what follows.

    \begin{DEF} We define the capacity of a set $E \subset \mathbb{R}^d$ by 
         \begin{equation*}
              \text{Cap}(E) = \inf \left \{ \int_{\mathbb{R}^d}|\nabla v|^2dx\;,\;\text{ where }\; v \in H^1(\mathbb{R}^d)\;\text{ and }\; E \subset \{ v\geq 1\}^0 \right \}	
         \end{equation*}	
         Moreover, for $\epsilon>0$ and $\omega\in\Omega$ we introduce the capacity of $E \subset \mathbb{R}^d$ with respect to  $\mathcal{O}_\epsilon(\omega)$,
         \begin{equation*}
              \text{Cap}_{\epsilon,\omega}(E) = \inf \left \{ \int_{\mathcal{O}_\epsilon(\omega)}|\nabla v|^2dx\;,\;\text{ where }\; v \in H^1(\mathcal{O}_\epsilon(\omega))\;\text{ and }\; E\cap \mathcal{O}_\epsilon(\omega) \subset \{ v\geq 1\}^0 \right \}	
         \end{equation*}
    \end{DEF}

    \begin{DEF} 
         Given $E$ and any of the two notions of capacity, we can find a  unique $u \in H^1$ that achieves the infimum above. In either case this $u$ will be called the \emph{capacitary potential} of $E$.	
    \end{DEF}

    \begin{rem}\label{rem: capacity properties}
         It is immediate from the definition of capacity that if $A\subset B$ then $\text{Cap}(A) \leq\text{Cap}(B)$, likewise $\text{Cap}_{\epsilon,\omega}(A)\leq \text{Cap}_{\epsilon,\omega}(B)$. Moreover, we always have $\text{Cap}(A) \geq \text{Cap}_{\epsilon,\omega}(A)$.	
    \end{rem}

    The connection between capacity and the fundamental solution is illustrated by the next propositions.
    \begin{prop}\label{prop: capacity 1}[See \cite[Section 7]{LiStWe1963}]
         Given $\epsilon>0,\omega \in \Omega$, $a>0$ and $y \in \mathbb{R}^d \setminus \mathcal{O}_\epsilon(\omega)$ let
         \begin{equation*}
              J_{\epsilon}(a,y,\omega):= \left \{ x \mid G^\epsilon(x,y,\omega) \geq a\right \}.	
         \end{equation*}	
         Then,
         \begin{equation*}
              \text{Cap}_{\epsilon,\omega}(J_{\epsilon}(a,y,\omega)) = \tfrac{1}{a}.	
         \end{equation*}
    \end{prop}
 
    \begin{proof}
    	  It is well known \cite[Chapter 4]{EvGa1992} that the infimum is achieved by a function given by
          \begin{equation*}
          	  v(\cdot) = \int_{\mathcal{O}_\epsilon}G^\epsilon(x,\cdot,\omega)d\nu(x),
          \end{equation*}
          where $\nu$ is the so called capacitary distribution of the set $J_{\epsilon}(a,y,\omega)$. In particular, one can check that $v$ is continuous at $y$ and that $v(y)=1$. From this it follows that
          \begin{equation*}
          	  1 = \int_{\mathcal{O}_\epsilon}G^\epsilon(x,y,\omega)d\nu(x).
          \end{equation*}
          On the other hand $d\nu$ is by definition supported on $\partialom J_{\epsilon}(a,y,\omega)$, and there we know that $G^\epsilon(x,y,\omega)$ is identically equal to $a$ , from the definition of $J_{\epsilon}(a,y,\omega)$ (note we are using that $G^\epsilon$ is continuous away from $y$, which follows in particular from Section \ref{sec: elliptic regularity}). We conclude that
         \begin{equation*}
              1 = a|\nu| = a \text{Cap}_{\epsilon,\omega}(J_{\epsilon}(a,y,\omega))	
         \end{equation*}
         dividing both sides by $a$ the proposition follows.
    \end{proof}

    \begin{prop}\label{prop: capacity 2}
         For any $y \in \mathbb{R}^d\setminus \mathcal{O}_\epsilon(\omega)$ and any $r>0$ we have
         \begin{equation*}
         \inf\limits_{x\in\partial B_r(y)} G^\epsilon(x,y,\omega) \leq \left [ \text{Cap}_{\epsilon,\omega}(B_r(y)) \right ]^{-1}\leq \sup\limits_{x\in\partial B_r(y)} G^\epsilon(x,y,\omega)	
         \end{equation*}
    \end{prop}

    \begin{proof}
         We fix $\epsilon,\omega$ and $y$ for the rest of the proof. Given $r>0$, define
         \begin{equation}
              a(r) = \inf\limits_{x \in \partial B_r(y)} G^\epsilon(x,y,\omega),\;\;	b(r) = \sup\limits_{x \in \partial B_r(y)} G^\epsilon(x,y,\omega)
         \end{equation}
         Then, it is clear that
         \begin{equation*}
              J_{\epsilon}(b(r),t,\omega) \subset B_r(y) \subset J_{\epsilon}(a(r),y,\omega) 
         \end{equation*}
         and by the previous remark,
         \begin{equation*}
              \text{Cap}_{\epsilon,\omega}(J_{\epsilon}(b(r),y,\omega)) \leq \text{Cap}_{\epsilon,\omega}(B_r(y)) \leq \text{Cap}_{\epsilon,\omega}(J_{\epsilon}(a(r),y,\omega)) 
         \end{equation*}	
         we then apply Proposition \ref{prop: capacity 1} and conclude the proof.
    \end{proof}

    The last part of the proof is modeled on a well known argument in percolation theory regarding the connectivity of the infinite percolation cluster, see for instance \cite{Ke1982} or \cite{Gri1999}. In our case, Assumption \ref{hyp: regularity} guarantees that $\mathcal{O}$ is always ``connected enough'' so the proof is much simpler. We start by quantifying this connectivity.

    \begin{prop}\label{prop: network of tubes}
         Fix $L>1,\epsilon>0$ and a point $x_0$. Then$\pal$ for any $x \in \partialom \mathcal{B}^\epsilon_1(x_0) $ there is a Lipschitz path $\gamma(t) \subset \mathcal{O}_\epsilon$ such that for a universal constant $C$ and $e:= (x-x_0)/|x-x_0|$ we have
         \begin{align*}
	          \inf \limits_{t \in [0,1]} d(\gamma(t),\partial \mathcal{O}_\epsilon)\geq \epsilon d_0/3,\\
              \gamma(0)=x,\;\;\gamma(1) \in \partialom \mathcal{B}^\epsilon_L(x_0),\\
              (\dot{\gamma}, e) \geq C^{-1}|\dot{\gamma}|,\\
              \left | \gamma(t) - (\gamma(t),e)e \right | \leq C\epsilon,\\
              \int_0^1|\dot{\gamma}(t)|dt \leq CL.
         \end{align*}
         Here $d_0$ is the constant from Assumption \ref{hyp: regularity}. Moreover, for each $\epsilon$ one can find almost surely a finite family of paths $\gamma_1,...,\gamma_N$ satisfying all the properties above and such that  $d(\gamma_i,\gamma_j)\geq C\epsilon$ for $i\neq j$. Moreover, if $T_1,...,T_n$ denote their corresponding $C\epsilon$-tubular neighborhoods then
         \begin{equation*}
              T_i \subset O_\epsilon(\omega) \text{ and } \left |\union \limits_{i=1}^N T_i \right|	\geq c_0
         \end{equation*}
    \end{prop}

    \begin{proof}
         Let $\omega$ be taken from the set of full measure for which $\mathcal{O}(\omega)$ satisfies the structure bounds, as explained in Assumption \ref{hyp: regularity}. Recall that in this case $P_k(\omega)$ denotes the separated components of $\mathcal{O}(\omega)$. We start by taking the segment starting from $x \in \partialom \mathcal{B}^\epsilon_1$ in the direction $e$ until we are $\epsilon d_0/3$ away from the closest obstruction $P_k(\omega)$. Consider the projection of this $\Pi_e P_k(\omega)$ on the hyperplane perpendicular to $e$. Then $\Pi_e P_k(\omega) \subset\subset B'_\epsilon(0)$ and we can pick some $\tau \in B_\epsilon(\omega) \setminus \Pi_e P_k(\omega)$.

         Next, we move along the projection of the segment moving along $\tau$ onto the boundary of the obstruction until we reach at a point where $e$ is supporting, then we move straight up until we are within $d_0/3$ of the next obstruction.

         In each case, because the perforations are uniformly Lipschitz, after a uniformly bounded time we will have left the perforation, which shows we went up by a term of order $\epsilon$. Moreover, by picking $C$ large enough, it is guaranteed that we will move again in the $e$ direction before hitting the boundary of the cylinder, which guarantees the last inequality. Since we move by a uniform amount upward in finite time, it follows that we reach the top of cylinder in a controlled time.  Lastly, the Lipschitz curve can be re-parametrized, so that $t$ ranges form $0$ to $1$.
    \end{proof}

    \begin{lem}\label{lem: capacity lower bound}
         There is a universal $c_0>0$ such that$\pal$ given $R>0$ and $x \in \mathcal{O}_\epsilon(\omega)$ we have
         \begin{equation*}
              \text{Cap}_{\epsilon,\omega}(\mathcal{B}^\epsilon_R(x_0)) \geq c_0\text{Cap}(B_R(x_0))\;\;\forall\;\epsilon>0.
         \end{equation*}
    \end{lem}

    \begin{proof}
         Fix $\epsilon>0$ and $\omega \in \Omega$ such that the bounds in Assumption \ref{hyp: regularity} hold. First let us suppose that $R=1$. Let $u^\epsilon$ be the capacitary potential of $\mathcal{B}_1^\epsilon(x_0)$, then 
         \begin{equation*}
              \text{Cap}_{\epsilon,\omega}(\mathcal{B}^\epsilon_1(x_0)) = \int_{\mathcal{O}_\epsilon}|\nabla u^\epsilon|^2\;dx= \int_{\mathcal{O}_\epsilon \setminus \mathcal{B}^\epsilon_1(x_0)}|\nabla u^\epsilon|^2\;dx	
         \end{equation*}
         Then, the homogenization result from Theorem \ref{thm: stochastic homogenization linear} and the uniform H\"older estimate from Theorem \ref{thm: elliptic estimates} apply to $u^\epsilon$. This, together with a standard compactness argument, shows there is a universal constant $L_0>0$ such that
         \begin{equation*}
              u^\epsilon\leq \tfrac{1}{2} \text{ on } \partialom \mathcal{B}^\epsilon_{L_0}(x_0),
         \end{equation*}
         in particular, this constant does not depend on $\epsilon,x_0$ or $\omega$.

         Given this $L_0$, Proposition \ref{prop: network of tubes} says that for each $\epsilon>0$ there is a finite family of Lipschitz continuous paths $\gamma_1,...,\gamma_N$ ($N$ depending on $\epsilon$) such that if $T_1,...,T_N$ denote their $\epsilon$-tubular neighborhoods then they are pairwise disjoint, and for some $c_0>0$ independent of $\epsilon$ we have
         \begin{align*}
              \left | \gamma_i \right| \sim L_0 \;\;\;\forall\; i,\\
              T_i \subset \mathcal{O}_\epsilon(\omega)\;\;\forall\; i, \\ 
              \text{ and } \left |\union \limits_{i=1}^N T_i \right | \geq c_0.	
         \end{align*}
         Moreover, each $\gamma_i$ connects a point in $\mathcal{B}^\epsilon_1(x_0)$ with a point in $\{u\leq \tfrac{1}{2}\}$. Thus $u^\epsilon(\gamma_i(t))$ is equal to $1$ for $t=0$ and less than $1/2$ for $t=1$, which says that
         \begin{equation*}
              \int_0^1 \nabla u^\epsilon(\gamma_i(t))\cdot\dot\gamma_i(t)\;dt\geq 1/2.	
         \end{equation*}
         By applying Cauchy-Schwartz, this leads to
         \begin{equation*}
         	  CL_0\int_0^1 |\nabla u^\epsilon(\gamma_i(t))|^2 dt\geq L_0/2
         \end{equation*}
         We can repeat this argument for the volume integral over $T_i$, and conclude for each $T_i$ that
         \begin{equation*}
              \int_{T_i} |\nabla u|^2\;dx \geq C^{-1}|T_i|
         \end{equation*}
         for some larger but still universal constant $C>0$. Since the $T_i$ are all disjoint we have
         \begin{equation*}
              \int_{\mathcal{O}_\epsilon \setminus \mathcal{B}^\epsilon_1(x_0) } |\nabla u|^2\;dx \geq C^{-1}|\union \limits_i T_i|\geq C^{-1}|\mathcal{B}^\epsilon_{L_0}(x_0)|\geq C^{-1}L_0^d,
         \end{equation*}
        by making the universal constant $C$ a bit larger. Letting $c_0 = C^{-1}L_0^d\text{Cap}(B_1(0))^{-1}$, we conclude that
         \begin{equation*}
              \text{Cap}_{\epsilon,\omega}(\mathcal{B}^\epsilon_1(x_0))\geq c_0 \text{Cap}(B_1(x_0))
         \end{equation*}
         Since the constant $c_0$ is independent of $\epsilon>0$, we can reduce the case $R \neq 1$ to this one by rescaling  $u^\epsilon, \mathcal{B}^\epsilon_1(x_0)$ and $\mathcal{O}_{\epsilon}(\omega)$, and the lemma is proved.

    \end{proof}

    All that is left is to gather the previous lemmas and prove the main theorem of this section.
 
    \begin{proof}[Proof of Theorem \ref{thm: green function bound}]
         Using Remark \ref{rem: capacity properties} and Lemma \ref{lem: capacity lower bound} it follows that
         \begin{equation*}
              \text{Cap}(B_r(y))^{-1} \leq \text{Cap}_{\epsilon,\omega}(\mathcal{B}^\epsilon_r(y))^{-1} \leq c_0^{-1}\text{Cap}(B_r(y))^{-1},	
         \end{equation*}
         which becomes
         \begin{equation*}
     	     \frac{c_n}{r^{d-2}} \leq \text{Cap}_{\epsilon,\omega}(\mathcal{B}^\epsilon_r(y))^{-1} \leq c_0^{-1} \frac{c_n}{r^{d-2}}.
         \end{equation*}
         Then, in light of Proposition \ref{prop: capacity 2} we conclude that
         \begin{equation*}
              \inf\limits_{x\in\partialom \mathcal{B}^\epsilon_r(y)} G^\epsilon(x,y,\omega) \leq \left [ \text{Cap}_{\epsilon,\omega}(\mathcal{B}^\epsilon_r(y)) \right ]^{-1} \leq \sup\limits_{x\in\partialom \mathcal{B}^\epsilon_r(y)} G^\epsilon(x,y,\omega).
         \end{equation*}
         On the other hand, applying the Harnack inequality (Theorem \ref{thm: elliptic estimates}) to $G^\epsilon(.,y)$ we obtain
         \begin{equation*}
          	 \sup\limits_{x\in\partialom \mathcal{B}^\epsilon_r(y)} G^\epsilon(x,y,\omega) \leq C \inf\limits_{x\in\partialom \mathcal{B}^\epsilon_r(y)} G^\epsilon(x,y,\omega),
         \end{equation*}
         where $C$ is universal. It follows that
         \begin{equation*}
          	 C^{-1}\frac{c_d}{r^{d-2}}  \leq  \inf\limits_{x\in\partialom \mathcal{B}^\epsilon_r(y)} G^\epsilon(x,y,\omega)\leq \sup\limits_{x\in\partialom \mathcal{B}^\epsilon_r(y)} G^\epsilon(x,y,\omega) \leq Cc_0^{-1} \frac{c_d}{r^{d-2}}.
         \end{equation*}
         Equivalently, this shows there is some universal $C$ such that for all $\epsilon>0$ we have
         \begin{equation*}
	          \frac{C^{-1}}{|x-y|^{d-2}}\leq G^\epsilon(x,y,\omega) \leq \frac{C}{|x-y|^{d-2}}\;\; \pal
         \end{equation*}
    \end{proof}

\bibliography{homogenizationrefs}
\bibliographystyle{plain}
\end{document}